\documentclass[11pt]{article}

\makeatletter
\usepackage{fullpage}

\usepackage{amssymb, amsmath, amsthm, latexsym}
\usepackage{url}
\usepackage{algorithm}
\usepackage{algorithmic}
\usepackage{tabularx}
\usepackage{booktabs}
\usepackage{paralist}
\usepackage{mathtools}
\usepackage{xcolor}
\usepackage{hyperref}%
\usepackage[capitalise]{cleveref}%

\usepackage{bbm} 
\usepackage{bm}
\usepackage{wrapfig}
\usepackage{makecell}
\usepackage{multirow}

\usepackage{nicefrac}       

\usepackage[flushleft]{threeparttable} 

\usepackage{caption}
\usepackage{multirow}
\usepackage{colortbl}
\definecolor{bgcolor}{rgb}{0.8,1,1}
\definecolor{bgcolor2}{rgb}{0.8,1,0.8}
\definecolor{niceblue}{rgb}{0.0,0.19,0.56}

\usepackage{hyperref}
\hypersetup{colorlinks,linkcolor={blue},citecolor={niceblue},urlcolor={blue}}

\usepackage{natbib}
\usepackage{sidecap}

\newtheorem{theorem}{Theorem}
\newtheorem{proposition}{Proposition}
\newtheorem{corollary}{Corollary}

\newtheorem{assumption}{Assumption}
\newtheorem{remark}{Remark}%
\newtheorem{lemma}{Lemma}%

\crefname{assumption}{Assumption}{Assumptions}
\Crefname{assumption}{Assumption}{Assumptions}

\newcommand*{\R}{{\mathbb R}}
\newcommand*{\E}{{\mathbb E}}

\newcommand{\ip}[2]{\left\langle #1,#2\right\rangle}

\renewcommand{\|}{\parallel}

\newcommand\ev[1]{ \left\langle #1 \right\rangle}

\newcommand{\norm}[1]{\left\lVert#1\right\rVert}
\newcommand{\sqn}[1]{{\left\lVert#1\right\rVert}^2}

\definecolor{tiffanyblue}{rgb}{0.04, 0.73, 0.71}
\usepackage[most]{tcolorbox}
\newtcolorbox{remarkbox}{
    colback=tiffanyblue!20,    
    colframe=black,     
    boxrule=0.5mm,      
    sharp corners,      
    left=2mm,           
    right=2mm,          
    top=2mm,            
    bottom=2mm          
}

\newtcolorbox{remarkbox2}{
    colback=tiffanyblue!20,    
    colframe=black,     
    boxrule=0.5mm,      
    sharp corners,      
    left=2mm,           
    right=2mm,          
    top=0mm,            
    bottom=0mm          
}

\usepackage{pifont}
\definecolor{PineGreen}{RGB}{0,110,51}
\definecolor{BrickRed}{RGB}{143,20,2}
\newcommand{\cmark}{{\color{PineGreen}\ding{51}}}%
\newcommand{\xmark}{{\color{BrickRed}\ding{55}}}%

\usepackage{tikz-cd} 

\usepackage{graphicx}

\newcolumntype{Y}{>{\centering\arraybackslash}X}

\usepackage{xspace}

\definecolor{mycolor}{RGB}{0,150,70}
\newcommand{\algname}[1]{{\color{PineGreen}\sf  #1}\xspace}

\newcommand{\cO}{{\cal O}}

\def\clip{\texttt{clip}}

\usepackage{amsthm}
\usepackage{hyperref}
\newcommand{\mytag}[1]{%
  \refstepcounter{equation}%
  \edef\@currentlabel{\theequation}%
  {({\@currentlabel})}%
  \@bsphack
  \begingroup
    \@onelevel@sanitize\@currentlabelname
    \edef\@currentlabelname{%
      \expandafter\strip@period\@currentlabelname\relax.\relax\@@@%
    }%
    \protected@write\@auxout{}{%
      \string\newlabel{#1}{%
        {\@currentlabel}%
        {\thepage}%
        {\@currentlabelname}%
        {\@currentHref}{}%
      }%
    }%
  \endgroup
  \@esphack
}
\makeatother

\title{\bf High-Probability Convergence of Clipped SGD under Heavy-Tailed Noise and $(L_0,L_1)$-Smoothness}
\author{
\begin{tabular}[t]{c}
     {\bf Taha El Bakkali El Kadi}\textsuperscript{*}\\
     {\small MBZUAI\textsuperscript{$\dagger$}, UM6P} \\
\end{tabular}
\quad
\begin{tabular}[t]{c}
     {\bf Savelii Chezhegov}\textsuperscript{*}\\
     {\small BRAIn Lab, Sber AI Lab}
\end{tabular}
\quad
\begin{tabular}[t]{c}
     {\bf Aleksandr Beznosikov}\\
     {\small BRAIn Lab, Innopolis University}
\end{tabular}
\\[2em]
\begin{tabular}[t]{c}
     {\bf Samuel Horv\'ath}\\
     {\small MBZUAI}
\end{tabular}
\quad
\begin{tabular}[t]{c}
     {\bf Eduard Gorbunov}\textsuperscript{$\ddagger$}\\
     {\small MBZUAI}
\end{tabular}
}

\begin{document}

\maketitle
\begingroup
\renewcommand{\thefootnote}{\fnsymbol{footnote}}
\footnotetext[1]{Equal contribution.}
\footnotetext[2]{Part of this work was done when Taha El Bakkali El Kadi was a visiting student in the group of Prof.~Eduard Gorbunov at MBZUAI.}
\footnotetext[3]{Corresponding author: \texttt{eduard.gorbunov@mbzuai.ac.ae}.}
\endgroup
\begin{abstract}
  Gradient clipping is widely used in language-model training to control heavy-tailed gradient noise and can improve convergence guarantees over stochastic gradient descent (\algname{SGD}) under $(L_0,L_1)$-smoothness. Under these joint conditions, a central challenge is to obtain high-probability guarantees without exponential dependence on $L_1R_0$, where $R_0$ bounds the initial distance to a minimizer. We resolve this challenge for convex objectives, establishing, to the best of our knowledge, the first such guarantees for standard \algname{Clip-SGD}. We assume unbiased stochastic gradients with bounded central $\alpha$-th moment, $\alpha\in(1,2]$. Our bounds have only polylogarithmic dependence on the inverse failure probability and recover known deterministic generalized-smoothness rates when the noise vanishes, as well as the standard high-probability rate under heavy-tailed noise in the classical $L$-smooth setting. The convex rate is attained by a computable output that averages iterates whose sampled stochastic gradients are not clipped, requiring neither function values nor extra oracle calls. We establish a lower bound for \algname{Clip-SGD} with any fixed stepsize and clipping level, showing that, under the stated iteration conditions, our convex stochastic rate is optimal up to logarithmic factors in the iteration budget. Our upper-bound analysis uses a directional clipping-bias bound to absorb part of the bias into the progress generated by the clipped population gradient, avoiding an exponentially large local smoothness constant. We also obtain nonconvex high-probability guarantees that recover the noiseless rate, match the classical-smooth stochastic iteration rate when $L_1=0$, and have no explicit $L_1$ dependence in the asymptotically dominant stochastic term.
\end{abstract}

\clearpage

\tableofcontents

\section{Introduction}

Stochastic Gradient Descent (\algname{SGD}) \citep{robbins1951stochastic} and its variants are central to large-scale machine learning, where computing full gradients can be prohibitively expensive \citep{shalev2014understanding, goodfellow2016deep}. Gradient clipping stabilizes these methods by limiting the magnitude of stochastic gradients. It is used to control exploding gradients in recurrent neural networks \citep{pascanu2013difficulty} and is a standard component in training models such as BERT \citep{devlin2019bert}, GPT-3 \citep{brown2020language}, Switch Transformers \citep{fedus2022switch}, and LLaMA \citep{touvron2023llama}. Two properties help explain its effectiveness: robustness to heavy-tailed gradient noise and adaptation to gradient-dependent curvature.

Heavy-tailed stochastic gradient noise has been observed in tasks such as BERT pre-training \citep{zhang2020adaptive}. Such noise need not have finite variance, placing it beyond the scope of classical bounded-variance analyses. Clipping limits the influence of extreme gradient samples and enables \emph{high-probability} convergence guarantees under weaker moment assumptions. In particular, recent analyses obtain bounds with polylogarithmic dependence on the inverse failure probability $1/\delta$ when the noise has a bounded central $\alpha$-th moment, with $\alpha\in(1,2]$ \citep{cutkosky2021high, sadiev2023high, nguyen2023improved, chezhegov2024gradient}. These guarantees balance clipping bias against stochastic fluctuations, often using a clipping threshold that grows with the iteration budget.

Clipping also helps when curvature depends on the gradient norm. \citet{zhang2019gradient} observed this dependence along training trajectories, motivating the $(L_0,L_1)$-smoothness assumption. This condition extends classical $L$-smoothness by allowing local curvature to increase with the gradient norm. Under this assumption, clipping can improve convergence guarantees over unclipped gradient methods \citep{zhang2019gradient, zhang2020improved, koloskova2023revisiting, gorbunov2024methods, vankov2024optimizing}. Several such analyses calibrate the clipping threshold to the smoothness parameters $L_0$ and $L_1$.

For convex $(L_0,L_1)$-smooth objectives with heavy-tailed gradient noise, we seek a high-probability guarantee for standard \algname{Clip-SGD}, expressed in terms of an upper bound $R_0$ on the initial distance to a minimizer. A direct reduction to classical smoothness can introduce exponential dependence on $L_1R_0$, since the effective smoothness constant over a localization region can grow exponentially with its radius. Our main question is therefore:
\begin{gather*}
    \textit{Can standard \algname{Clip-SGD} achieve high-probability convergence}\\
    \textit{for convex $(L_0,L_1)$-smooth objectives under heavy-tailed noise}\\
    \textit{without exponential dependence on $L_1R_0$?}
\end{gather*}

We answer this question affirmatively. We first introduce the problem setup and assumptions, then present our main results.

\subsection{Problem setup and assumptions}

\paragraph{Notation.} We denote the Euclidean norm on $\mathbb{R}^d$ by $\norm{\cdot}$.
For $\lambda>0$, $\clip(x,\lambda)$ denotes the Euclidean projection
of $x$ onto the closed ball of radius $\lambda$ centered at the origin.
When $f$ attains its minimum, we use $R_0$ to denote an upper bound on
the distance from $x_0$ to some minimizer $x^\star$.
We let $\mathcal F_k:=\sigma(\xi_0,\ldots,\xi_{k-1})$ denote the history
generated by \algname{Clip-SGD} before the $k$-th stochastic-gradient query.
We also let $\nu \in (0,1)$ denote the unique positive solution of
$\nu e^\nu = 1$.

\paragraph{Problem.} We focus on the classical stochastic optimization problem, which can be stated as
\begin{align}
    \min_{x \in \mathbb{R}^d} \left\{f(x) := \mathbb{E}_{\xi \sim \mathcal{D}}[f(x, \xi)]\right\}. \label{eq:main_problem}
\end{align}
This formulation is foundational in machine learning
\citep{shalev2014understanding}, where $f$ is the population loss,
$x$ the model parameters, and $\xi\sim\mathcal D$ a data sample.

We write $f^\star:=\inf_{x\in\mathbb{R}^d}f(x)$ and assume
$f^\star>-\infty$. For our convex guarantees, we additionally
assume that this infimum is attained.
We consider \algname{Clip-SGD} (\cref{alg:clip-sgd})
applied to this problem.

\begin{algorithm}[H]
   \caption{\algname{Clip-SGD}} \label{alg:clip-sgd}
\begin{algorithmic}[1]
    \STATE {\bfseries Input:} Starting point $x_0$, level of clipping $\lambda$, learning rate $\gamma$
    \FOR{$k = 0, \ldots, K-1$}
    \STATE Sample $\nabla f(x_k, \xi_k)$
    \STATE $x_{k+1} = x_k - \gamma\clip(\nabla f(x_k, \xi_k), \lambda)$
    \ENDFOR 
\end{algorithmic}
\end{algorithm}

\paragraph{Assumptions.}
Our convex guarantees use the following three assumptions.

\begin{assumption}[Convexity]
    \label{asm: convexity}
     $f$ is convex: for all $x,y\in\mathbb{R}^d$,
    $
        f(y)\ge f(x)+\ev{\nabla f(x),y-x}.
    $
\end{assumption}

\begin{assumption}[$(L_0,L_1)$-smoothness]
    \label{asm: smoothness}
    There exist $L_0,L_1\ge0$ such that, for all
    $x,y\in\mathbb{R}^d$,
    $$
        \norm{\nabla f(x)-\nabla f(y)}
        \le
        \left(L_0+L_1\sup_{u\in[x,y]}\norm{\nabla f(u)}\right)
        \norm{x-y}.
    $$
    \end{assumption}
  This assumption is strictly more general than classical $L$-smoothness, i.e., the assumption that $\nabla f$ is globally $L$-Lipschitz, which is recovered by setting $(L_0,L_1)=(L,0)$. It also covers functions that are not globally $L$-smooth, such as $f(x)=\exp(\norm{x})+\exp(-\norm{x})$ and $f(x)=\norm{x}^n$ for $n>2$ \citep{zhang2019gradient,zhang2020improved,chen2023generalized}.

\begin{assumption}[Stochastic oracle]
    \label{asm: stochastic}
The stochastic oracle is unbiased and has bounded central $\alpha$-th moment:
$\exists\,\alpha\in(1,2],\ \sigma\ge0$ such that, for all $x\in\mathbb{R}^d$,
$$\E[\nabla f(x,\xi)]=\nabla f(x) \qquad \E[\norm{\nabla f(x,\xi)-\nabla f(x)}^\alpha]\le\sigma^\alpha$$
\end{assumption}
For $\alpha=2$, this recovers the bounded-variance setting \citep{ghadimi2013stochastic}, while
$\alpha\in(1,2)$ allows noise with infinite variance.

\subsection{Our contributions}

We establish, to the best of our knowledge, the first high-probability
convergence guarantees for standard \algname{Clip-SGD} on convex
objectives under the joint assumptions of heavy-tailed noise and
$(L_0,L_1)$-smoothness. Under Assumptions~\ref{asm: convexity},\ref{asm: smoothness} and \ref{asm: stochastic}, with probability at least $1-\delta$, we obtain
$$
f(\widehat{x}_K)-f^\star \le \mathcal{O}\left( \max\left\{\frac{L_0R_0^2}{K},
\;\sigma R_0\left(\frac{\max\{L_1R_0,\log(12/\delta)\}}{K}\right)^{\frac{\alpha-1}{\alpha}}
\right\}\right),
$$
provided that, for $L_0>0$ and $R_0>0$,
$$
K=\Omega\left(
\log^2\left(\frac{12}{\delta}\right)
+
L_1R_0
\log\left[
\frac{12e}{\delta}
\left(1+\frac{K\Delta_0}{L_0R_0^2}\right)
\right]
\right),
$$
where $\Delta_0:=f(x_0)-f^\star$. An explicit sufficient
iteration requirement is
$$
K=\Omega\left(
\log^2\left(\frac{12}{\delta}\right)
+
L_1R_0
\log\left[
\frac{12e}{\delta}
\left(1+\frac{L_1\Delta_0}{L_0R_0}\right)
\right]
\right).
$$
Here, $\widehat{x}_K$ averages the iterates whose sampled stochastic
gradients are not clipped, requiring neither function values nor
additional oracle calls. At fixed confidence, this explicit iteration
requirement can be substantially smaller than a quadratic requirement
in $L_1R_0$.

Our convex guarantees recover deterministic rates under
generalized smoothness when the noise vanishes
\citep{gorbunov2024methods,vankov2024optimizing}, and the standard
high-probability rate under heavy-tailed noise in the classical
$L_0$-smooth setting ($L_1=0$) \citep{sadiev2023high}.
The key ingredient is the directional clipping-bias bound in
\cref{lem:directional-clipping-bias}, which separates the bias into
a heavy-tail remainder and a component that can be absorbed into
the progress generated by the clipped population gradient.
Together with structural inequalities for this gradient, this
avoids introducing an exponentially large local smoothness constant.
A comparison with prior convex guarantees appears in
\cref{tab:comparison}.

We complement the convex upper bound with a lower bound for
\algname{Clip-SGD} with arbitrary fixed stepsize and clipping level;
see \cref{thm:fixed-clip-lower-bound}.
For every $\delta\in(0,1/4)$, under the iteration conditions of that
theorem, and for every such parameter choice and every possibly
randomized output rule based on the observed transcript, with output
denoted by $x_K^{\mathrm{out}}$, we construct
a one-dimensional convex $(1,L_1)$-smooth instance with a minimizer
within distance $R_0$ of the initial point and an unbiased
stochastic-gradient oracle satisfying the same central $\alpha$-th
moment bound, such that
$$
\mathbb{P}\left(f(x_K^{\mathrm{out}})-f^\star\ge \frac{\sigma R_0}{128}\left[
\frac{\max\{L_1R_0/\log(eK),\,\log(12/\delta)\}}{K}\right]^{\frac{\alpha-1}{\alpha}} \right)>\delta.
$$
Thus, our convex stochastic rate, including its confidence dependence,
is optimal up to logarithmic factors in the iteration budget for
fixed-parameter \algname{Clip-SGD}.

We also establish a unified high-probability guarantee for the
average gradient norm on nonconvex objectives; see
\cref{appendix_nonconvex_extension}.
It recovers the deterministic generalized-smoothness rate when
$\sigma=0$ and the classical-smooth heavy-tailed stochastic
iteration rate when $L_1=0$.
The noiseless recovery improves on the stated bounds of
\citet{nguyen2023improved} and \citet{an2025tight} under their
prescribed tuning.
For fixed problem parameters with $L_0\Delta_0>0$ and $\sigma>0$,
and fixed confidence, our asymptotically dominant stochastic term
has the same polynomial iteration rate as the corresponding
leading term of \citet{an2025tight}, while containing no explicit
$L_1$ dependence and no logarithmic factors in $K$ or $1/\delta$.
In our bound, explicit $L_1$ dependence and logarithmic factors appear
only in terms that decay faster with $K$.

\begin{table*}[t!]
\centering
\caption{
Convergence guarantees for convex objectives.
Here $R_0\geq \lVert x_0-x^\star\rVert$ is an upper bound on the initial
distance to a minimizer, and
$M_\delta=\max\{L_1R_0,\log(12/\delta)\}$.
Stochastic bounds hold with high probability.
The symbols $\widetilde{\mathcal O}$ and $\widetilde{\Theta}$
hide logarithmic factors in $K$ and $1/\delta$.
Method details and sufficient iteration budgets are provided in the
\hyperref[par:closest-prior-work]{comparison with closest prior work}.
}
\label{tab:comparison}

\small
\setlength{\tabcolsep}{4pt}
\renewcommand{\arraystretch}{1.2}

\resizebox{0.985\textwidth}{!}{%
\begin{tabular}{lccc}
\toprule
\textbf{Reference}
& \textbf{Stochasticity}
& \textbf{Clipping scale}
& \textbf{Rate}
\\
\midrule

\rowcolor{blue!8}
\multicolumn{4}{c}{\textbf{$L$-smooth case}}
\\
\midrule

\cite{sadiev2023high}
& HT
& $\widetilde{\Theta}\!\left(
    \max\!\left\{
    LR_0,\;
    \sigma K^{\frac{1}{\alpha}}
    \right\}
  \right)$
& $\widetilde{\mathcal O}\!\left(
    \max\!\left\{
    \frac{LR_0^2}{K},\;
    \frac{R_0\sigma}{K^{\frac{\alpha-1}{\alpha}}}
    \right\}
  \right)$
\\
\midrule

\rowcolor{blue!8}
\multicolumn{4}{c}{\textbf{$(L_0,L_1)$-smooth case}}
\\
\midrule

\makecell[l]{
\cite{gorbunov2024methods}\\
\cite{vankov2024optimizing}\\
\cite{lobanov2024linear}
}
& Det.
& $\Theta\!\left(\frac{L_0}{L_1}\right)$
& $\mathcal O\!\left(\frac{L_0R_0^2}{K}\right)$
\\
\addlinespace[4pt]

\cite{gaash2025convergence}
& LT
& $\widetilde{\Theta}\!\left(
    \max\!\left\{
    \frac{L_0}{L_1},\;
    \frac{\sigma\sqrt K}{L_1R_0}
    \right\}
  \right)$
& $\widetilde{\mathcal O}\!\left(
    \max\!\left\{
    \frac{L_0R_0^2}{K},\;
    \frac{R_0\sigma}{\sqrt K}
    \right\}
  \right)$
\\
\midrule

\textbf{This work}
& HT
& $\Theta\!\left(
    \max\!\left\{
    \frac{L_0R_0}{M_\delta},
    \;
    \sigma\left(\frac{K}{M_\delta}\right)^{\frac{1}{\alpha}}
    \right\}
  \right)$
& $\mathcal O\!\left(
    \max\!\left\{
    \frac{L_0R_0^2}{K},
    \;
    R_0\sigma
    \left(\frac{M_\delta}{K}\right)^{\frac{\alpha-1}{\alpha}}
    \right\}
  \right)$
\\
\bottomrule
\end{tabular}%
}

\vspace{3pt}
\begin{minipage}{0.985\textwidth}
\footnotesize
\textit{Stochasticity:}
Det. = deterministic;
LT = light-tailed noise;
HT = heavy-tailed noise.
\end{minipage}
\end{table*}

\paragraph{Comparison with closest prior work.}
\label{par:closest-prior-work}
For classical $L$-smooth objectives, \citet{sadiev2023high} obtain the
standard high-probability heavy-tailed rate through a localization argument.
When such a guarantee is transferred to an $(L_0,L_1)$-smooth objective via
local classical smoothness, the effective smoothness constant on a ball of
radius $r$ can scale as $L_0\exp(L_1r)$. Since the relevant localization
radius is of order $R_0$, this reduction can introduce exponential dependence
on $L_1R_0$.

For deterministic $(L_0,L_1)$-smooth optimization,
\citet{gorbunov2024methods} obtain the rate displayed in
Table~\ref{tab:comparison} under a quadratic sufficient transient in
$L_1R_0$. The refined analyses of \citet{vankov2024optimizing} and
\citet{lobanov2024linear} obtain the same rate under a sufficient
iteration condition that depends logarithmically on the initial
objective gap and can be substantially weaker than the quadratic
requirement on favorable instances. In the noiseless case, our guarantee
recovers this same rate under the corresponding refined sufficient iteration
condition, while our stochastic result extends this refined transient picture
to heavy-tailed gradients.
The accelerated method of \citet{vankov2024optimizing} is not included in
Table~\ref{tab:comparison}, as it relies on a line search using function
values.

For comparison with the closest stochastic result, we consider the
sub-Gaussian condition
\begin{align}
\E\!\left[ \exp\!\left( \nicefrac{\sqn{\nabla f(x,\xi)-\nabla f(x)}}{\sigma^2}
        \right) \right] \le \exp(1),
    \qquad \forall x\in\mathbb{R}^d,
    \label{eq:sub_Gaussian}
\end{align}
for some $\sigma>0$.
Together with unbiasedness, this implies \cref{asm: stochastic} for every
$\alpha\in(1,2]$.

The closest stochastic result is \citet{gaash2025convergence}, who obtain
high-probability guarantees for convex $(L_0,L_1)$-smooth objectives
under the light-tailed noise condition in \cref{eq:sub_Gaussian}.
The rate displayed in \cref{tab:comparison} holds under a sufficient
iteration budget quadratic in $L_1R_0$. Their method uses a
double-sampling scheme in which one stochastic gradient determines
the clipping test and an independent stochastic gradient determines
the update. In contrast, we analyze standard \algname{Clip-SGD},
where the same stochastic gradient is clipped and used for the update,
under the weaker bounded central $\alpha$-th moment condition
in \cref{asm: stochastic}, with $\alpha\in(1,2]$.
Our sufficient iteration condition retains the refined dependence
on the initial objective gap discussed above rather than imposing
a quadratic transient uniformly.
Further details on all comparisons above are provided in \cref{app:detailed-convex-comparison}.

\section{Related Work}
\label{sec: rw}

\paragraph{Convergence under $(L_0, L_1)$-smoothness - convex case.}
Early studies of first-order methods under $(L_0,L_1)$-smoothness focused mainly on the non-convex setting
\citep{zhang2019gradient,zhang2020improved,zhao2021convergence,chen2023generalized,
hubler2024parameter,khirirat2024error,crawshaw2022robustness,faw2023beyond,
wang2022provable,wang2023convergence,li2023convergence,bilel2024complexities,liu2025nonconvex,
vankov2024optimizing}.

In the convex setting, \citet{koloskova2023revisiting} provided convergence guarantees for \algname{Clip-GD} under convexity, $(L_0,L_1)$-smoothness and $L$-smoothness, showing that clipping can outperform standard \algname{GD} for suitable parameter choices.
Building on this, \citet{takezawa2024parameter} analyzed \algname{GD} with Polyak stepsizes, while
\citet{li2023convex} considered \algname{GD} and Nesterov's accelerated gradient
\citep{nesterov1983method} under the broader class of $(r,\ell)$-smooth functions.
However, the additional smoothness quantities in
\citet{koloskova2023revisiting, takezawa2024parameter, li2023convex}
can lead to exponentially large factors in $L_1R_0$.
More recent results of
\citet{gorbunov2024methods, vankov2024optimizing, lobanov2024linear}
avoid this drawback and provide tighter bounds for deterministic convex
$(L_0,L_1)$-smooth problems.
Earlier complexity bounds are detailed in \cref{appendix:extra_related}, while recent rates and sufficient iteration conditions are compared in \cref{tab:comparison} and \cref{app:detailed-convex-comparison}.
In addition, \citet{tyurin2024toward} gives a unified analysis of \algname{GD} for convex and non-convex problems under the more general
$\ell(\norm{\nabla f(x)})$-smoothness condition, while
\citet{yu2025mirror} studies Mirror Descent and its variants under a non-Euclidean version of $(r,\ell)$-smoothness \citep{li2023convex}.

Many of these works also give stochastic convergence results
\citep{zhang2019gradient, zhang2020adaptive, zhao2021convergence,
chen2023generalized, crawshaw2022robustness, faw2023beyond,
wang2022provable, wang2023convergence, li2023convergence,
hubler2024parameter, gorbunov2024methods, yu2025mirror}.
\citet{yang2024independently} analyze a variant of \algname{Normalized SGD} with independent normalization, while
\citet{yu2025convergence} study accelerated \algname{SGD} with constant and adaptive stepsizes under $(L_0,L_1)$-smoothness and relaxed affine variance assumptions.
\citet{tovmasyan2025revisiting} introduce $\psi$-smoothness and derive convergence bounds for the Stochastic Proximal Point Method
\citep{bertsekas2011incremental}.
The bounds of \citet{faw2023beyond, wang2023convergence} have inverse-power dependence on $\delta$, whereas
\citet{li2023convergence,yu2025convergence} assume sub-Gaussian noise related to \eqref{eq:sub_Gaussian}.
Further details on light-tailed high-probability guarantees are given in
Appendix~\ref{appendix:extra_related}.

For convex objectives, the closest stochastic result is
\citet{gaash2025convergence}, who derive high-probability rates with
polylogarithmic dependence on $1/\delta$ under sub-Gaussian noise.
Their method uses two independent stochastic gradients, one for the update
direction and another for the clipping multiplier.
While this avoids exponentially large factors in $L_1R_0$, its behavior
under heavy-tailed noise remains unclear.

\paragraph{High-probability convergence under heavy-tailed noise.}
Gradient clipping is a standard approach for obtaining high-probability guarantees under heavy-tailed noise.
Early work includes the truncated Stochastic Mirror Descent method of
\citet{nazin2019algorithms}, which established high-probability bounds for convex and strongly convex problems under bounded variance
(Assumption~\ref{asm: stochastic} with $\alpha=2$).
\citet{gorbunov2020stochastic} later gave the first comprehensive high-probability bounds for
\algname{Clip-SGD} (Algorithm~\ref{alg:clip-sgd}) and introduced an accelerated variant based on the Stochastic Similar Triangles Method
(\algname{SSTM}) \citep{gasnikov2016universal}.
Subsequent work extended these results to Assumption~\ref{asm: stochastic} with $\alpha<2$
\citep{cutkosky2021high, sadiev2023high, nguyen2023improved,
gorbunov2024high, chezhegov2026highprobability}.
Extensions to non-smooth optimization, variational inequalities, adaptive methods, zeroth-order methods, distributed optimization, and differential privacy are discussed in Appendix~\ref{appendix:extra_related}.

Despite these advances, it remained open whether standard \algname{Clip-SGD} admits high-probability guarantees for convex $(L_0,L_1)$-smooth objectives under heavy-tailed noise, with polylogarithmic dependence on $1/\delta$ and without exponential dependence on $L_1R_0$.

\section{High-Probability Guarantees for Convex Objectives}
\label{sec:main_results}

In this section, we present our main high-probability convergence guarantees
for \algname{Clip-SGD} under convexity, $(L_0,L_1)$-smoothness, and heavy-tailed
stochastic gradients. 

\subsection{Best-Iterate Convergence}
\label{subsec:best_iterate}

We establish a high-probability convergence guarantee for the best
iterate generated by \algname{Clip-SGD}.

The case $L_0R_0+\sigma=0$ is immediate. Indeed, if $R_0=0$, then
$x_0=x^\star$. If $R_0>0$, then $L_0=\sigma=0$, and since $f$ admits
a minimizer, Assumption~\ref{asm: smoothness} implies that $f$ is
constant. Therefore, in either case, $\min_{0\le k<K}\left(f(x_k)-f^\star\right)=0.$
We therefore restrict attention to $L_0R_0+\sigma>0$. This condition ensures that the
clipping level in the following theorem is strictly positive and, consequently, that the displayed stepsize is well-defined.
\begin{theorem}
\label{thm:best_iterate}
Suppose that Assumptions~\ref{asm: convexity},
\ref{asm: smoothness}, and \ref{asm: stochastic} hold.
Let $\delta\in(0,1)$ and assume that $L_0 R_0+\sigma>0$. Suppose that $K\ge\frac{160}{\nu^2}\max\left\{L_1R_0,
\log\frac{4}{\delta}\right\}^2.$ Choose the clipping level $\lambda$ and stepsize $\gamma$
as specified in \cref{tab:best-iterate-parameters}.
Then, with probability at least $1-\delta$,
$$
\min_{0\le k<K}\left(f(x_k)-f^\star\right) \le\frac{40}{\nu} \max\left\{\frac{L_0R_0^2}{K},
\,\frac{R_0\sigma\max\left\{L_1R_0,\log\frac{4}{\delta}\right\}^{\frac{\alpha-1}{\alpha}}
}{K^{\frac{\alpha-1}{\alpha}}}\right\}.
$$
\end{theorem}
\paragraph{Proof overview.}
\textbf{1) Deterministic backbone.}
We first recall the standard argument for gradient descent on a convex $L$-smooth function. With $x_{k+1}=x_k-\frac{1}{L}\nabla f(x_k)$, we have
$$
\norm{x_{k+1}-x^\star}^2=\norm{x_k-x^\star}^2-\frac{2}{L}\ip{\nabla f(x_k)}{x_k-x^\star}+\frac{1}{L^2}\norm{\nabla f(x_k)}^2.
$$
The argument relies on the two inequalities
$$
\ip{\nabla f(x_k)}{x_k-x^\star}\ge f(x_k)-f^\star,
\qquad
\norm{\nabla f(x_k)}^2\le L\ip{\nabla f(x_k)}{x_k-x^\star},
$$
where the first follows from convexity and the second from convexity and $L$-smoothness. The second inequality absorbs the quadratic term into the progress term, while the first converts the remaining progress into an objective-gap bound:
$$
\frac{1}{L}\bigl(f(x_k)-f^\star\bigr)\le\norm{x_k-x^\star}^2-\norm{x_{k+1}-x^\star}^2.
$$
Our analysis follows the same structure for the clipped population gradient $q_k:=\clip(\nabla f(x_k),\lambda)$. Under convexity and $(L_0,L_1)$-smoothness, our structural lemmas establish
$$
\ip{q_k}{x_k-x^\star}\ge\min\left\{f(x_k)-f^\star,\frac{\nu\lambda^2}{2(L_0+L_1\lambda)}\right\},
\qquad
\norm{q_k}^2\le\frac{2(L_0+L_1\lambda)}{\nu}\ip{q_k}{x_k-x^\star}.
$$
The second inequality again allows us to absorb the quadratic term generated by the update. The first converts the remaining progress into the objective gap, truncated at the level $\frac{\nu\lambda^2}{2(L_0+L_1\lambda)}$.

Crucially, these inequalities depend on the scale $L_0+L_1\lambda$ without requiring a uniform bound on $\norm{\nabla f(x)}$ over the localization region. Such a bound can grow exponentially in $L_1R_0$ under generalized smoothness, introducing an exponentially large local smoothness constant. Working directly with the clipped population gradient avoids this reduction.

\textbf{2) Stochastic perturbations.}
Let $G_k:=\nabla f(x_k,\xi_k)$, $g_k:=\clip(G_k,\lambda)$, and $e_k:=g_k-q_k$. Since $x_{k+1}=x_k-\gamma(q_k+e_k)$, the structural inequalities imply, for $\gamma\le\frac{\nu}{8(L_0+L_1\lambda)}$,
\begin{equation}
\label{eq:stochastic-descent}
\begin{aligned}
\frac{3\gamma}{2}\min\left\{f(x_k)-f^\star,\frac{\nu\lambda^2}{2(L_0+L_1\lambda)}\right\}
&\le \frac{3\gamma}{2}\ip{q_k}{x_k-x^\star} \\
&\le \norm{x_k-x^\star}^2-\norm{x_{k+1}-x^\star}^2 \\
&\quad -2\gamma\ip{e_k-\E[e_k\mid\mathcal F_k]}{x_k-x^\star} \\
&\quad -2\gamma\ip{\E[e_k\mid\mathcal F_k]}{x_k-x^\star}+2\gamma^2\norm{e_k}^2.
\end{aligned}
\end{equation}
Thus, three stochastic terms perturb the deterministic argument: a centered martingale term, the clipping bias, and a quadratic error term.

We first address the clipping bias $b_k:=\E[e_k\mid\mathcal F_k]$. A direct norm bound on $b_k$ can be too crude when the population gradient is large. Instead, we control its component in the direction $x_k-x^\star$, which is the quantity appearing in the distance recursion. Our key ingredient, \cref{lem:directional-clipping-bias}, gives
$$
-\ip{b_k}{x_k-x^\star}\le\frac{\sigma}{\lambda}\ip{q_k}{x_k-x^\star}+\norm{x_k-x^\star}\frac{\sigma^\alpha}{\lambda^{\alpha-1}}.
$$
Suppose temporarily that $\norm{x_k-x^\star}\le R$. For $\lambda\ge8\sigma$, this yields
$$
-2\gamma\ip{b_k}{x_k-x^\star}\le\frac{\gamma}{4}\ip{q_k}{x_k-x^\star}+2\gamma R\frac{\sigma^\alpha}{\lambda^{\alpha-1}}.
$$
The first term is absorbed into the $\frac{3\gamma}{2}\ip{q_k}{x_k-x^\star}$ progress term in \cref{eq:stochastic-descent}, leaving $\frac{5\gamma}{4}\ip{q_k}{x_k-x^\star}$. The first structural inequality then converts this remaining progress into the truncated objective gap. Thus, the bias contributes only the heavy-tail remainder $2\gamma R\sigma^\alpha/\lambda^{\alpha-1}$ after absorption.

We next control the quadratic error. Although \cref{asm: stochastic} only assumes a bounded $\alpha$-th moment, clipping gives a conditional second-moment bound. Indeed, nonexpansiveness of clipping and the bound on the clipped vectors imply
$$
\norm{e_k}\le\min\left\{\norm{G_k-\nabla f(x_k)},2\lambda\right\},
\qquad
\E[\norm{e_k}^2\mid\mathcal F_k]\le2^{2-\alpha}\sigma^\alpha\lambda^{2-\alpha}.
$$
In particular, this bound does not require the population gradient to lie below the clipping threshold. Splitting the cumulative quadratic error into its conditional mean and a centered martingale sum, and applying Freedman's inequality, gives the high-probability bound in \cref{lem:quadratic-error-hp}.

Finally, let $\zeta_k:=e_k-b_k$. The term $-2\gamma\ip{\zeta_k}{x_k-x^\star}$ is a martingale difference. Under the temporary radius bound, its magnitude is of order $\gamma R\lambda$, and its conditional variance is controlled by $\gamma^2R^2\E[\norm{e_k}^2\mid\mathcal F_k]$. The same second-moment estimate therefore allows another application of Freedman's inequality to control its cumulative contribution.

\textbf{3) Localization.}
We now justify the temporary radius bound. Define $r_k:=\norm{x_k-x^\star}$ and $\tau:=\inf\{k\ge0:r_k>2R_0\}$, with $\inf\varnothing=+\infty$. Multiplying \cref{eq:stochastic-descent} by $\mathbf 1_{\{k<\tau\}}$ and summing makes the distance terms telescope, since
$\mathbf 1_{\{k<\tau\}}\bigl(r_k^2-r_{k+1}^2\bigr)=r_{k\wedge\tau}^2-r_{(k+1)\wedge\tau}^2.$ Combining the bias and concentration bounds for the stopped process, with $R=2R_0$, gives, with probability at least $1-\delta$,
$$
r_{K\wedge\tau}^2+\frac{5\gamma}{4}\sum_{k=0}^{K-1}\mathbf 1_{\{k<\tau\}}
\min\left\{f(x_k)-f^\star,\frac{\nu\lambda^2}{2(L_0+L_1\lambda)}\right\}\le R_0^2+A_K.
$$
Here $A_K$ collects the accumulated bias, martingale fluctuations, and quadratic error. The parameter choices in \cref{thm:best_iterate} ensure $A_K\le2R_0^2$, as verified in the proof of that theorem.

On this event, $r_{K\wedge\tau}^2\le3R_0^2$. If $\tau\le K$, the definition of $\tau$ would instead give $r_{K\wedge\tau}^2=r_\tau^2>4R_0^2$, a contradiction. Hence $\tau>K$. Thus, the temporary radius bound is valid for the entire trajectory up to time $K$, and the stopping indicators can be removed:
$$
\frac1K\sum_{k=0}^{K-1}\min\left\{f(x_k)-f^\star,\frac{\nu\lambda^2}{2(L_0+L_1\lambda)}\right\}\le\frac{12R_0^2}{5\gamma K}.
$$
The assumed lower bound on $K$ places the right-hand side strictly below the truncation level. Therefore, at least one iterate has an untruncated objective gap at most $\frac{12R_0^2}{5\gamma K}$.

To see how the rate follows, write $M:=\max\{L_1R_0,\log(4/\delta)\}$. The theorem chooses
$$
\lambda=\max\left\{\frac{L_0R_0}{M},\sigma\left(\frac{K}{M}\right)^{1/\alpha}\right\},
\qquad
\gamma=\frac{\nu R_0}{16\lambda M}.
$$
The first term in $\lambda$ ensures $L_0\le\lambda M/R_0$; together with
$L_1R_0\le M$, this gives $L_0+L_1\lambda\le2\lambda M/R_0$, which
ensures the required smoothness restriction on $\gamma$. The second term ensures $(\sigma/\lambda)^\alpha\le M/K$, which keeps the heavy-tail remainder at the scale of the final optimization error. Substituting these choices into the preceding bound gives
$$
\min_{0\le k<K}\bigl(f(x_k)-f^\star\bigr)\le\frac{40}{\nu}\max\left\{\frac{L_0R_0^2}{K},\frac{R_0\sigma M^{(\alpha-1)/\alpha}}{K^{(\alpha-1)/\alpha}}\right\}.
$$

\subsection{A Computable Output}
\label{sec:computable-output}

The guarantee in \cref{thm:best_iterate} concerns the best iterate, which is generally not observable. We now show that the same convergence rate can be achieved by an output
constructed solely from the stochastic gradients already queried by
\algname{Clip-SGD}, without any additional oracle calls.

\begin{theorem}
\label{prop:computable-output}
Suppose that Assumptions~\ref{asm: convexity},
\ref{asm: smoothness}, and \ref{asm: stochastic} hold.
Let $\delta\in(0,1)$ and assume that $L_0,R_0>0$.
Suppose that $
K\ge1024\left[\log^2\left(\frac{12}{\delta}\right)+L_1R_0\log\left\{\frac{12e}{\delta}
\left(1+\frac{K\Delta_0}{L_0R_0^2}\right)\right\}\right].$
Choose the clipping level $\lambda$ and stepsize $\gamma$
as specified in \cref{tab:computable-output-parameters}.
Let $\mathcal U=\{0\le k<K:\norm{G_k}\le\lambda\}$, and set
$\widehat x_K=|\mathcal U|^{-1}\sum_{k\in\mathcal U}x_k$
when $\mathcal U\ne\varnothing$, and $\widehat x_K=x_0$ otherwise.
Then, with probability at least $1-\delta$,
$$
f(\widehat x_K)-f^\star
\le\frac{128}{\nu}\max\left\{
\frac{L_0R_0^2}{K},
\frac{
R_0\sigma
\max\left\{L_1R_0,\log\frac{12}{\delta}\right\}^{\frac{\alpha-1}{\alpha}}
}{
K^{\frac{\alpha-1}{\alpha}}
}
\right\}.
$$
\end{theorem}

\paragraph{Proof overview.}
We show that the observed clipping test selects enough iterates, control
a trimmed average used only in the analysis, and transfer its guarantee
to the actual output.

\textbf{1) Ensuring enough selected iterates.}
Let $G_k:=\nabla f(x_k,\xi_k)$,
$q_k:=\clip(\nabla f(x_k),\lambda)$, and
$M:=\max\{L_1R_0,\log(12/\delta)\}$.
Repeating the localization argument with $\delta$ replaced by $\delta/3$,
while retaining the population progress term, gives $\tau>K$,
$\norm{x_k-x^\star}\le2R_0$ for all $k\le K$, and $\sum_{k<K}\ip{q_k}{x_k-x^\star}\le 8R_0^2/(5\gamma)$ with probability at least $1-\delta/3$.
We combine this event with the concentration events below by a union
bound; their intersection has probability at least $1-\delta$.

For our parameter choices, \cref{lem:clipped-drift} gives
$\E[f(x_{k+1})-f(x_k)\mid\mathcal F_k] \le -\frac{3\gamma}{4}\ip{\nabla f(x_k)}{q_k}
+\frac{R_0\lambda}{K}.$

Let $\mathcal C:=\{0\le k<K:\norm{\nabla f(x_k)}<5\lambda/6\}$
and $D_k:=f(x_k)-f^\star+R_0\lambda$.
For $k<\tau$ and $k\notin\mathcal C$, convexity and clipping imply
$\ip{\nabla f(x_k)}{q_k}\ge\lambda D_k/(4R_0)$.
For the iterations $k< \tau$ and $k \in \mathcal{C}$, we drop the nonpositive descent term $-\frac{3\gamma}{4}\ip{\nabla f(x_k)}{q_k}$.
Since $D_k\ge R_0\lambda$, both cases yield, for $k<\tau$,
$$
\E[D_{k+1}\mid\mathcal F_k] \le \left(1-\frac{3\gamma\lambda}{16R_0}\mathbf1_{\{k\notin\mathcal C\}}+\frac1K\right)D_k \le \exp\left(-\frac{3\gamma\lambda}{16R_0}\mathbf1_{\{k\notin\mathcal C\}}+\frac1K\right)D_k.
$$
Every iteration with a large population gradient therefore contributes a fixed negative term to the exponent. Since the recurrence holds for $k<\tau$, we apply
it up to the stopping time $\tau$, which gives an exponential-moment
bound on the number of large-gradient iterations before $\tau$. The exponential form of Markov's inequality then gives $$
\sum_{k<K}\mathbf1_{\{k<\tau,\;k\notin\mathcal C\}} \le\frac{16R_0}{3\gamma\lambda}\log\left(\frac{6e}{\delta}\left(1+\frac{\Delta_0}{R_0\lambda}\right)
\right).$$
This explains the logarithmic dependence on the initial objective gap.
The iteration requirement makes the right-hand side at most $K/4$,
so localization gives $|\mathcal C|\ge3K/4$;
see \cref{lem:population-core}.

Define $\mathcal U:=\{0\le k<K:\norm{G_k}\le\lambda\}$.
If $k\in\mathcal C$ is not selected, then
$\norm{G_k-\nabla f(x_k)}>\lambda/6$, so
$\Pr(k\in\mathcal C\setminus\mathcal U\mid\mathcal F_k)
\le(6\sigma/\lambda)^\alpha$.
Martingale concentration bounds the number of these rejected indices by
$K/12$, giving $ |\mathcal U| \ge|\mathcal C|-|\mathcal C\setminus\mathcal U|
\ge\frac{2K}{3}.$ The details are given in \cref{lem:observable-selection}.

\textbf{2) Controlling the trimmed average.}
A selected iterate can have a large population gradient only if the
stochastic error is large enough to reduce the observed gradient below
the threshold. Introduce the exceptional set $\mathcal O:=\left\{k\in\mathcal U:\norm{G_k-\nabla f(x_k)}>\sigma\left(K/\log(12/\delta)\right)^{1/\alpha} \right\}.$
This exclusion is used only in the analysis; the actual output retains
all indices in $\mathcal U$.
The moment assumption and martingale concentration give
$|\mathcal O|<4\log(12/\delta)$.
We control the remaining selected indices by splitting according to
their population gradient norms.

For $\norm{\nabla f(x_k)}\le2\lambda$, convexity gives
$f(x_k)-f^\star\le2\ip{q_k}{x_k-x^\star}$.
The population progress bound therefore controls their total contribution.

For $\norm{\nabla f(x_k)}>2\lambda$, selection $k\in \mathcal U$ requires
$\norm{G_k-\nabla f(x_k)}\ge\norm{\nabla f(x_k)}-\lambda$.
On this large-gradient event,
$$
\Pr(k\in\mathcal U\mid\mathcal F_k) \le  \frac{\sigma^\alpha}{\bigl(\norm{\nabla f(x_k)}-\lambda\bigr)^\alpha}.
$$
This upper bound decreases with the population gradient, while convexity
and localization give $f(x_k)-f^\star\le2R_0\norm{\nabla f(x_k)}$.
Define $Y_k:=\mathbf1_{\{k<\tau,\;k\in\mathcal U\setminus\mathcal O,\;\norm{\nabla f(x_k)}>2\lambda\}}\bigl(f(x_k)-f^\star\bigr).$ Combining the preceding bounds gives
$
\E[Y_k\mid\mathcal F_k] \le 4R_0\sigma^\alpha/ \lambda^{\alpha-1}$ and $
\sum_{k<K}\E[Y_k\mid\mathcal F_k]\le4R_0\lambda M,$
where the second inequality uses $(\sigma/\lambda)^\alpha\le M/K$.
Thus, rare selection compensates for the potentially large objective gap. On $\mathcal U\setminus\mathcal O$, we also have
$
\norm{\nabla f(x_k)}
\le
\lambda+\sigma\left(   K/\log(12/\delta) \right)^{1/\alpha}.
$
Together with the stopping indicator, this bounds $Y_k$ deterministically.
Martingale concentration then gives
$\sum_{k<K}Y_k=O(R_0\lambda M).$
Combining this with the small-gradient contribution and
$\gamma=\nu R_0/(16\lambda M)$ gives
$$
\sum_{k\in\mathcal U\setminus\mathcal O}
\bigl(f(x_k)-f^\star\bigr)
=
O\left(\frac{R_0\lambda M}{\nu}\right).
$$
Since $|\mathcal U|\ge2K/3$ and $|\mathcal O|<4\log(12/\delta)$,
the iteration requirement ensures that
$\mathcal U\setminus\mathcal O$ contains a constant fraction of all
iterates. The trimmed average
$
\bar x_K:=
\frac{1}{|\mathcal U\setminus\mathcal O|}
\sum_{k\in\mathcal U\setminus\mathcal O}x_k
$
is therefore well-defined on the common event, and convexity gives
$
\Delta:=f(\bar x_K)-f^\star = O\left(\frac{R_0\lambda M}{\nu K}\right).
$

\textbf{3) Transferring the guarantee to the actual output.}
The actual output $\widehat x_K$ averages all of $\mathcal U$.
Localization bounds the distance between any iterate and $\bar x_K$
by $4R_0$, so

$$
D_K:=\norm{\widehat x_K-\bar x_K}
\le
\frac{4R_0|\mathcal O|}{|\mathcal U|}
<
\frac{24R_0\log(12/\delta)}{K}.
$$
The iteration requirement ensures $L_1 D_K<1/32$.
\cref{lemma_grad_bound} gives
$\norm{\nabla f(\bar x_K)}\le4L_1\Delta+2\sqrt{L_0\Delta}$.
Combining this with local smoothness and Young's inequality yields $f(\widehat x_K)-f^\star \le \frac54\Delta+12L_0D_K^2.$ Moreover, the iteration requirement and $\lambda M\ge L_0R_0$ give
$L_0D_K^2=O(R_0\lambda M/K)$.
Thus, restoring the excluded indices preserves the rate.
Keeping the numerical constants, as verified in the full proof, gives
$$
f(\widehat x_K)-f^\star
\le
\frac{128}{\nu}\frac{R_0\lambda M}{K}
=
\frac{128}{\nu}
\max\left\{
\frac{L_0R_0^2}{K},\,
R_0\sigma\left(\frac{M}{K}\right)^{(\alpha-1)/\alpha}
\right\}.
$$
The common event has probability at least $1-\delta$, proving the claim.
\subsection{Optimality for Fixed-Parameter Clip-SGD}
\label{sec:convex-lower-bound}
We next establish a lower bound for \algname{Clip-SGD} with any fixed stepsize and clipping level. Under the iteration conditions below, it matches the dependence of the stochastic term in
\cref{prop:computable-output} on $L_1R_0$ and $\log(1/\delta)$ up to
logarithmic factors in the iteration budget, even when the output
is an arbitrary randomized function of the observed transcript.
We state the result for $L_0=1$.

\begin{theorem}
\label{thm:fixed-clip-lower-bound}
Let $L_1,R_0,\sigma>0$, $\alpha\in(1,2]$, and
$\delta\in(0,1/4)$. Suppose the integer $K$ satisfies
$$
K\ge \max\left\{1,\,L_1R_0,\,\sigma L_1,\,\log\frac{1}{2\delta}\max\left\{
1,\,\left(\frac{\sigma}{R_0}\right)^{\alpha/(\alpha-1)}\right\}\right\}.
$$
Then, for every fixed $\gamma,\lambda>0$ and every fixed possibly
randomized output rule depending only on the observed transcript and
independent auxiliary randomness, let $x_K^{\mathrm{out}}$ denote the
point returned by this rule. There exist a convex
$(1,L_1)$-smooth function $f:\mathbb R\to\mathbb R$, with a minimizer
$x^\star$ satisfying $|x^\star|\le R_0$, and an unbiased
stochastic-gradient oracle satisfying \cref{asm: stochastic}, such that
\algname{Clip-SGD}, started from $x_0=0$ and using fresh independent
oracle samples, satisfies
$$
\mathbb P\left(f(x_K^{\mathrm{out}})-f^\star\ge \frac{\sigma R_0}{128}
\left[\frac{\max\left\{L_1R_0/\log(eK),\,\log(12/\delta)\right\}}{K}
\right]^{(\alpha-1)/\alpha}\right)>\delta.
$$
The fixed parameters and output rule may depend on $K$ and $\delta$.
\end{theorem}
\section{Conclusion}

We established high-probability convergence guarantees for standard
\algname{Clip-SGD} on convex $(L_0,L_1)$-smooth objectives under
heavy-tailed noise, without exponential dependence on $L_1R_0$.
The analysis combines a directional clipping-bias bound with structural
inequalities for the clipped population gradient. The resulting guarantee
applies to a computable average of iterates selected by the observed
clipping test, requiring neither function values nor additional oracle
calls. We also established a lower bound showing that, under the stated
iteration conditions, our convex stochastic rate, including its confidence
dependence, is optimal up to logarithmic factors in the iteration budget
for \algname{Clip-SGD} with any fixed stepsize and clipping level.
We further obtained high-probability bounds for the average gradient norm
on nonconvex $(L_0,L_1)$-smooth objectives under the same heavy-tailed
noise assumption; see \cref{appendix_nonconvex_extension}. Our convex
and nonconvex bounds recover the corresponding noiseless rates and the
classical-smooth stochastic iteration rates when $L_1=0$. In the nonconvex
case, the asymptotically dominant stochastic term has no explicit $L_1$
dependence and no logarithmic factors in $K$ or $1/\delta$.

Interesting directions include understanding whether adaptive parameter
choices or different algorithms can improve the convex stochastic rate,
and developing accelerated methods with comparable high-probability
guarantees under heavy-tailed noise.

We hope that our results inspire further research in these directions and contribute to the broader understanding of stochastic optimization under realistic noise and smoothness assumptions.

\bibliographystyle{apalike}
\bibliography{refs}

\begin{thebibliography}{}

\bibitem[An et~al., 2025]{an2025tight}
An, W., Liu, Y., Shang, F., Yu, H., Liu, J., and Liu, H. (2025).
\newblock Tight high-probability bounds for nonconvex heavy-tailed scenario under weaker assumptions.
\newblock In {\em The Thirty-ninth Annual Conference on Neural Information Processing Systems}.

\bibitem[Bertsekas, 2011]{bertsekas2011incremental}
Bertsekas, D.~P. (2011).
\newblock Incremental proximal methods for large scale convex optimization.
\newblock {\em Mathematical programming}, 129(2):163--195.

\bibitem[Beygelzimer et~al., 2011]{beygelzimer2011contextual}
Beygelzimer, A., Langford, J., Li, L., Reyzin, L., and Schapire, R. (2011).
\newblock Contextual bandit algorithms with supervised learning guarantees.
\newblock In {\em Proceedings of the Fourteenth International Conference on Artificial Intelligence and Statistics}, pages 19--26. JMLR Workshop and Conference Proceedings.

\bibitem[Bilel, 2024]{bilel2024complexities}
Bilel, B. (2024).
\newblock Complexities of armijo-like algorithms in deep learning context.
\newblock {\em arXiv preprint arXiv:2412.14637}.

\bibitem[Brown et~al., 2020]{brown2020language}
Brown, T., Mann, B., Ryder, N., Subbiah, M., Kaplan, J.~D., Dhariwal, P., Neelakantan, A., Shyam, P., Sastry, G., Askell, A., et~al. (2020).
\newblock Language models are few-shot learners.
\newblock {\em Advances in neural information processing systems}, 33:1877--1901.

\bibitem[Chen et~al., 2023]{chen2023generalized}
Chen, Z., Zhou, Y., Liang, Y., and Lu, Z. (2023).
\newblock Generalized-smooth nonconvex optimization is as efficient as smooth nonconvex optimization.
\newblock In {\em International Conference on Machine Learning}, pages 5396--5427. PMLR.

\bibitem[Chezhegov et~al., 2024]{chezhegov2024gradient}
Chezhegov, S., Klyukin, Y., Semenov, A., Beznosikov, A., Gasnikov, A., Horv{\'a}th, S., Tak{\'a}{\v{c}}, M., and Gorbunov, E. (2024).
\newblock Clipping improves adam-norm and adagrad-norm when the noise is heavy-tailed.
\newblock {\em arXiv preprint arXiv:2406.04443}.

\bibitem[Chezhegov et~al., 2026]{chezhegov2026highprobability}
Chezhegov, S., Parletta, D.~A., Paudice, A., and Gorbunov, E. (2026).
\newblock High-probability bounds for the last iterate of clipped {SGD}.
\newblock In {\em The Fourteenth International Conference on Learning Representations}.

\bibitem[Crawshaw and Liu, 2025]{crawshaw2025complexity}
Crawshaw, M. and Liu, M. (2025).
\newblock Complexity lower bounds of adaptive gradient algorithms for non-convex stochastic optimization under relaxed smoothness.
\newblock In {\em International Conference on Learning Representations}.

\bibitem[Crawshaw et~al., 2022]{crawshaw2022robustness}
Crawshaw, M., Liu, M., Orabona, F., Zhang, W., and Zhuang, Z. (2022).
\newblock Robustness to unbounded smoothness of generalized signsgd.
\newblock {\em Advances in neural information processing systems}, 35:9955--9968.

\bibitem[Cutkosky and Mehta, 2021]{cutkosky2021high}
Cutkosky, A. and Mehta, H. (2021).
\newblock High-probability bounds for non-convex stochastic optimization with heavy tails.
\newblock {\em Advances in Neural Information Processing Systems}, 34:4883--4895.

\bibitem[Devlin et~al., 2019]{devlin2019bert}
Devlin, J., Chang, M.-W., Lee, K., and Toutanova, K. (2019).
\newblock Bert: Pre-training of deep bidirectional transformers for language understanding.
\newblock In {\em Proceedings of the 2019 Conference of the North American Chapter of the Association for Computational Linguistics: Human Language Technologies, Volume 1 (Long and Short Papers)}, pages 4171--4186.

\bibitem[Dvurechensky and Gasnikov, 2016]{dvurechensky2016stochastic}
Dvurechensky, P. and Gasnikov, A. (2016).
\newblock Stochastic intermediate gradient method for convex problems with stochastic inexact oracle.
\newblock {\em Journal of Optimization Theory and Applications}, 171:121--145.

\bibitem[Faw et~al., 2023]{faw2023beyond}
Faw, M., Rout, L., Caramanis, C., and Shakkottai, S. (2023).
\newblock Beyond uniform smoothness: A stopped analysis of adaptive sgd.
\newblock In {\em The Thirty Sixth Annual Conference on Learning Theory}, pages 89--160. PMLR.

\bibitem[Fedus et~al., 2022]{fedus2022switch}
Fedus, W., Zoph, B., and Shazeer, N. (2022).
\newblock Switch transformers: Scaling to trillion parameter models with simple and efficient sparsity.
\newblock {\em Journal of Machine Learning Research}, 23(120):1--39.

\bibitem[Gaash et~al., 2025]{gaash2025convergence}
Gaash, O., Levy, K.~Y., and Carmon, Y. (2025).
\newblock Convergence of clipped {SGD} on convex $(l_0,l_1)$-smooth functions.
\newblock In {\em The Thirty-ninth Annual Conference on Neural Information Processing Systems}.

\bibitem[Gasnikov and Nesterov, 2016]{gasnikov2016universal}
Gasnikov, A. and Nesterov, Y. (2016).
\newblock Universal fast gradient method for stochastic composit optimization problems.
\newblock {\em arXiv preprint arXiv:1604.05275}.

\bibitem[Ghadimi and Lan, 2012]{ghadimi2012optimal}
Ghadimi, S. and Lan, G. (2012).
\newblock Optimal stochastic approximation algorithms for strongly convex stochastic composite optimization i: A generic algorithmic framework.
\newblock {\em SIAM Journal on Optimization}, 22(4):1469--1492.

\bibitem[Ghadimi and Lan, 2013]{ghadimi2013stochastic}
Ghadimi, S. and Lan, G. (2013).
\newblock Stochastic first-and zeroth-order methods for nonconvex stochastic programming.
\newblock {\em SIAM Journal on Optimization}, 23(4):2341--2368.

\bibitem[Goodfellow et~al., 2016]{goodfellow2016deep}
Goodfellow, I., Bengio, Y., and Courville, A. (2016).
\newblock {\em Deep learning}.
\newblock MIT press.

\bibitem[Gorbunov et~al., 2022]{gorbunov2022clipped}
Gorbunov, E., Danilova, M., Dobre, D., Dvurechenskii, P., Gasnikov, A., and Gidel, G. (2022).
\newblock Clipped stochastic methods for variational inequalities with heavy-tailed noise.
\newblock {\em Advances in Neural Information Processing Systems}, 35:31319--31332.

\bibitem[Gorbunov et~al., 2020]{gorbunov2020stochastic}
Gorbunov, E., Danilova, M., and Gasnikov, A. (2020).
\newblock Stochastic optimization with heavy-tailed noise via accelerated gradient clipping.
\newblock {\em Advances in Neural Information Processing Systems}, 33:15042--15053.

\bibitem[Gorbunov et~al., 2024a]{gorbunov2024high_non_smooth}
Gorbunov, E., Danilova, M., Shibaev, I., Dvurechensky, P., and Gasnikov, A. (2024a).
\newblock High-probability complexity bounds for non-smooth stochastic convex optimization with heavy-tailed noise.
\newblock {\em Journal of Optimization Theory and Applications}, pages 1--60.

\bibitem[Gorbunov et~al., 2024b]{gorbunov2024high}
Gorbunov, E., Sadiev, A., Danilova, M., Horv\'{a}th, S., Gidel, G., Dvurechensky, P., Gasnikov, A., and Richt\'{a}rik, P. (2024b).
\newblock High-probability convergence for composite and distributed stochastic minimization and variational inequalities with heavy-tailed noise.
\newblock In Salakhutdinov, R., Kolter, Z., Heller, K., Weller, A., Oliver, N., Scarlett, J., and Berkenkamp, F., editors, {\em Proceedings of the 41st International Conference on Machine Learning}, volume 235 of {\em Proceedings of Machine Learning Research}, pages 15951--16070. PMLR.

\bibitem[Gorbunov et~al., 2025]{gorbunov2024methods}
Gorbunov, E., Tupitsa, N., Choudhury, S., Aliev, A., Richt{\'a}rik, P., Horv{\'a}th, S., and Tak{\'a}{\v{c}}, M. (2025).
\newblock Methods for convex $({L}_0, {L}_1) $-smooth optimization: Clipping, acceleration, and adaptivity.
\newblock {\em International Conference on Learning Representations}.

\bibitem[Harvey et~al., 2019]{harvey2019simple}
Harvey, N.~J., Liaw, C., and Randhawa, S. (2019).
\newblock Simple and optimal high-probability bounds for strongly-convex stochastic gradient descent.
\newblock {\em arXiv preprint arXiv:1909.00843}.

\bibitem[H{\"u}bler et~al., 2024]{hubler2024parameter}
H{\"u}bler, F., Yang, J., Li, X., and He, N. (2024).
\newblock Parameter-agnostic optimization under relaxed smoothness.
\newblock In {\em International Conference on Artificial Intelligence and Statistics}, pages 4861--4869. PMLR.

\bibitem[Khah et~al., 2025]{khah2025differentially}
Khah, S.~V., Chezhegov, S., Farahmand, S., Horv{\'a}th, S., and Gorbunov, E. (2025).
\newblock Differentially private clipped-sgd: High-probability convergence with arbitrary clipping level.
\newblock {\em arXiv preprint arXiv:2507.23512}.

\bibitem[Khirirat et~al., 2024]{khirirat2024error}
Khirirat, S., Sadiev, A., Riabinin, A., Gorbunov, E., and Richt{\'a}rik, P. (2024).
\newblock Error feedback under $({L}_0, {L}_1)$-smoothness: Normalization and momentum.
\newblock {\em arXiv preprint arXiv:2410.16871}.

\bibitem[Koloskova et~al., 2023]{koloskova2023revisiting}
Koloskova, A., Hendrikx, H., and Stich, S.~U. (2023).
\newblock Revisiting gradient clipping: Stochastic bias and tight convergence guarantees.
\newblock In {\em International Conference on Machine Learning}, pages 17343--17363. PMLR.

\bibitem[Kornilov et~al., 2023]{kornilov2023accelerated}
Kornilov, N., Shamir, O., Lobanov, A., Dvinskikh, D., Gasnikov, A., Shibaev, I., Gorbunov, E., and Horv{\'a}th, S. (2023).
\newblock Accelerated zeroth-order method for non-smooth stochastic convex optimization problem with infinite variance.
\newblock {\em Advances in Neural Information Processing Systems}, 36:64083--64102.

\bibitem[Lee et~al., 2025]{lee2025efficient}
Lee, S.~H., Zaheer, M., and Li, T. (2025).
\newblock Efficient distributed optimization under heavy-tailed noise.
\newblock {\em arXiv preprint arXiv:2502.04164}.

\bibitem[Li et~al., 2023a]{li2023convex}
Li, H., Qian, J., Tian, Y., Rakhlin, A., and Jadbabaie, A. (2023a).
\newblock Convex and non-convex optimization under generalized smoothness.
\newblock {\em Advances in Neural Information Processing Systems}, 36:40238--40271.

\bibitem[Li et~al., 2023b]{li2023convergence}
Li, H., Rakhlin, A., and Jadbabaie, A. (2023b).
\newblock Convergence of adam under relaxed assumptions.
\newblock In {\em Thirty-seventh Conference on Neural Information Processing Systems}.

\bibitem[Li and Liu, 2023]{li2023high}
Li, S. and Liu, Y. (2023).
\newblock High probability analysis for non-convex stochastic optimization with clipping.
\newblock In {\em ECAI 2023}, pages 1406--1413. IOS Press.

\bibitem[Li and Orabona, 2020]{li2020high}
Li, X. and Orabona, F. (2020).
\newblock A high probability analysis of adaptive sgd with momentum.
\newblock {\em arXiv preprint arXiv:2007.14294}.

\bibitem[Liu et~al., 2023]{liu2023high}
Liu, Z., Nguyen, T.~D., Nguyen, T.~H., Ene, A., and Nguyen, H. (2023).
\newblock High probability convergence of stochastic gradient methods.
\newblock In {\em International Conference on Machine Learning}, pages 21884--21914. PMLR.

\bibitem[Liu and Zhou, 2025]{liu2025nonconvex}
Liu, Z. and Zhou, Z. (2025).
\newblock Nonconvex stochastic optimization under heavy-tailed noises: Optimal convergence without gradient clipping.
\newblock In {\em The Thirteenth International Conference on Learning Representations}.

\bibitem[Lobanov et~al., 2024]{lobanov2024linear}
Lobanov, A., Gasnikov, A., Gorbunov, E., and Tak{\'a}{\v{c}}, M. (2024).
\newblock Linear convergence rate in convex setup is possible! gradient descent method variants under $({L}_0, {L}_1)$-smoothness.
\newblock {\em arXiv preprint arXiv:2412.17050}.

\bibitem[Madden et~al., 2024]{madden2024high}
Madden, L., Dall'Anese, E., and Becker, S. (2024).
\newblock High probability convergence bounds for non-convex stochastic gradient descent with sub-weibull noise.
\newblock {\em Journal of Machine Learning Research}, 25(241):1--36.

\bibitem[Nazin et~al., 2019]{nazin2019algorithms}
Nazin, A.~V., Nemirovsky, A.~S., Tsybakov, A.~B., and Juditsky, A.~B. (2019).
\newblock Algorithms of robust stochastic optimization based on mirror descent method.
\newblock {\em Automation and Remote Control}, 80:1607--1627.

\bibitem[Nemirovski et~al., 2009]{nemirovski2009robust}
Nemirovski, A.~S., Juditsky, A.~B., Lan, G., and Shapiro, A. (2009).
\newblock Robust stochastic approximation approach to stochastic programming.
\newblock {\em SIAM Journal on Optimization}, 19(4):1574--1609.

\bibitem[Nesterov, 1983]{nesterov1983method}
Nesterov, Y.~E. (1983).
\newblock A method for solving the convex programming problem with convergence rate {O}$(1/k^2)$.
\newblock In {\em Dokl. akad. nauk Sssr}, volume 269, pages 543--547.

\bibitem[Nguyen et~al., 2023]{nguyen2023improved}
Nguyen, T.~D., Nguyen, T.~H., Ene, A., and Nguyen, H. (2023).
\newblock Improved convergence in high probability of clipped gradient methods with heavy tailed noise.
\newblock In Oh, A., Naumann, T., Globerson, A., Saenko, K., Hardt, M., and Levine, S., editors, {\em Advances in Neural Information Processing Systems}, volume~36, pages 24191--24222. Curran Associates, Inc.

\bibitem[Parletta et~al., 2024]{parletta2024high}
Parletta, D.~A., Paudice, A., Pontil, M., and Salzo, S. (2024).
\newblock High probability bounds for stochastic subgradient schemes with heavy tailed noise.
\newblock {\em SIAM Journal on Mathematics of Data Science}, 6(4):953--977.

\bibitem[Pascanu et~al., 2013]{pascanu2013difficulty}
Pascanu, R., Mikolov, T., and Bengio, Y. (2013).
\newblock On the difficulty of training recurrent neural networks.
\newblock In {\em International conference on machine learning}, pages 1310--1318. Pmlr.

\bibitem[Robbins and Monro, 1951]{robbins1951stochastic}
Robbins, H. and Monro, S. (1951).
\newblock A stochastic approximation method.
\newblock {\em The annals of mathematical statistics}, pages 400--407.

\bibitem[Sadiev et~al., 2023]{sadiev2023high}
Sadiev, A., Danilova, M., Gorbunov, E., Horv\'{a}th, S., Gidel, G., Dvurechensky, P., Gasnikov, A., and Richt\'{a}rik, P. (2023).
\newblock High-probability bounds for stochastic optimization and variational inequalities: the case of unbounded variance.
\newblock In Krause, A., Brunskill, E., Cho, K., Engelhardt, B., Sabato, S., and Scarlett, J., editors, {\em Proceedings of the 40th International Conference on Machine Learning}, volume 202 of {\em Proceedings of Machine Learning Research}, pages 29563--29648. PMLR.

\bibitem[Shalev-Shwartz and Ben-David, 2014]{shalev2014understanding}
Shalev-Shwartz, S. and Ben-David, S. (2014).
\newblock {\em Understanding machine learning: From theory to algorithms}.
\newblock Cambridge university press.

\bibitem[Takezawa et~al., 2024]{takezawa2024parameter}
Takezawa, Y., Bao, H., Sato, R., Niwa, K., and Yamada, M. (2024).
\newblock Parameter-free clipped gradient descent meets polyak.
\newblock {\em Advances in Neural Information Processing Systems}, 37:44575--44599.

\bibitem[Touvron et~al., 2023]{touvron2023llama}
Touvron, H., Lavril, T., Izacard, G., Martinet, X., Lachaux, M.-A., Lacroix, T., Rozi{\`e}re, B., Goyal, N., Hambro, E., Azhar, F., et~al. (2023).
\newblock Llama: Open and efficient foundation language models.
\newblock {\em arXiv preprint arXiv:2302.13971}.

\bibitem[Tovmasyan et~al., 2025]{tovmasyan2025revisiting}
Tovmasyan, Z., Malinovsky, G., Condat, L., and Richt{\'a}rik, P. (2025).
\newblock Revisiting stochastic proximal point methods: Generalized smoothness and similarity.
\newblock {\em arXiv preprint arXiv:2502.03401}.

\bibitem[Tyurin, 2024]{tyurin2024toward}
Tyurin, A. (2024).
\newblock Toward a unified theory of gradient descent under generalized smoothness.
\newblock {\em arXiv preprint arXiv:2412.11773}.

\bibitem[Vankov et~al., 2025]{vankov2024optimizing}
Vankov, D., Rodomanov, A., Nedich, A., Sankar, L., and Stich, S.~U. (2025).
\newblock Optimizing $({L}_0, {L}_1) $-smooth functions by gradient methods.
\newblock {\em International Conference on Learning Representations}.

\bibitem[Wang et~al., 2023]{wang2023convergence}
Wang, B., Zhang, H., Ma, Z., and Chen, W. (2023).
\newblock Convergence of adagrad for non-convex objectives: Simple proofs and relaxed assumptions.
\newblock In {\em The Thirty Sixth Annual Conference on Learning Theory}, pages 161--190. PMLR.

\bibitem[Wang et~al., 2022]{wang2022provable}
Wang, B., Zhang, Y., Zhang, H., Meng, Q., Ma, Z.-M., Liu, T.-Y., and Chen, W. (2022).
\newblock Provable adaptivity in adam.
\newblock {\em arXiv preprint arXiv:2208.09900}.

\bibitem[Yang et~al., 2024]{yang2024independently}
Yang, Y., Tripp, E., Sun, Y., Zou, S., and Zhou, Y. (2024).
\newblock Independently-normalized sgd for generalized-smooth nonconvex optimization.
\newblock {\em arXiv preprint arXiv:2410.14054}.

\bibitem[Yu et~al., 2025a]{yu2025convergence}
Yu, C., Hong, Y., and Lin, J. (2025a).
\newblock Convergence analysis of stochastic accelerated gradient methods for generalized smooth optimizations.
\newblock {\em arXiv preprint arXiv:2502.11125}.

\bibitem[Yu et~al., 2025b]{yu2025mirror}
Yu, D., Jiang, W., Wan, Y., and Zhang, L. (2025b).
\newblock Mirror descent under generalized smoothness.
\newblock {\em arXiv preprint arXiv:2502.00753}.

\bibitem[Zhang et~al., 2020a]{zhang2020improved}
Zhang, B., Jin, J., Fang, C., and Wang, L. (2020a).
\newblock Improved analysis of clipping algorithms for non-convex optimization.
\newblock In {\em Advances in Neural Information Processing Systems}.

\bibitem[Zhang et~al., 2020b]{zhang2019gradient}
Zhang, J., He, T., Sra, S., and Jadbabaie, A. (2020b).
\newblock Why gradient clipping accelerates training: A theoretical justification for adaptivity.
\newblock In {\em International Conference on Learning Representations}.

\bibitem[Zhang et~al., 2020c]{zhang2020adaptive}
Zhang, J., Karimireddy, S.~P., Veit, A., Kim, S., Reddi, S., Kumar, S., and Sra, S. (2020c).
\newblock Why are adaptive methods good for attention models?
\newblock {\em Advances in Neural Information Processing Systems}, 33:15383--15393.

\bibitem[Zhao et~al., 2021]{zhao2021convergence}
Zhao, S.-Y., Xie, Y.-P., and Li, W.-J. (2021).
\newblock On the convergence and improvement of stochastic normalized gradient descent.
\newblock {\em Science China Information Sciences}, 64:1--13.

\end{thebibliography}

\clearpage

\appendix

\section{Additional Related Work}
\label{appendix:extra_related}
This section complements the related-work discussion in the main text by providing additional details on deterministic convex guarantees, high-probability results under light-tailed noise, and further applications of gradient clipping.

\paragraph{Detailed deterministic convex guarantees.}
\citet{koloskova2023revisiting} provided convergence guarantees for \algname{Clip-GD} under convexity, $(L_0,L_1)$-smoothness and $L$-smoothness, deriving a complexity bound of
$$\cO\left(\max\left\{\nicefrac{(L_0+\lambda L_1)R_0^2}{\varepsilon},\sqrt{\nicefrac{R_0^4L(L_0+\lambda L_1)^2}{\lambda^2\varepsilon}}\right\}\right).$$
The leading term in this complexity bound is independent of $L_1$ and $L$ if $\lambda\sim\nicefrac{L_0}{L_1}$, and can significantly improve upon the corresponding bound for standard \algname{GD}. Building on this, \citet{takezawa2024parameter} analyzed \algname{GD} with Polyak stepsizes and derived the complexity bound
$$\cO\left(\max\left\{\nicefrac{L_0R_0^2}{\varepsilon},\sqrt{\nicefrac{R_0^4LL_1^2}{\varepsilon}}\right\}\right).$$
For convex $(L_0,L_1)$-smooth objectives, \citet{li2023convex} obtain $\cO(1/\varepsilon)$ and $\cO(1/\sqrt{\varepsilon})$ iteration complexities for \algname{GD} and Nesterov's accelerated gradient method \citep{nesterov1983method}, respectively, with constants depending on the generalized-smoothness parameters and the initialization. Through the additional smoothness quantities appearing in these analyses, the bounds from \citet{koloskova2023revisiting, takezawa2024parameter, li2023convex} can introduce exponential dependence on $L_1R_0$, a drawback addressed by the more recent results of \citet{gorbunov2024methods, vankov2024optimizing, lobanov2024linear}.

\paragraph{High-probability convergence under light-tailed noise.}
High-probability convergence guarantees have long been a critical component in the analysis of stochastic first-order methods, particularly when the noise in the stochastic gradients is light-tailed. In these settings, methods like \algname{SGD} and its variants can achieve convergence rates with polylogarithmic dependence on the inverse failure probability $1/\delta$. Under sub-Gaussian noise assumptions, this behavior has been rigorously established for \algname{SGD} \citep{nemirovski2009robust, harvey2019simple}, its accelerated counterparts \citep{ghadimi2012optimal, dvurechensky2016stochastic}, and adaptive methods like \algname{AdaGrad} \citep{li2020high, liu2023high}. Recent extensions to broader classes of noise, such as sub-Weibull distributions, have further expanded this theoretical framework \citep{madden2024high}.
The results of \citet{li2023convergence, yu2025convergence} rely on light-tailed noise assumptions related to \cref{eq:sub_Gaussian}. In particular, \citet{yu2025convergence} use a more general version of \cref{eq:sub_Gaussian}, with $\sigma^2$ replaced by
$$A(f(x)-f^\star)+B\norm{\nabla f(x)}^2+C.$$

\paragraph{Further results under heavy-tailed noise.}
High-probability clipping guarantees have also been extended to broader problem classes, including nonsmooth optimization \citep{gorbunov2024high_non_smooth, parletta2024high} and unconstrained variational inequalities \citep{gorbunov2022clipped}.
Adaptive variants have also been developed: \citet{li2023high} analyzed \algname{Clip-AdaGrad} with scalar stepsizes, while \citet{chezhegov2024gradient} obtained similar bounds for both scalar and coordinate-wise versions of \algname{Clip-AdaGrad} and \algname{Clip-Adam}. In the zeroth-order setting, \citet{kornilov2023accelerated} proposed a clipped variant of \algname{SSTM}. For distributed optimization, \citet{lee2025efficient} provided in-expectation convergence rates for the \algname{TailOPT} method. Moreover, clipping was adapted to differential privacy with high-probability convergence guarantees \citep{khah2025differentially}.

\section{Preliminary Results and Notation Table}
\label{sec:prelim-results}

\begin{proposition}[{\cite[Lemma~2.2]{gorbunov2024methods}}]
\label{prop:self_bounding}
Suppose Assumption~\ref{asm: smoothness} holds. Then, for all $x\in\R^d$,
$$
\nu \norm{\nabla f(x)}^2\le2\left(L_0+L_1\norm{\nabla f(x)}\right)\left(f(x)-f^\star\right),
$$
where $\nu$ is the solution of $x\exp(x)=1$.
\end{proposition}
\begin{lemma}[Consequence of   {\cite[Lemma 2.1]{gorbunov2024methods}}  ]
\label{lem:local-smoothness-short-segment}
Assume that \cref{asm: smoothness} holds. Let $x,y\in\mathbb{R}^d$, and set
$s \coloneqq y-x$. If $L_1 \norm{s} \leq 1,$ then
$$
f(y) \leq f(x)+ \ip{\nabla f(x)}{s}+ 2\left(L_0 + L_1\norm{\nabla f(x)}\right)\norm{s}^2.
$$
\end{lemma}

\begin{lemma}[{\cite[Lemma A.1]{nguyen2023improved}}]\label{freedman_lemma}
Let $(X_t)_{t\ge 1}$ be a martingale difference sequence with respect to a
filtration $(\mathcal G_t)_{t\ge0}$. Assume that there exists a constant $c$
such that $|X_t|<c$ almost surely for all $t\ge 1$, and define
$$
    \sigma_t^2
    =
    \E\!\left[
        X_t^2
        \mid
        \mathcal G_{t-1}
    \right].
$$
Then, for all $b>0$, $F>0$, and $T\ge 1$,
$$
    \Pr\left[
        \left|\sum_{t=1}^T X_t\right|>b
        \ \text{and}\
        \sum_{t=1}^T \sigma_t^2 < F
    \right]
    \le
    2\exp\left(
        -\frac{b^2}{2F+2cb/3}
    \right).
$$
\end{lemma}

\begin{corollary}
\label{cor:freedman-hp}
Let $(X_t)_{t\ge 1}$ be a martingale difference sequence with respect to a
filtration $(\mathcal G_t)_{t\ge0}$ such that
$$
    |X_t| \le c
    \quad\text{almost surely for all } t\ge1.
$$
Define
$V_T :=\sum_{t=1}^T \E\!\left[X_t^2\mid\mathcal G_{t-1}\right].$
Assume that, for some deterministic $F>0$,
$V_T \le F \quad \text{almost surely}.$
Then, for every $\delta\in(0,1)$, with probability at least $1-\delta$,
$$
\left|\sum_{t=1}^T X_t \right| \le
    \sqrt{ 2F\log\frac{2}{\delta}} + \frac{2c}{3}\log\frac{2}{\delta}.
$$
\end{corollary}

\begin{proof}[Proof of \cref{cor:freedman-hp}]
Let $\varepsilon>0$. Since $|X_t|\le c<c+\varepsilon$ and
$V_T\le F<F+\varepsilon$, we may apply
Lemma~\ref{freedman_lemma} with $c+\varepsilon$ and $F+\varepsilon$.
Let $b_\varepsilon :=\sqrt{2(F+\varepsilon)\log\frac{2}{\delta}}+
\frac{2(c+\varepsilon)}{3}\log\frac{2}{\delta}.$
Then $b_\varepsilon^2 \ge\log\frac{2}{\delta}\left(2(F+\varepsilon) +
 \frac{2(c+\varepsilon)b_\varepsilon}{3}\right).$
Therefore,
$$
2\exp\left(-\frac{b_\varepsilon^2}{2(F+\varepsilon)+2(c+\varepsilon)b_\varepsilon/3}
\right)\le 2e^{-\log\frac{2}{\delta}}=\delta.
$$
Since $V_T<F+\varepsilon$ almost surely, Lemma~\ref{freedman_lemma} gives
$$
    \Pr\left[
        \left|
            \sum_{t=1}^T X_t
        \right|>b_\varepsilon
    \right]
    \le \delta.
$$
Letting $\varepsilon\downarrow0$ gives the result.
\end{proof}

The following proposition is a direct consequence of the
random-variance Freedman inequality of
\citet[Theorem~1]{beygelzimer2011contextual}. It provides an
absorbable bound when the predictable quadratic variation is
controlled by the same nonnegative random quantity that appears
in the descent inequality.

\begin{proposition}\label{prop:freedman-absorbable}
Let $(X_t)_{t\ge1}$ be a martingale-difference sequence with respect
to $(\mathcal G_t)_{t\ge0}$ such that $X_t\le c$ almost surely. Suppose that, for some nonnegative random variable $S$ and deterministic numbers $a,b\ge0$,
$$
\sum_{t=1}^T \E[X_t^2\mid\mathcal G_{t-1}] \le aS+b
\quad\text{almost surely}.
$$
Then, for every $\eta>0$ and $\delta\in(0,1)$, with probability at
least $1-\delta$,
$$
\sum_{t=1}^T X_t \le \eta S + \frac{3a}{4\eta}\log\frac{1}{\delta}+
\sqrt{3b\log\frac{1}{\delta}} + c\log\frac{1}{\delta}.
$$
\end{proposition}

\begin{proof}[Proof of \cref{prop:freedman-absorbable}]
Define $V_T:=\sum_{t=1}^T\E[X_t^2\mid\mathcal G_{t-1}].$
By \citet[Theorem~1]{beygelzimer2011contextual}, for every deterministic
$V'$ satisfying $V'\ge\frac{c^2}{e-2}\log\frac{1}{\delta},$
with probability at least $1-\delta$,
$$
\sum_{t=1}^T X_t \le \sqrt{(e-2)\log\frac{1}{\delta}}
\left(\frac{V_T}{\sqrt{V'}}+ \sqrt{V'}\right).
$$
We choose $V'$ such that
$$
\sqrt{V'}:= \frac{a}{\eta} \sqrt{(e-2)\log\frac{1}{\delta}}
+\sqrt{b}+\frac{c}{\sqrt{e-2}}\sqrt{\log\frac{1}{\delta}}.
$$
This choice satisfies the required lower bound on $V'$. Using
$V_T\le aS+b$, we obtain
$$
\begin{aligned}
\sum_{t=1}^T X_t
&\le
\sqrt{(e-2)\log\frac{1}{\delta}}
\left(
\frac{aS+b}{\sqrt{V'}}
+
\sqrt{V'}
\right)
\\
&\le
\eta S
+
\frac{(e-2)a}{\eta}\log\frac{1}{\delta}
+
2\sqrt{
(e-2)b\log\frac{1}{\delta}
}
+
c\log\frac{1}{\delta}.
\end{aligned}
$$
Finally, since $e-2\le\frac34$ and $
2\sqrt{e-2}\le\sqrt3,$
we conclude that
$$
\sum_{t=1}^T X_t \le
\eta S + \frac{3a}{4\eta}\log\frac{1}{\delta}
+\sqrt{3b\log\frac{1}{\delta}}+c\log\frac{1}{\delta}.
$$
\end{proof}

\begin{lemma}\label{lemma_grad_bound}
Assume that Assumption 2 holds. Then, for every $x \in \mathbb{R}^d$,
$$
\norm{\nabla f(x)} \le \frac{1}{\nu} \left(L_1\bigl(f(x)-f^\star\bigr)
+\sqrt{L_1^2\bigl(f(x)-f^\star\bigr)^2+2\nu L_0\bigl(f(x)-f^\star\bigr)}\right).
$$
In particular, if $f(x)-f^\star\le 2\Delta_0$, then
$$
\norm{\nabla f(x)} \le \frac{2}{\nu} \left(L_1\Delta_0 + \sqrt{L_1^2\Delta_0^2+\nu L_0\Delta_0}\right).
$$
\end{lemma}

\begin{proof}[Proof of Lemma \ref{lemma_grad_bound}]
By \cref{prop:self_bounding},
$$
\nu\norm{\nabla f(x)}^2 - 2L_1\bigl(f(x)-f^\star\bigr)\norm{\nabla f(x)} -
2L_0\bigl(f(x)-f^\star\bigr) \le 0.
$$
Since $\nu>0$, this quadratic inequality implies that
$\norm{\nabla f(x)}$ is bounded above by its larger root. Therefore,
$$
\norm{\nabla f(x)} \le \frac{1}{\nu} \left(L_1\bigl(f(x)-f^\star\bigr)+
\sqrt{L_1^2\bigl(f(x)-f^\star\bigr)^2+2\nu L_0\bigl(f(x)-f^\star\bigr)}\right).
$$
If $f(x)-f^\star\le 2\Delta_0$, we obtain
\begin{align*}
\norm{\nabla f(x)} &\le \frac{1}{\nu} \left(2L_1\Delta_0+\sqrt{4L_1^2\Delta_0^2
+4\nu L_0\Delta_0}\right)\\
&= \frac{2}{\nu} \left(L_1\Delta_0+\sqrt{L_1^2\Delta_0^2+\nu L_0\Delta_0}\right).
\end{align*}
\end{proof}

\begin{table}[ht]
\centering
\caption{Auxiliary notation used in the proofs.}
\label{tab:clipping-notation}
\renewcommand{\arraystretch}{1.25}
\begin{tabular}{c l l}
\toprule
Symbol & Definition & Description \\
\midrule
$\mathcal F_k$ & $\sigma(\xi_0,\ldots,\xi_{k-1})$ & History before the $k$-th oracle call\\
$G_k$ & $\nabla f(x_k,\xi_k)$ & Stochastic gradient \\
$q_k$ & $\clip(\nabla f(x_k),\lambda)$ & Clipped true gradient \\ 
$g_k$ & $\clip(G_k,\lambda)$ & Clipped stochastic gradient\\
$e_k$ & $g_k-q_k$ & Clipping error  \\
\bottomrule
\end{tabular}
\end{table}

Throughout the analysis, $x_0$ is deterministic and $\xi_k$ is
sampled independently of the past. Consequently, $x_k$ and $q_k$
are $\mathcal F_k$-measurable. Moreover, Assumption~\ref{asm: stochastic},
applied conditionally on $\mathcal F_k$ at $x_k$, gives
$$
\mathbb E[G_k\mid\mathcal F_k]
=
\nabla f(x_k),
\qquad
\mathbb E\left[
\norm{G_k-\nabla f(x_k)}^\alpha
\mid\mathcal F_k
\right]
\le
\sigma^\alpha.
$$

\section{High-Probability Guarantees for Convex Objectives}
\subsection{Best-Iterate Convergence}
\begin{table}[htbp]
\centering
\caption{Clipping level and stepsize for \cref{thm:best_iterate}.}
\label{tab:best-iterate-parameters}
\renewcommand{\arraystretch}{1.8}
\setlength{\tabcolsep}{10pt}
\begin{tabular}{c l}
\hline
\textbf{Parameter} & \textbf{Choice} \\
\hline
$\lambda$
& $\displaystyle
\max\left\{
\frac{L_0R_0}{\max\{L_1R_0,\log(4/\delta)\}},
\;\sigma\left(
\frac{K}{\max\{L_1R_0,\log(4/\delta)\}}
\right)^{1/\alpha}
\right\}$ \\[8pt]
$\gamma$
& $\displaystyle
\frac{\nu R_0}
{16\lambda\max\{L_1R_0,\log(4/\delta)\}}$ \\[8pt]
\hline
\end{tabular}
\end{table}
\begin{lemma}
\label{Lemmacle} Suppose that Assumptions~\ref{asm: convexity}
and \ref{asm: smoothness} hold.
For every $x\in\R^d$ and every $\lambda>0$, define
$
    q(x,\lambda):=\clip(\nabla f(x),\lambda).
$
Then
$$
    \norm{q(x,\lambda)}^2
    \le
    \frac{2\left(L_0+L_1\lambda\right)}{\nu}
    \ip{q(x,\lambda)}{x-x^\star}.
$$
Moreover, if $\norm{\nabla f(x)} \geq \lambda,$ then
$$
    \ip{q(x,\lambda)}{x-x^\star}
    \ge
    \frac{\nu\lambda^2}
    {2\left(L_0+L_1\lambda\right)}.
$$
\end{lemma}

\begin{lemma}
\label{lem:certificate} Suppose that Assumptions~\ref{asm: convexity}
and \ref{asm: smoothness} hold.
Let
$
    q_k:=\clip(\nabla f(x_k),\lambda),
$
where $\lambda>0$. Then, for every $k$,
$$
    \min\left\{
        f(x_k)-f^\star,
        \frac{\nu\lambda^2}
        {2\left(L_0+L_1\lambda\right)}
    \right\}
    \le
    \ip{q_k}{x_k-x^\star}.
$$
\end{lemma}

\begin{lemma}
\label{lem:onestep-budget}
Suppose that Assumptions~\ref{asm: convexity}
and \ref{asm: smoothness} hold.
Assume that
$
\gamma\le \frac{\nu}{8(L_0+L_1\lambda)}.
$
Then, for every $k$,
$$
\begin{aligned}
\frac{3\gamma}{2}
\min\left\{
f(x_k)-f^\star,\,
\frac{\nu\lambda^2}
{2(L_0+L_1\lambda)}
\right\}
&\le
\frac{3\gamma}{2}
\ip{q_k}{x_k-x^\star}
\\
&\le
\norm{x_k-x^\star}^2
-
\norm{x_{k+1}-x^\star}^2
\\
&\quad
-2\gamma
\ip{
e_k-\E[e_k\mid\mathcal F_k]
}{
x_k-x^\star
}
\\
&\quad
-2\gamma
\ip{
\E[e_k\mid\mathcal F_k]
}{
x_k-x^\star
}
+
2\gamma^2\norm{e_k}^2.
\end{aligned}
$$
\end{lemma}
For $k\ge0$, let
$
r_k:=\norm{x_k-x^\star},
$
and define the stopping time
$$
\tau
:=
\inf\left\{
k\ge0:\ r_k>2R_0
\right\},
$$
with the convention $\inf\emptyset=+\infty$. On the event $\{k<\tau\}$, we have $r_k\leq 2R_0$.
\begin{corollary}
\label{cor:stopped-onestep-budget}
Suppose that Assumptions~\ref{asm: convexity}
and \ref{asm: smoothness} hold.
Assume that
$
\gamma\le \frac{\nu}{8(L_0+L_1\lambda)}.
$
Then, for every $k$,
$$
\begin{aligned}
\frac{3\gamma}{2}\mathbf 1_{\{k<\tau\}}\min\left\{f(x_k)-f^\star,\frac{\nu\lambda^2}
{2(L_0+L_1\lambda)}\right\}&\le\frac{3\gamma}{2}\mathbf 1_{\{k<\tau\}}\ip{q_k}{x_k-x^\star}\\
&\le r_{k\wedge\tau}^2 -r_{(k+1)\wedge\tau}^2\\
&\quad -2\gamma\mathbf 1_{\{k<\tau\}}\ip{e_k-\E[e_k\mid\mathcal F_k]}{x_k-x^\star}\\
&\quad-2\gamma\mathbf 1_{\{k<\tau\}}\ip{\E[e_k\mid\mathcal F_k]}{x_k-x^\star}\\
&\quad+2\gamma^2\norm{e_k}^2.
\end{aligned}
$$
\end{corollary}

\begin{proof}
Multiplying Lemma~\ref{lem:onestep-budget} by $\mathbf 1_{\{k<\tau\}}$, we use
$
    \mathbf 1_{\{k<\tau\}}
    \left(r_k^2-r_{k+1}^2\right)
    =
    r_{k\wedge\tau}^2
    -
    r_{(k+1)\wedge\tau}^2$ and $
    2\gamma^2\mathbf 1_{\{k<\tau\}}\norm{e_k}^2
    \le
    2\gamma^2\norm{e_k}^2.$
\end{proof}

\begin{lemma}
\label{lem:quadratic-error-hp}
Under \cref{asm: stochastic}, for every $K\ge 1$ and every
$\delta\in(0,1)$, with probability at least $1-\delta$,
$$
    \sum_{k=0}^{K-1}\norm{e_k}^2
    \le
     2^{3-\alpha}K\sigma^\alpha\lambda^{2-\alpha}
    +
    5\lambda^2\log\frac{2}{\delta}.
$$
Consequently, with probability at least $1-\delta$,
$$
    2\gamma^2
    \sum_{k=0}^{K-1}\norm{e_k}^2
    \le
    2^{4-\alpha}K\gamma^2\sigma^\alpha\lambda^{2-\alpha}
    +
    10\gamma^2\lambda^2\log(2/\delta).
$$
\end{lemma}

\begin{lemma}
\label{lem:martingale-inner-product-hp}
Under Assumption~\ref{asm: stochastic}, for every $K\ge1$ and every
$\delta\in(0,1)$, with probability at least $1-\delta$,
$$
-2\gamma
\sum_{k=0}^{K-1}
\mathbf 1_{\{k<\tau\}}
\left\langle
    e_k-\E[e_k\mid\mathcal F_k],
    x_k-x^\star
\right\rangle
\le
8\gamma R_0
\sqrt{
    K\sigma^\alpha\lambda^{2-\alpha}
    \log\frac{2}{\delta}
}
+
\frac{32}{3}\gamma R_0\lambda
\log\frac{2}{\delta}.
$$
\end{lemma}

\begin{lemma}
\label{lem:directional-clipping-bias}
Let $x\in\R^d$ and define $G:=\nabla f(x,\xi)$, $q:=\clip(\nabla f(x),\lambda)$, $g:=\clip(G,\lambda)$, and $b:=\E[g]-q$. Suppose that \cref{asm: stochastic} holds. Then, for every $d\in\R^d$,
$$
\left|\ip{b}{d}\right|
\le
\frac{\sigma}{\lambda}\left|\ip{q}{d}\right|
+
\norm{d}\frac{\sigma^\alpha}{\lambda^{\alpha-1}}.
$$
\end{lemma}

\begin{remark}
If we assume that both Assumptions \ref{asm: convexity} and \ref{asm: smoothness}  hold, then
$
\ip{\nabla f(x)}{x-x^\star}
\ge
f(x)-f^\star
\ge
0.
$
Hence, \cref{lem:directional-clipping-bias} applies with
$
d=x-x^\star,
$
and yields
$$
-\ip{b}{x-x^\star}
\le
\frac{\sigma}{\lambda}
\ip{q}{x-x^\star}
+
\norm{x-x^\star}
\frac{\sigma^\alpha}{\lambda^{\alpha-1}}.
$$
\end{remark}

\begin{lemma}
\label{lem:bias-absorption}
Suppose that Assumptions~\ref{asm: convexity} and \ref{asm: stochastic} hold.
Assume $\lambda\ge 8\sigma.$ Then, for every $k$,
$$
\begin{aligned}
-2\gamma
\mathbf 1_{\{k<\tau\}}
\ip{\E[e_k\mid\mathcal F_k]}{x_k-x^\star}
\le{}&
4\gamma R_0
\frac{\sigma^\alpha}{\lambda^{\alpha-1}}+\frac{\gamma}{4}\mathbf 1_{\{k<\tau\}}\ip{q_k}{x_k-x^\star}.
\end{aligned}
$$
\end{lemma}

\begin{proposition}
\label{prop:stopped-hp-budget}
Suppose that Assumptions~\ref{asm: convexity},
\ref{asm: smoothness}, and \ref{asm: stochastic} hold.
Assume that
$
\lambda\ge 8\sigma
$
and
$
\gamma
\le
\frac{\nu}{8(L_0+L_1\lambda)}.
$
Then, for every $K\ge1$ and every $\delta\in(0,1)$, with probability at least
$1-\delta$,
$$
\begin{aligned}
r_{\min(K,\tau)}^2
+
\frac{5\gamma}{4}
\sum_{k=0}^{K-1}
\mathbf 1_{\{k<\tau\}}
\min\left\{
f(x_k)-f^\star,
\frac{\nu\lambda^2}
{2(L_0+L_1\lambda)}
\right\}
\le\;&
R_0^2
+
4\gamma K R_0
\frac{\sigma^\alpha}{\lambda^{\alpha-1}}
\\
&+
8\gamma R_0
\sqrt{
K\sigma^\alpha\lambda^{2-\alpha}
\log\frac{4}{\delta}
}
\\
&+
\frac{32}{3}\gamma R_0\lambda
\log\frac{4}{\delta}
\\
&+
2^{4-\alpha}K\gamma^2
\sigma^\alpha\lambda^{2-\alpha}
\\
&+
10\gamma^2\lambda^2
\log\frac{4}{\delta}.
\end{aligned}
$$
\end{proposition}

\begin{proof}
Summing \cref{cor:stopped-onestep-budget} from $k=0$ to $K-1$,
using $r_0=\norm{x_0-x^\star}\le R_0$,
and applying \cref{lem:bias-absorption}, the term
$
\frac{\gamma}{4}
\mathbf 1_{\{k<\tau\}}
\ip{q_k}{x_k-x^\star}
$
is absorbed into the
$
\frac{3\gamma}{2}
\mathbf 1_{\{k<\tau\}}
\ip{q_k}{x_k-x^\star}
$
progress term, leaving the coefficient
$
\frac{5\gamma}{4}.
$
Applying \cref{lem:martingale-inner-product-hp,lem:quadratic-error-hp}
with failure probability $\delta/2$ each, the claimed bound follows
by a union bound.
\end{proof}
\begin{proposition}
\label{prop:localization-hitting}
Suppose that Assumptions~\ref{asm: convexity},
\ref{asm: smoothness}, and \ref{asm: stochastic} hold.
Let $\delta\in(0,1)$. Assume that $\lambda\ge 8\sigma,$ that
$\gamma \le\frac{\nu}{8(L_0+L_1\lambda)},$
and that $\frac{\nu\lambda^2}{2(L_0+L_1\lambda)}\ge \epsilon.$
Let
$$
\begin{aligned}
\mathcal A_K :=\;& 4\gamma K R_0
\frac{\sigma^\alpha}{\lambda^{\alpha-1}}+ 8\gamma R_0 \sqrt{K\sigma^\alpha\lambda^{2-\alpha}
\log\frac{4}{\delta}}\\
&+\frac{32}{3}\gamma R_0\lambda\log\frac{4}{\delta}+2^{4-\alpha}K\gamma^2\sigma^\alpha\lambda^{2-\alpha}+10\gamma^2\lambda^2 \log\frac{4}{\delta}.
\end{aligned}
$$
Suppose moreover that $\mathcal A_K\le 2R_0^2$ and $K>\frac{12R_0^2}{5\gamma\epsilon}.$
Then, with probability at least $1-\delta$, we have 
$
\tau>K$ and $$
\min_{0\le k<K}
\bigl(f(x_k)-f^\star\bigr)
\le
\epsilon.
$$
\end{proposition}
\paragraph{Derivation of the tuning scales.} 
For a fixed iteration budget $K$, we seek to choose $\lambda$ and
$\gamma$ so that the accuracy bound $\epsilon$ certified by
\cref{prop:localization-hitting} is as small as possible. The three conditions used to derive the tuning envelope are $\gamma\le\frac{\nu}{8(L_0+L_1\lambda)},  A_K\le2R_0^2,$ and $K>\frac{12R_0^2}{5\gamma\epsilon}.$
The last condition is equivalent to
\begin{equation}
\label{eq:convex-accuracy-condition}
\epsilon>
\frac{12R_0^2}{5K\gamma}.
\end{equation}
For a given $\lambda$, this lower bound decreases as $\gamma$ increases and is therefore minimized by taking $\gamma$ as large as the first two conditions allow. We consequently derive upper bounds on $\gamma$ from the stepsize and localization conditions, substitute them into the preceding inequality, and then minimize the resulting bound over $\lambda$.

The proposition also imposes the feasibility conditions
$\lambda\ge8\sigma$ and $\epsilon\le\frac{\nu\lambda^2}{2(L_0+L_1\lambda)}$.
These conditions do not provide upper bounds on $\gamma$ and therefore
do not enter the preceding elimination of $\gamma$. We first obtain
candidate values of $\lambda$, $\gamma$, and $\epsilon$ from the three
tuning conditions above. In the subsequent theorem, we verify that
these candidates satisfy both feasibility conditions under the stated
lower bound on $K$.

We now derive the resulting lower bounds on $\epsilon$.

If $R_0=0$, the right-hand side of every lower bound derived below is
zero, so each bound follows from $\epsilon>0$. We may therefore assume
that $R_0>0$.

Combining \cref{eq:convex-accuracy-condition} with $\gamma\le\frac{\nu}{8(L_0+L_1\lambda)}$
yields
$$
\epsilon > \frac{R_0}{K} \frac{96}{5\nu} \left(L_0R_0+\lambda L_1R_0\right).
$$

Because all the terms defining $A_K$ are nonnegative, its first term
is bounded by $A_K$. Therefore, $ 4\gamma KR_0 \frac{\sigma^\alpha}{\lambda^{\alpha-1}}
\le A_K \le 2R_0^2.$ For $\sigma>0$, this gives $\gamma \le \frac{R_0\lambda^{\alpha-1}}{2K\sigma^\alpha},$ and hence
$$
\epsilon>\frac{R_0}{K} \frac{24K\sigma^\alpha}{5\lambda^{\alpha-1}}.
$$
When $\sigma=0$, the right-hand side of the preceding inequality
vanishes, so the same conclusion follows directly from
$\epsilon>0$. Thus, the inequality holds in both cases.

Similarly, the third term in $A_K$ satisfies $ \frac{32}{3}\gamma R_0\lambda \log\frac{4}{\delta} \le A_K \le 2R_0^2.$ Therefore, $\gamma\le\frac{3R_0}{16\lambda\log\frac{4}{\delta}},$
and  \cref{eq:convex-accuracy-condition} yields
$$
\epsilon > \frac{R_0}{K} \frac{64}{5}\lambda \log\frac{4}{\delta}.
$$
Combining the preceding three necessary inequalities, we obtain
$$
\epsilon > \frac{R_0}{K} \max\left\{ \frac{96}{5\nu} \left( L_0R_0+\lambda L_1R_0 \right),
\,\frac{64}{5}\lambda\log\frac{4}{\delta},\, \frac{24K\sigma^\alpha} {5\lambda^{\alpha-1}}\right\}.
$$
Similarly, requiring the second, fourth, and fifth terms of $A_K$
to be bounded by $2R_0^2$ gives additional upper bounds on $\gamma$.
After substituting these bounds into
\cref{eq:convex-accuracy-condition}, the resulting lower bounds on
$\epsilon$ are, up to numerical factors, dominated by
$$
\frac{R_0}{K}
\max\left\{
\frac{K\sigma^\alpha}{\lambda^{\alpha-1}},
\lambda\log\frac{4}{\delta}
\right\}.
$$
Indeed, $
\sqrt{K\sigma^\alpha\lambda^{2-\alpha}\log\frac{4}{\delta}}=
\sqrt{\frac{K\sigma^\alpha}{\lambda^{\alpha-1}}\cdot \lambda\log\frac{4}{\delta}
}\le \max\left\{ \frac{K\sigma^\alpha}{\lambda^{\alpha-1}},\lambda\log\frac{4}{\delta}\right\}.
$ Moreover, since $\alpha\in(1,2]$ and
$\log(4/\delta)>1$,
$
\sqrt{2^{3-\alpha}K\sigma^\alpha\lambda^{2-\alpha}} \le 2\max\left\{\frac{K\sigma^\alpha}{\lambda^{\alpha-1}},\lambda\log\frac{4}{\delta}\right\},
$ and $
\lambda\sqrt{\log\frac{4}{\delta}} \le \lambda\log\frac{4}{\delta}.$
Thus, these terms can change only the universal numerical constants
and introduce no additional tuning scale.

\begin{lemma}\label{lem:optimal-localization-tuning}
Let $\delta\in(0,1)$ and $K\ge1$. Suppose that
$\lambda,\gamma,\epsilon>0$ satisfy
$
\gamma\le\frac{\nu}{8(L_0+L_1\lambda)},\, A_K\le2R_0^2$ and  $K>\frac{12R_0^2}{5\gamma\epsilon},
$
where $A_K$ is defined in
\cref{prop:localization-hitting}. Then
$$
\epsilon>\frac{R_0}{K}\max\left\{\frac{96}{5\nu}\left(L_0R_0+\lambda L_1R_0\right),
\,\frac{64}{5}\lambda\log\frac{4}{\delta},\,\frac{24K\sigma^\alpha}{5\lambda^{\alpha-1}}\right\}.
$$
In particular, since $\nu\in(0,1]$, the preceding inequality implies
the simpler tuning bound
$$
\epsilon> \frac{R_0}{K}\max\left\{L_0R_0,\,\lambda \max\left\{L_1R_0,\log\frac{4}{\delta}\right\},
\,\frac{K\sigma^\alpha}{\lambda^{\alpha-1}}\right\}.
$$
Moreover,
$$
\begin{aligned}
&\inf_{\eta>0}\max\left\{L_0R_0,\,\eta\max\left\{L_1R_0,\log\frac{4}{\delta}\right\},
\,\frac{K\sigma^\alpha}{\eta^{\alpha-1}}\right\}\\
&\qquad=\max\left\{L_0R_0,\,\sigma K^{\frac{1}{\alpha}}\max\left\{L_1R_0,\log\frac{4}{\delta}\right\}^{\frac{\alpha-1}{\alpha}}\right\}.
\end{aligned}
$$
Provided that $L_0R_0+\sigma>0$, one minimizer is
$$
\lambda^\star=\max\left\{\frac{L_0R_0}{\max\left\{L_1R_0,\log\frac{4}{\delta}\right\}},\,
\sigma\left(\frac{K}{\max\left\{L_1R_0,\log\frac{4}{\delta}\right\}}\right)^{\frac{1}{\alpha}}\right\}.
$$
\end{lemma}

\paragraph{Choice of parameters.}
For brevity, let $M:=\max\left\{L_1R_0,\log\frac{4}{\delta}\right\}.$ In the nontrivial case $L_0 R_0+\sigma>0$, the minimizer identified by \cref{lem:optimal-localization-tuning} gives
$$
\lambda
=
\max\left\{
\frac{L_0R_0}{M},
\sigma\left(\frac{K}{M}\right)^{\frac{1}{\alpha}}
\right\}.
$$

This choice of $\lambda$ also determines the relevant stepsize scale.
Indeed, consider the two possible regimes in its definition.

If $\lambda=\frac{L_0R_0}{M},$ then the generalized-smoothness restriction gives
$$
\gamma \le \frac{\nu}{8(L_0+L_1\lambda)} \le\frac{\nu}{8L_0} =\frac{\nu R_0}{8\lambda M}.
$$
If instead $\lambda
=
\sigma\left(\frac{K}{M}\right)^{\frac{1}{\alpha}},$
then the first term of the localization condition
$A_K\le2R_0^2$ gives $4\gamma KR_0\frac{\sigma^\alpha}{\lambda^{\alpha-1}}\le2R_0^2.$
Consequently,
$$
\gamma
\le
\frac{R_0\lambda^{\alpha-1}}{2K\sigma^\alpha}
=
\frac{R_0}{2\lambda M},
$$
where the equality follows from
$\lambda^\alpha=K \frac{\sigma^\alpha}{M}$. Thus, in either regime, the
conditions of \cref{prop:localization-hitting} restrict the stepsize
to the scale $\frac{R_0}{\lambda M}.$

Since \cref{eq:convex-accuracy-condition} favors taking $\gamma$ as large as the
remaining conditions permit, we choose a uniform stepsize at this
scale. More precisely, we write
$$
\gamma=\frac{c\nu R_0}{\lambda M},
$$
where $c>0$ is a  constant. By the definitions of $\lambda$
and $M$, $L_0\le \frac{\lambda M}{R_0}$ and $L_1\lambda \le\frac{\lambda M}{R_0}.
$
Therefore, $L_0+L_1\lambda \le\frac{2\lambda M}{R_0}.$
It follows that
$$
\gamma\le \frac{\nu}{8(L_0+L_1\lambda)}
$$
is guaranteed whenever $c\le \frac{1}{16}$.

We next verify that the same choice also controls the localization
budget. By construction, $\frac{K\sigma^\alpha}{\lambda^\alpha}\le M,$ and the definition of $M$ gives $\log\frac{4}{\delta}\le M.$ Using these inequalities, together with $M>1$, in the five terms
defining $A_K$, respectively, yields
$$
A_K\le  \left(4c\nu+8c\nu+\frac{32}{3}c\nu+8c^2\nu^2+10c^2\nu^2\right)R_0^2=\left(\frac{68}{3}c\nu+18c^2\nu^2\right)R_0^2.
$$
We therefore choose $c=\frac{1}{16}$, that is,
$\gamma=\frac{\nu R_0}{16\lambda M}$. This choice satisfies the
generalized-smoothness restriction. Moreover, since $\nu< \frac{3}{5}$,
$$
\begin{aligned}
A_K &\le\left(\frac{17\nu}{12}+\frac{9\nu^2}{128}\right)R_0^2\\
&< \left(\frac{17}{20}+\frac{81}{3200}\right)R_0^2\\
&=\frac{2801}{3200}R_0^2\\
&< R_0^2.
\end{aligned}
$$
It remains to verify the two feasibility conditions of
\cref{prop:localization-hitting} that were not used to derive the
tuning scale, namely
$
\lambda\ge8\sigma$ and $\epsilon \le \frac{\nu\lambda^2}{2(L_0+L_1\lambda)}.
$
For the chosen stepsize, \cref{eq:convex-accuracy-condition} requires $\epsilon>\frac{192R_0\lambda M}{5\nu K}.$
We therefore set
$$
\epsilon:=\frac{40R_0\lambda M}{\nu K}.
$$
Since $L_0+L_1\lambda\le \frac{2\lambda M}{R_0}$, the second feasibility
condition is satisfied whenever $
K\ge\frac{160}{\nu^2}M^2.$

The same lower bound also implies $\frac{K}{M}\ge160>8^\alpha,$ and hence $\lambda\ge
\sigma\left(\frac{K}{M}\right)^{\frac{1}{\alpha}}
\ge8\sigma.$ Thus, all the conditions of
\cref{prop:localization-hitting} are satisfied.

\subsubsection{Proofs}
\begin{proof}[Proof of Lemma \ref{Lemmacle}]

\textbf{1)}
We distinguish two cases.

\begin{itemize}
    \item If $\norm{\nabla f(x)}\leq\lambda$, then
    $q(x,\lambda)=\nabla f(x)$. By \cref{prop:self_bounding} and convexity,
    $$
    \begin{aligned}
        \norm{q(x,\lambda)}^2
        &=
        \norm{\nabla f(x)}^2\\
        &\le
        \frac{2}{\nu}
        \left(
            L_0+L_1\norm{\nabla f(x)}
        \right)
        \left(
            f(x)-f^\star
        \right)\\
        &\le
        \frac{2(L_0+L_1\lambda)}{\nu}
        \left(
            f(x)-f^\star
        \right)\\
        &\le
        \frac{2(L_0+L_1\lambda)}{\nu}
        \ip{q(x,\lambda)}{x-x^\star}.
    \end{aligned}
    $$

    \item If $\norm{\nabla f(x)} \ge \lambda$, then
    $$
        q(x,\lambda)
        =
        \lambda
        \frac{\nabla f(x)}
        {\norm{\nabla f(x)}}.
    $$
    Hence, by convexity and \cref{prop:self_bounding},
    $$
    \begin{aligned}
        \ip{q(x,\lambda)}{x-x^\star}
        &=
        \lambda
        \frac{
            \ip{\nabla f(x)}{x-x^\star}
        }{
            \norm{\nabla f(x)}
        }\\
        &\ge
        \lambda
        \frac{
            f(x)-f^\star
        }{
            \norm{\nabla f(x)}
        }\\
        &\ge
        \frac{
            \nu\lambda\norm{\nabla f(x)}
        }{
            2\left(
                L_0+L_1\norm{\nabla f(x)}
            \right)
        }.
    \end{aligned}
    $$
    Since
    $
        t\mapsto \frac{t}{L_0+L_1t}
    $
    is nondecreasing on $(0,\infty)$ and
    $\norm{\nabla f(x)}\ge\lambda$, we obtain
    $$
        \ip{q(x,\lambda)}{x-x^\star}
        \ge
        \frac{\nu\lambda^2}
        {2(L_0+L_1\lambda)}.
    $$
    Since $\norm{q(x,\lambda)}=\lambda$, it follows that
    $$
        \norm{q(x,\lambda)}^2
        \le
        \frac{2(L_0+L_1\lambda)}{\nu}
        \ip{q(x,\lambda)}{x-x^\star}.
    $$
\end{itemize}

\textbf{2)} It is already proven in the first part.
\end{proof}

\begin{proof}[Proof of Lemma \ref{lem:certificate}]
We distinguish two cases.

 \begin{itemize}
     \item  If
$
    \norm{\nabla f(x_k)}\leq \lambda,
$
then $q_k=\nabla f(x_k)$. Hence, by convexity,
$$
\begin{aligned}
    \min\left\{
        f(x_k)-f^\star,
        \frac{\nu\lambda^2}
        {2(L_0+L_1\lambda)}
    \right\}
    &\le
    f(x_k)-f^\star\\
    &\le
    \ip{\nabla f(x_k)}{x_k-x^\star}\\
    &=
    \ip{q_k}{x_k-x^\star}.
\end{aligned}
$$
  \item 
If
$
    \norm{\nabla f(x_k)}>\lambda,
$
then, by Lemma~\ref{Lemmacle},
$$
\begin{aligned}
    \min\left\{
        f(x_k)-f^\star,
        \frac{\nu\lambda^2}
        {2(L_0+L_1\lambda)}
    \right\}
    &\le
    \frac{\nu\lambda^2}
    {2(L_0+L_1\lambda)}\\
    &\le
    \ip{q_k}{x_k-x^\star}.
\end{aligned}
$$
 \end{itemize}

\end{proof}

\begin{proof}[Proof of Lemma \ref{lem:onestep-budget}]
We have $x_{k+1}=x_k-\gamma g_k= x_k-\gamma(q_k+e_k).$ Therefore,
$$
\begin{aligned}
\norm{x_{k+1}-x^\star}^2
&=
\norm{x_k-x^\star-\gamma(q_k+e_k)}^2
\\
&=
\norm{x_k-x^\star}^2
-2\gamma\ip{q_k}{x_k-x^\star}
-2\gamma\ip{e_k}{x_k-x^\star}
+\gamma^2\norm{q_k+e_k}^2.
\end{aligned}
$$
Using $\norm{q_k+e_k}^2\le2\norm{q_k}^2+2\norm{e_k}^2,$ we obtain
$$
\begin{aligned}
\norm{x_{k+1}-x^\star}^2
&\le
\norm{x_k-x^\star}^2
-2\gamma\ip{q_k}{x_k-x^\star}
-2\gamma\ip{e_k}{x_k-x^\star}
\\
&\quad
+2\gamma^2\norm{q_k}^2
+2\gamma^2\norm{e_k}^2.
\end{aligned}
$$
By Lemma~\ref{Lemmacle}, $\norm{q_k}^2\le\frac{2(L_0+L_1\lambda)}{\nu}\ip{q_k}{x_k-x^\star}.$
Hence,
$$
2\gamma^2\norm{q_k}^2 \le \frac{4\gamma^2(L_0+L_1\lambda)}{\nu} \ip{q_k}{x_k-x^\star}.$$
Since $\gamma\le \frac{\nu}{8(L_0+L_1\lambda)},$ we have $\frac{4\gamma^2(L_0+L_1\lambda)}{\nu} \le\frac{\gamma}{2}.$ Therefore,
$$
\begin{aligned}
\norm{x_{k+1}-x^\star}^2
&\le
\norm{x_k-x^\star}^2
-\frac{3\gamma}{2}
\ip{q_k}{x_k-x^\star}
\\
&\quad
-2\gamma\ip{e_k}{x_k-x^\star}
+2\gamma^2\norm{e_k}^2.
\end{aligned}
$$
Rearranging gives
$$
\begin{aligned}
\frac{3\gamma}{2}
\ip{q_k}{x_k-x^\star}
&\le
\norm{x_k-x^\star}^2
-
\norm{x_{k+1}-x^\star}^2
\\
&\quad
-2\gamma\ip{e_k}{x_k-x^\star}
+
2\gamma^2\norm{e_k}^2.
\end{aligned}
$$
Decomposing
$
e_k
=
e_k-\E[e_k\mid\mathcal F_k]
+
\E[e_k\mid\mathcal F_k],
$
we obtain
$$
\begin{aligned}
\frac{3\gamma}{2}
\ip{q_k}{x_k-x^\star}
&\le
\norm{x_k-x^\star}^2
-
\norm{x_{k+1}-x^\star}^2
\\
&\quad
-2\gamma
\ip{
e_k-\E[e_k\mid\mathcal F_k]
}{
x_k-x^\star
}
\\
&\quad
-2\gamma
\ip{
\E[e_k\mid\mathcal F_k]
}{
x_k-x^\star
}
+
2\gamma^2\norm{e_k}^2.
\end{aligned}
$$
Finally, by Lemma~\ref{lem:certificate},
$$
\min\left\{
f(x_k)-f^\star,\,
\frac{\nu\lambda^2}
{2(L_0+L_1\lambda)}
\right\}
\le
\ip{q_k}{x_k-x^\star}.
$$
Therefore,
$$
\begin{aligned}
\frac{3\gamma}{2}
\min\left\{
f(x_k)-f^\star,\,
\frac{\nu\lambda^2}
{2(L_0+L_1\lambda)}
\right\}
&\le
\frac{3\gamma}{2}
\ip{q_k}{x_k-x^\star}
\\
&\le
\norm{x_k-x^\star}^2
-
\norm{x_{k+1}-x^\star}^2
\\
&\quad
-2\gamma
\ip{
e_k-\E[e_k\mid\mathcal F_k]
}{
x_k-x^\star
}
\\
&\quad
-2\gamma
\ip{
\E[e_k\mid\mathcal F_k]
}{
x_k-x^\star
}
+
2\gamma^2\norm{e_k}^2,
\end{aligned}
$$
\end{proof}

\begin{proof}[Proof of Lemma \ref{lem:quadratic-error-hp}]
Set $
    X_k:=\norm{e_k}^2.$
Since the clipping map is nonexpansive,
$$
    \norm{e_k}
    \le
    \norm{G_k-\nabla f(x_k)}.
$$
Moreover, both clipped vectors have norm at most $\lambda$, hence
$
    \norm{e_k}\le 2\lambda.$
Therefore,
$$
\begin{aligned}
    X_k
    =
    \norm{e_k}^2
    &=
    \norm{e_k}^{2-\alpha}\norm{e_k}^\alpha  \\
    &\le
    (2\lambda)^{2-\alpha}
    \norm{G_k-\nabla f(x_k)}^\alpha .
\end{aligned}
$$
Using Assumption~\ref{asm: stochastic}, we obtain
$$
\begin{aligned}
    \E[X_k\mid\mathcal F_k]
    &\le
    (2\lambda)^{2-\alpha}
    \E\!\left[
        \norm{G_k-\nabla f(x_k)}^\alpha
        \mid \mathcal F_k
    \right]  \\
    &\le
    2^{2-\alpha}\sigma^\alpha\lambda^{2-\alpha}.
\end{aligned}
$$
Thus
$$
    \sum_{k=0}^{K-1}\E[X_k\mid\mathcal F_k]
    \le
    2^{2-\alpha}\sigma^\alpha\lambda^{2-\alpha} K.
$$
We now use the decomposition
$$
    \sum_{k=0}^{K-1}X_k
    =
    \sum_{k=0}^{K-1}\E[X_k\mid\mathcal F_k]
    +
    \sum_{k=0}^{K-1}
    \Bigl(
        X_k-\E[X_k\mid\mathcal F_k]
    \Bigr).
$$
Let
$
    Z_k:=X_k-\E[X_k\mid\mathcal F_k].
$
Then $Z_k$ is $\mathcal F_{k+1}$-measurable and
$
    \E[Z_k\mid\mathcal F_k]=0.
$
Thus, after reindexing $Y_{k+1}:=Z_k$, $(Y_k)_{k\ge1}$
is a martingale difference sequence with respect to
$(\mathcal F_k)_{k\ge0}$. Also,
$$
    |Z_k|\le 4\lambda^2.
$$
Indeed, $0\le X_k\le 4\lambda^2$ and
$0\le \E[X_k\mid\mathcal F_k]\le 4\lambda^2$. We also have
$$
\begin{aligned}
    \E[Z_k^2\mid\mathcal F_k]
    &\le
    \E[X_k^2\mid\mathcal F_k]  \\
    &\le
    4 \lambda^2 \,\E[X_k\mid\mathcal F_k]  \\
    &\le
    2^{4-\alpha}\sigma^\alpha\lambda^{4-\alpha}.
\end{aligned}
$$
Hence 
$$
    \sum_{k=0}^{K-1}
    \E[Z_k^2\mid\mathcal F_k]
    \le
    2^{4-\alpha}\sigma^\alpha\lambda^{4-\alpha} K.
$$
By Corollary~\ref{cor:freedman-hp}, with probability at least
$1-\delta$,
$$
\begin{aligned}
    \sum_{k=0}^{K-1} Z_k
    &\le
    \sqrt{
        2\cdot
        2^{4-\alpha}\sigma^\alpha\lambda^{4-\alpha}K
        \log\frac{2}{\delta}
    }
    +
    \frac{8\lambda^2}{3}
    \log\frac{2}{\delta}.
\end{aligned}
$$
Therefore, on the same event,
$$
\begin{aligned}
    \sum_{k=0}^{K-1}X_k
    &\le
    2^{2-\alpha}\sigma^\alpha\lambda^{2-\alpha}K
    +
    \sqrt{
        2\cdot
        2^{4-\alpha}\sigma^\alpha\lambda^{4-\alpha}K
        \log\frac{2}{\delta}
    }
    +
    \frac{8\lambda^2}{3}
    \log\frac{2}{\delta}.
\end{aligned}
$$
Using $2\sqrt{ab}\le a+b$, with
$
    a:=2^{2-\alpha}\sigma^\alpha\lambda^{2-\alpha}K$ and $
    b:=2\lambda^2\log\frac{2}{\delta},$
we get
$$
    \sqrt{
        2\cdot
        2^{4-\alpha}\sigma^\alpha\lambda^{4-\alpha}K
        \log\frac{2}{\delta}
    }
    \le
    2^{2-\alpha}\sigma^\alpha\lambda^{2-\alpha}K
    +
    2\lambda^2\log\frac{2}{\delta}.
$$
Hence
$$
\begin{aligned}
    \sum_{k=0}^{K-1}X_k
    &\le
    2^{3-\alpha}\sigma^\alpha\lambda^{2-\alpha}K
    +
    \left(
        2+\frac{8}{3}
    \right)
    \lambda^2\log\frac{2}{\delta}  \\
    &\le
    2^{3-\alpha}\sigma^\alpha\lambda^{2-\alpha}K
    +
    5\lambda^2\log\frac{2}{\delta}.
\end{aligned}
$$
We deduce that
$$
\begin{aligned}
    2\gamma^2
    \sum_{k=0}^{K-1}\norm{e_k}^2
    &\le
    2^{4-\alpha}K\gamma^2\sigma^\alpha\lambda^{2-\alpha}
    +
    10\gamma^2\lambda^2\log\frac{2}{\delta}.
\end{aligned}
$$
\end{proof}

\begin{proof}[Proof of Lemma \ref{lem:martingale-inner-product-hp}]
Define
$$
    Z_k:=
    -2\gamma
    \mathbf 1_{\{k<\tau\}}
    \left\langle
        e_k-\E[e_k\mid\mathcal F_k],
        x_k-x^\star
    \right\rangle.
$$
Then $Z_k$ is $\mathcal F_{k+1}$-measurable and
$$
    \E[Z_k\mid\mathcal F_k]=0.
$$
Thus, after reindexing $Y_{k+1}:=Z_k$, $(Y_k)_{k\ge1}$
is a martingale difference sequence with respect to
$(\mathcal F_k)_{k\ge0}$.
We have
$$
    \|e_k-\E[e_k\mid\mathcal F_k]\|_2
    \le
    \|e_k\|_2
    +
    \|\E[e_k\mid\mathcal F_k]\|_2
    \le
    4\lambda.
$$
On the event $\{k<\tau\}$, we have
$
    \|x_k-x^\star\|_2\le 2R_0.
$
Therefore,
$$
    |Z_k|
    \le
    16\gamma R_0\lambda.
$$
Moreover,
$$
    \E[\|e_k-\E[e_k\mid\mathcal F_k]\|_2^2\mid\mathcal F_k]
    \le
    \E[\|e_k\|_2^2\mid\mathcal F_k].
$$
As in Lemma~\ref{lem:quadratic-error-hp},
$$
    \E[\|e_k\|_2^2\mid\mathcal F_k]
    \le
    2^{2-\alpha}\sigma^\alpha\lambda^{2-\alpha}.
$$
Thus
$$
\begin{aligned}
    \E[Z_k^2\mid\mathcal F_k]
    &\le
    4\gamma^2
    \mathbf 1_{\{k<\tau\}}
    \|x_k-x^\star\|_2^2
    \E[\| e_k-\E[e_k\mid\mathcal F_k]\|_2^2\mid\mathcal F_k] \\
    &\le
    16\gamma^2R_0^2
    \E[\|e_k\|_2^2\mid\mathcal F_k] \\
    &\le
    2^{6-\alpha}
    \gamma^2R_0^2
    \sigma^\alpha\lambda^{2-\alpha}.
\end{aligned}
$$
Consequently,
$$
    \sum_{k=0}^{K-1}
    \E[Z_k^2\mid\mathcal F_k]
    \le
    F,
$$
where $
    F
    :=
    2^{6-\alpha}
    \gamma^2R_0^2
    K\sigma^\alpha\lambda^{2-\alpha}.$
By Corollary~\ref{cor:freedman-hp}, with probability at least $1-\delta$,
$$
\begin{aligned}
    \sum_{k=0}^{K-1}Z_k
    \le\;&
    \sqrt{
        2F\log\frac{2}{\delta}
    }
    +
    \frac{2}{3}
    \left(16\gamma R_0\lambda\right)
    \log\frac{2}{\delta}.
\end{aligned}
$$
Since $\alpha\in(1,2]$,
$$
    \sqrt{2F\log\frac{2}{\delta}}
    \le
    8\gamma R_0
    \sqrt{
        K\sigma^\alpha\lambda^{2-\alpha}
        \log\frac{2}{\delta}
    }.
$$
Therefore, with probability at least $1-\delta$,
$$
\begin{aligned}
    -2\gamma
    \sum_{k=0}^{K-1}
    \mathbf 1_{\{k<\tau\}}
    \left\langle
        e_k-\E[e_k\mid\mathcal F_k],
        x_k-x^\star
    \right\rangle
    \le\;&
    8\gamma R_0
    \sqrt{
        K\sigma^\alpha\lambda^{2-\alpha}
        \log\frac{2}{\delta}
    } \\
    &+
    \frac{32}{3}
    \gamma R_0\lambda
    \log\frac{2}{\delta}.
\end{aligned}
$$
\end{proof}

\begin{proof}[Proof of Lemma \ref{lem:directional-clipping-bias}]
We have $\ip{b}{d}=\E\left\langle\left(\min\left\{1,\frac{\lambda}{\norm G}\right\}-\min\left\{1,\frac{\lambda}{\norm{\nabla f(x)}}\right\}\right)G,d\right\rangle.$ Hence,
$$
\ip{b}{d}=\E\left[\left(\min\left\{1,\frac{\lambda}{\norm G}\right\}-\min\left\{1,\frac{\lambda}{\norm{\nabla f(x)}}\right\}\right)\ip{G}{d}\right].
$$
This implies that
\begin{align*}
|\ip{b}{d}|\le&
\underbrace{\left|\E\left[\min\left\{1,\frac{\lambda}{\norm G}\right\}-\min\left\{1,\frac{\lambda}{\norm{\nabla f(x)}}\right\}\right]\right||\ip{\nabla f(x)}{d}|}_{\text{alignment term}}\\
&+\underbrace{\E\left[\left|\min\left\{1,\frac{\lambda}{\norm G}\right\}-\min\left\{1,\frac{\lambda}{\norm{\nabla f(x)}}\right\}\right||\ip{G-\nabla f(x)}{d}|\right]}_{\text{noise term}}.
\end{align*}

\textbf{1) Bounding the noise term.}  Let $\phi(r):=\min\{1,\frac{\lambda}{r}\}$, with $\phi(0)=1$. On $[0,\lambda]$, $\phi$ is $0$-Lipschitz. On $[\lambda,\infty)$, $\phi(r)=\frac{\lambda}{r}$ and $\forall r \ge \lambda , \,\, \left|\left(\phi\big|_{[\lambda,\infty)}\right)'(r)\right|= \frac{\lambda}{r^2}\le \frac{1}{\lambda}$, so $\phi$ is $\frac{1}{\lambda}$-Lipschitz on $[\lambda,\infty)$. If $x\le\lambda \le y$, then 
$$
|\phi(x)-\phi(y)|\le |\phi(x)-\phi(\lambda)|+|\phi(\lambda)-\phi(y)|\le \frac{|x-y|}{\lambda}.
$$
Hence, $\phi$ is $\frac{1}{\lambda}$-Lipschitz on $[0,\infty)$. Moreover, since $\phi$ takes values in $[0,1]$
\begin{align*}
\text{noise term}
&\le \norm d\,\E\left[\norm{G-\nabla f(x)}\min\left\{1,\frac{\norm{G-\nabla f(x)}}{\lambda}\right\}\right]\\
&\le \frac{\norm d}{\lambda^{\alpha-1}}\E\left[\norm{G-\nabla f(x)}^\alpha\right]\\
&\le \norm d\frac{\sigma^\alpha}{\lambda^{\alpha-1}},
\end{align*}
where in the second inequality we used $r\min\{1,\frac{r}{\lambda}\}\le \frac{r^\alpha}{\lambda^{\alpha-1}}$ for all $r\ge0$.

\textbf{2) Bounding the alignment term.} We now show that $\phi$ satisfies a stronger property than $\frac{1}{\lambda}$-Lipschitzness:
$$
|\phi(x)-\phi(y)|\le \phi(x)\frac{|x-y|}{\lambda}.
$$
Indeed, since $\phi(r)=\frac{\lambda}{\max\{\lambda,r\}}$,
\begin{align*}
|\phi(x)-\phi(y)|
&=\frac{\lambda|\max\{\lambda,x\}-\max\{\lambda,y\}|}{\max\{\lambda,x\}\max\{\lambda,y\}}\\
&\le \frac{|\max\{\lambda,x\}-\max\{\lambda,y\}|}{\max\{\lambda,x\}}\\
&\le \frac{|x-y|}{\max\{\lambda,x\}}\\
&=\phi(x)\frac{|x-y|}{\lambda},
\end{align*}
where we used that $r\mapsto\max\{\lambda,r\}$ is $1$-Lipschitz.
Therefore,
\begin{align*}
\text{alignment term}
&\le \frac{\E[\norm{G-\nabla f(x)}]}{\lambda}\min\left\{1,\frac{\lambda}{\norm{\nabla f(x)}}\right\}|\ip{\nabla f(x)}{d}|\\
&=\frac{\E[\norm{G-\nabla f(x)}]}{\lambda}|\ip{q}{d}|\\
&\le \frac{\sigma}{\lambda}|\ip{q}{d}|.
\end{align*}

\end{proof}

\begin{proof}[Proof of Lemma \ref{lem:bias-absorption}]
Let
$
b_k:=\E[e_k\mid\mathcal F_k].
$
By \cref{lem:directional-clipping-bias} with
$
d=x_k-x^\star,
$
and using convexity,
$
\ip{\nabla f(x_k)}{x_k-x^\star}\ge0,
$
we obtain
$$
-\ip{b_k}{x_k-x^\star}
\le
\frac{\sigma}{\lambda}
\ip{q_k}{x_k-x^\star}
+
\norm{x_k-x^\star}
\frac{\sigma^\alpha}{\lambda^{\alpha-1}}.
$$
Multiplying by
$
2\gamma\mathbf 1_{\{k<\tau\}}
$
gives
$$
\begin{aligned}
-2\gamma
\mathbf 1_{\{k<\tau\}}
\ip{b_k}{x_k-x^\star}
\le{}&
\frac{2\gamma\sigma}{\lambda}
\mathbf 1_{\{k<\tau\}}
\ip{q_k}{x_k-x^\star}
\\
&+
2\gamma
\mathbf 1_{\{k<\tau\}}
\norm{x_k-x^\star}
\frac{\sigma^\alpha}{\lambda^{\alpha-1}}.
\end{aligned}
$$
By the definition of $\tau$, on $\{k<\tau\}$,
$
\norm{x_k-x^\star}\le2R_0.
$
Therefore,
$$
\begin{aligned}
-2\gamma
\mathbf 1_{\{k<\tau\}}
\ip{b_k}{x_k-x^\star}
\le{}&
\frac{2\gamma\sigma}{\lambda}
\mathbf 1_{\{k<\tau\}}
\ip{q_k}{x_k-x^\star}
\\
&+
4\gamma R_0
\frac{\sigma^\alpha}{\lambda^{\alpha-1}}.
\end{aligned}
$$
Finally, since
$
\lambda\ge8\sigma,
$
we have
$
\frac{2\sigma}{\lambda}\le\frac14.
$
Hence,
$$
\begin{aligned}
-2\gamma
\mathbf 1_{\{k<\tau\}}
\ip{\E[e_k\mid\mathcal F_k]}{x_k-x^\star}
\le{}&
\frac{\gamma}{4}
\mathbf 1_{\{k<\tau\}}
\ip{q_k}{x_k-x^\star}
\\
&+
4\gamma R_0
\frac{\sigma^\alpha}{\lambda^{\alpha-1}},
\end{aligned}
$$
which proves the claim.
\end{proof}
\begin{proof}[Proof of Proposition \ref{prop:localization-hitting}]
We work on the event from \cref{prop:stopped-hp-budget}, which holds with
probability at least $1-\delta$. On this event,
$$
r_{\min(K,\tau)}^2+\frac{5\gamma}{4}\sum_{k=0}^{K-1}\mathbf 1_{\{k<\tau\}}\min\left\{
f(x_k)-f^\star,\frac{\nu\lambda^2}{2(L_0+L_1\lambda)}\right\}
\le R_0^2+\mathcal A_K.
$$
Since$\mathcal A_K\le 2R_0^2,$
we obtain
$$
r_{\min(K,\tau)}^2+\frac{5\gamma}{4}\sum_{k=0}^{K-1}\mathbf 1_{\{k<\tau\}}\min\left\{
f(x_k)-f^\star,\frac{\nu\lambda^2}{2(L_0+L_1\lambda)}\right\}
\le 3R_0^2.
$$
In particular, $r_{\min(K,\tau)}^2\le 3R_0^2.$ If $\tau\le K$, then by the definition of $\tau$,
$$
r_{\min(K,\tau)}^2=r_\tau^2>4R_0^2,
$$
which is a contradiction. Hence,
$
\tau>K.
$ Therefore,
$$
\frac{5\gamma}{4}
\sum_{k=0}^{K-1}
\min\left\{
f(x_k)-f^\star,
\frac{\nu\lambda^2}
{2(L_0+L_1\lambda)}
\right\}
\le
3R_0^2,
$$
and consequently,
$$
\frac{1}{K}
\sum_{k=0}^{K-1}
\min\left\{
f(x_k)-f^\star,
\frac{\nu\lambda^2}
{2(L_0+L_1\lambda)}
\right\}
\le
\frac{12R_0^2}{5\gamma K}.
$$
Since
$
K>
\frac{12R_0^2}{5\gamma\epsilon},
$
we have
$$
\frac{1}{K}
\sum_{k=0}^{K-1}
\min\left\{
f(x_k)-f^\star,
\frac{\nu\lambda^2}
{2(L_0+L_1\lambda)}
\right\}
<
\epsilon.
$$
Hence, there exists some $k\in\{0,\ldots,K-1\}$ such that
$$
\min\left\{
f(x_k)-f^\star,
\frac{\nu\lambda^2}
{2(L_0+L_1\lambda)}
\right\}
<
\epsilon.
$$
By assumption,
$
\frac{\nu\lambda^2}
{2(L_0+L_1\lambda)}
\ge
\epsilon.
$
Therefore, necessarily
$
f(x_k)-f^\star<\epsilon.
$
Thus,
$$
\min_{0\le k<K}
\bigl(f(x_k)-f^\star\bigr)
\le
\epsilon.
$$
\end{proof}

\begin{proof}[Proof of \cref{lem:optimal-localization-tuning}]
The first inequality was established in the derivation preceding the
lemma. Since $\nu\in(0,1]$,
$
\frac{96}{5\nu}>1, \frac{64}{5}>1, \frac{24}{5}>1,$ and
$
L_0R_0+\lambda L_1R_0\ge\max\left\{ L_0R_0,\lambda L_1R_0\right\}.
$
Therefore,
$$
\epsilon > \frac{R_0}{K} \max\left\{L_0R_0,\lambda\max\left\{ L_1R_0,\log\frac{4}{\delta}\right\},
\frac{K\sigma^\alpha}{\lambda^{\alpha-1}}\right\}.
$$

It remains to minimize the expression inside the maximum. Let
$$
M:=\max\left\{ L_1R_0,\log\frac{4}{\delta}\right\}.
$$
Since $\delta\in(0,1)$, we have $M>0$. Define $h:(0,\infty)\to\mathbb R$ by
$$
h(\eta):=\max\left\{ L_0R_0,\eta M,\frac{K\sigma^\alpha}{\eta^{\alpha-1}}\right\}.
$$

Suppose first that $\sigma>0$, and define the balancing point $\eta_{\mathrm{bal}}
:=\sigma \left( \frac{K}{M}\right)^{\frac{1}{\alpha}}.$ At this point,
$$
\eta_{\mathrm{bal}}M =\frac{K\sigma^\alpha} {\eta_{\mathrm{bal}}^{\alpha-1}}
=\sigma K^{\frac{1}{\alpha}}M^{\frac{\alpha-1}{\alpha}}.
$$
If $\eta\ge\eta_{\mathrm{bal}}$, then $\max\left\{\eta M,\frac{K\sigma^\alpha}{\eta^{\alpha-1}}
\right\}\ge\eta M \ge\eta_{\mathrm{bal}}M.$ If $\eta\le\eta_{\mathrm{bal}}$, then
$\max\left\{\eta M,\frac{K\sigma^\alpha}{\eta^{\alpha-1}}\right\}\ge \frac{K\sigma^\alpha}{\eta^{\alpha-1}} \ge\frac{K\sigma^\alpha} {\eta_{\mathrm{bal}}^{\alpha-1}}.
$
Consequently, for every $\eta>0$,
$$
h(\eta)
\ge
\max\left\{
L_0R_0,
\sigma K^{\frac{1}{\alpha}}
M^{\frac{\alpha-1}{\alpha}}
\right\}.
$$

Equality holds at $\eta=\eta_{\mathrm{bal}}$, because the two
$\eta$-dependent terms are equal there. Although
$\eta_{\mathrm{bal}}$ is already a minimizer when $\sigma>0$, we
introduce the following alternative minimizer because the same formula
also covers the deterministic case:
$$
\eta^\star := \max\left\{\frac{L_0R_0}{M},\sigma\left(\frac{K}{M}\right)^{\frac{1}{\alpha}}
\right\}.
$$

When $\sigma>0$, we have $
\eta^\star M
=
\max\left\{
L_0R_0,
\sigma K^{\frac{1}{\alpha}}
M^{\frac{\alpha-1}{\alpha}}
\right\}$ and $
\frac{K\sigma^\alpha}{(\eta^\star)^{\alpha-1}}
\le
\frac{K\sigma^\alpha}{\eta_{\mathrm{bal}}^{\alpha-1}}
=\sigma K^{\frac{1}{\alpha}} M^{\frac{\alpha-1}{\alpha}}.
$
Therefore,
$$
h(\eta^\star)
=
\max\left\{
L_0R_0,
\sigma K^{\frac{1}{\alpha}}
M^{\frac{\alpha-1}{\alpha}}
\right\},
$$
so $\eta^\star$ attains the preceding lower bound.

If $\sigma=0$, then
$$
h(\eta)=\max\left\{L_0R_0,\eta M\right\}.
$$
If $L_0R_0>0$, its infimum is attained at
$$
\eta^\star=\frac{L_0R_0}{M}.
$$
If $L_0R_0=0$, the infimum is zero but is not attained over
$\eta>0$. Hence, whenever $L_0R_0+\sigma>0$, $\eta^\star$ attains the stated
infimum. Substituting the definition of $M$ and setting
$\lambda^\star=\eta^\star$ gives the claimed minimizer.
\end{proof}

\begin{proof}[Proof of \cref{thm:best_iterate}]
If $R_0=0$, then $x_0=x^\star$, and the conclusion is immediate.
Hence, assume that $R_0>0$.

Let
$$
M
:=
\max\left\{
L_1R_0,
\log\frac{4}{\delta}
\right\},
\qquad
\epsilon
:=
\frac{40R_0\lambda M}{\nu K}.
$$
The conditions of \cref{prop:localization-hitting} have already been
verified for these choices of $\lambda$, $\gamma$, and $\epsilon$.
Therefore, with probability at least $1-\delta$,
$$
\min_{0\le k<K}
\left(f(x_k)-f^\star\right)
\le
\epsilon.
$$

By the definition of $\lambda$, $\lambda M=\max\left\{L_0R_0,\sigma K^{\frac{1}{\alpha}}M^{\frac{\alpha-1}{\alpha}}\right\}.$ Consequently,
$$
\epsilon=\frac{40}{\nu}\max\left\{\frac{L_0R_0^2}{K},\frac{R_0\sigma M^{\frac{\alpha-1}{\alpha}}}{K^{\frac{\alpha-1}{\alpha}}}\right\}.
$$
Substituting the definition of $M$ gives the claimed result.
\end{proof}

\subsection{A Computable Output}
\begin{table}[htbp]
\centering
\caption{Clipping level and stepsize for \cref{prop:computable-output}.}
\label{tab:computable-output-parameters}
\renewcommand{\arraystretch}{1.8}
\setlength{\tabcolsep}{10pt}
\begin{tabular}{c l}
\hline
\textbf{Parameter} & \textbf{Choice} \\
\hline
$\lambda$
& $\displaystyle
\max\left\{
\frac{L_0R_0}{\max\{L_1R_0,\log\left(\frac{12}{\delta}\right)\}},
\;\sigma\left(
\frac{K}{\max\{L_1R_0,\log\left(\frac{12}{\delta}\right)\}}
\right)^{1/\alpha}
\right\}$ \\[8pt]
$\gamma$
& $\displaystyle
\frac{\nu R_0}
{16\lambda\max\{L_1R_0,\log\left(\frac{12}{\delta}\right)\}}$ \\[8pt]
\hline
\end{tabular}
\end{table}
The cases $R_0=0$ and constant $f$ are immediate. We therefore consider
$R_0>0$ and $L_0>0$. We assume throughout this subsection that
$$
K\ge1024\left[\log^2\left(\frac{12}{\delta}\right) +L_1R_0\log\left\{\frac{12e}{\delta}\left(1+\frac{K\Delta_0}{L_0R_0^2}\right)\right\}\right].
$$
Choose
$$
\begin{cases}
\displaystyle \lambda=\max\left\{ \frac{L_0R_0}{\max\{L_1R_0,\log\left(\frac{12}{\delta}\right)\}},
\; \sigma\left(\frac{K}{\max\{L_1R_0,\log\left(\frac{12}{\delta}\right)\}}\right)^{1/\alpha}
\right\},\\[1.2em] 
\displaystyle \gamma= \frac{\nu R_0}{16\lambda\max\{L_1R_0,\log\left(\frac{12}{\delta}\right)\}}.
\end{cases}
$$

Writing $M:=\max\{L_1R_0,\log\left(\frac{12}{\delta}\right)\},$ the above choice implies
$$\begin{cases}
\displaystyle \frac{\sigma}{\lambda} \le \frac{1}{16},\\[0.8em]
\displaystyle 2\gamma(L_0+L_1\lambda) \le \frac{\nu}{4}.
\end{cases}$$
Indeed, $\lambda\ge \frac{L_0 R_0}{M}$ and $M\ge L_1R_0$ imply $
L_0+L_1\lambda\le \frac{2\lambda M}{R_0},$
and hence
$
2\gamma(L_0+L_1\lambda)\le \frac{\nu}{4}.
$
Moreover, the assumed lower bound on $K$ gives $K\ge1024M\log\left(\frac{12}{\delta}\right)\ge16^\alpha M.$
Consequently, $\left(\frac{\sigma}{\lambda}\right)^\alpha\le\frac{M}{K}$ and $
\frac{\sigma}{\lambda} \le\left(\frac{M}{K}\right)^{1/\alpha}\le\frac1{16}.$

The localization argument from the proof of \cref{thm:best_iterate},
with failure probability $\frac{\delta}{6}$ for each of its two concentration
bounds, gives an event $\mathcal E_1$ of probability at least
$1-\frac{\delta}{3}$ on which
$$
\begin{cases}
\tau>K,\\[0.4em]
\forall k \in \{0,\ldots,K\},\,\norm{x_k-x^\star}\le 2R_0,\\[0.4em]
\displaystyle
\sum_{k<K}\ip{q_k}{x_k-x^\star}
\le \frac{8R_0^2}{5\gamma}.
\end{cases}$$

\begin{lemma}
\label{lem:clipped-drift}
Suppose that Assumptions~\ref{asm: smoothness} and
\ref{asm: stochastic} hold.
Let $c_1,c_2>0$ satisfy $c_1^{-1}+2c_2<1$.
If $\lambda>0$ and $\gamma>0$ satisfy
$\lambda\ge c_1\sigma$ and
$\gamma(L_0+L_1\lambda)\le c_2$, then, for every iteration,
$$
\begin{aligned}
\E[f(x_{k+1})-f(x_k)\mid\mathcal F_k]
&\le
-\left(1-\frac1{c_1}-2c_2\right)
\gamma\ip{\nabla f(x_k)}{q_k}\\
&\quad+
\left(
4c_2+\frac{c_1}{2}(1-4c_2)^2
\right)
\gamma\sigma^\alpha\lambda^{2-\alpha}.
\end{aligned}
$$
\end{lemma}

\begin{lemma}
\label{lem:population-core}
Let
$\mathcal C=\{0\le k<K:\norm{\nabla f(x_k)}<5\lambda/6\}$.
There exists an event $\mathcal E_{\mathcal C}\subseteq\mathcal E_1$
with $\Pr(\mathcal E_{\mathcal C})\ge1-\delta/2$ such that, on
$\mathcal E_{\mathcal C}$,
$$
|\mathcal C|\ge\frac{3K}{4}.
$$
\end{lemma}

\begin{lemma}
\label{lem:observable-selection}
Let $\mathcal U=\{0\le k<K:\norm{G_k}\le\lambda\}$.
There exists an event $\mathcal E_2\subseteq\mathcal E_1$ with
$\Pr(\mathcal E_2)\ge1-\frac{2\delta}{3}$ such that, on $\mathcal E_2$,
$$
|\mathcal U|\ge\frac{2K}{3}.
$$
\end{lemma}

\begin{lemma}
\label{lem:selected-clean}
Let $\mathcal O:= \left\{k\in\mathcal U:\norm{G_k-\nabla f(x_k)}>\sigma\left(\frac{K}{\log(12/\delta)}
\right)^{1/\alpha}\right\}.$
There exists an event $\mathcal E_3\subseteq\mathcal E_2$ with
$\Pr(\mathcal E_3)\ge1-\delta$ such that, on $\mathcal E_3$, $
|\mathcal O|<4\log(12/\delta)$ and
$$
\sum_{k\in\mathcal U\setminus\mathcal O} \bigl(f(x_k)-f^\star\bigr)
\le \frac{16R_0^2}{5\gamma} + 13R_0\lambda\max\{L_1R_0,\log(12/\delta)\}.
$$
\end{lemma}

\begin{lemma}
\label{prop:selected-average}
Let
$
\bar x_K:=
\frac{1}{|\mathcal U\setminus\mathcal O|}
\sum_{k\in\mathcal U\setminus\mathcal O}x_k.$
Then, on $\mathcal E_3$,
$$
\begin{cases}
f(\bar x_K)-f^\star < \displaystyle\frac{97}{\nu}\frac{R_0\lambda M}{K},
\\[1ex]
\norm{\widehat x_K-\bar x_K}<\displaystyle\frac{24R_0\log(12/\delta)}{K}.
\end{cases}
$$
\end{lemma}

\subsubsection{Proofs}

\begin{proof}[Proof of \cref{lem:clipped-drift}]
Recall that $e_k=g_k-q_k$. Nonexpansiveness of clipping and
the oracle moment bound give
$
\E[\norm{e_k}\mid\mathcal F_k]\le\sigma$ and $\E[\norm{e_k}^2\mid\mathcal F_k]
\le2\sigma^\alpha\lambda^{2-\alpha}.$
We have
$$
L_1\norm{x_{k+1}-x_k}
\le\gamma L_1\lambda
\le\gamma(L_0+L_1\lambda)\le c_2 <1.
$$
Thus \cref{lem:local-smoothness-short-segment} gives
$$
f(x_{k+1})-f(x_k)
\le-\gamma\ip{\nabla f(x_k)}{g_k}
+2\gamma^2
\bigl(L_0+L_1\norm{\nabla f(x_k)}\bigr)\norm{g_k}^2.
$$
\begin{itemize}
\item Assume that $\norm{\nabla f(x_k)}\le\lambda$. Then $q_k=\nabla f(x_k)$.
Using $2\gamma(L_0+L_1\lambda)\le2c_2$, we obtain
$$
\begin{aligned}
f(x_{k+1})-f(x_k)&\le-\gamma\ip{\nabla f(x_k)}{g_k}+2\gamma c_2\norm{g_k}^2\\
&= -\gamma\left( \ip{q_k}{g_k}-2 c_2\norm{g_k}^2\right).
\end{aligned}
$$
Using $g_k=q_k+e_k$, we obtain
$$
\ip{q_k}{g_k}-2c_2\norm{g_k}^2
=(1-2c_2)\norm{q_k}^2
+(1-4c_2)\ip{q_k}{e_k}
-2c_2\norm{e_k}^2.
$$
By Young's inequality,  $ (1-4c_2)\ip{q_k}{e_k}\ge-\frac1{c_1}\norm{q_k}^2 -\frac{c_1}{4}(1-4c_2)^2\norm{e_k}^2.$ Then
$$
f(x_{k+1})-f(x_k) \le-\gamma\left(1-\frac1{c_1}-2c_2\right)\norm{q_k}^2+\gamma\left(
2c_2+\frac{c_1}{4}(1-4c_2)^2\right)\norm{e_k}^2.
$$
Taking conditional expectations and using
$\norm{q_k}^2=\ip{\nabla f(x_k)}{q_k}$, we obtain
$$
\begin{aligned}
\E[f(x_{k+1})-f(x_k)\mid\mathcal F_k]&\le-\gamma\left(1-\frac1{c_1}-2c_2\right)
\ip{\nabla f(x_k)}{q_k}\\
&\quad+\gamma\left(4c_2+\frac{c_1}{2}(1-4c_2)^2\right)\sigma^\alpha\lambda^{2-\alpha}.
\end{aligned}
$$

  \item  Assume that $\norm{\nabla f(x_k)}>\lambda$. Since $\norm{g_k}\le\lambda$,
$$
\begin{aligned}
\bigl(L_0+L_1\norm{\nabla f(x_k)}\bigr)\norm{g_k}^2
&\le\bigl(L_0+L_1\norm{\nabla f(x_k)}\bigr)\lambda^2\\
&\le(L_0+L_1\lambda)\lambda\norm{\nabla f(x_k)}\\
&=(L_0+L_1\lambda)\ip{\nabla f(x_k)}{q_k}.
\end{aligned}
$$
Moreover, $g_k=q_k+e_k$ and $\lambda\ge c_1\sigma$ imply
$$
\begin{aligned}
\E[\ip{\nabla f(x_k)}{g_k}\mid\mathcal F_k]
&\ge\ip{\nabla f(x_k)}{q_k}
-\sigma\norm{\nabla f(x_k)}\\
&=\left(1-\frac{\sigma}{\lambda}\right)
\ip{\nabla f(x_k)}{q_k}\\
&\ge\left(1-\frac1{c_1}\right)
\ip{\nabla f(x_k)}{q_k}.
\end{aligned}
$$
Substituting these two bounds into the descent inequality yields
$$
\begin{aligned}
\E[f(x_{k+1})-f(x_k)\mid\mathcal F_k]
&\le-\gamma\left(
1-\frac1{c_1}-2\gamma(L_0+L_1\lambda)
\right)\ip{\nabla f(x_k)}{q_k}\\
&\le-\gamma\left(1-\frac1{c_1}-2c_2\right)
\ip{\nabla f(x_k)}{q_k}.
\end{aligned}
$$
\end{itemize}
The second case also implies the claimed inequality, since its
noise term is nonnegative. 
\end{proof}

\begin{proof}[Proof of \cref{lem:population-core}]
We apply \cref{lem:clipped-drift} with $c_1=16$ and $c_2=3/32$.
Indeed, the parameter choices give $\lambda\ge16\sigma$ and
$\gamma(L_0+L_1\lambda)\le\nu/8<3/32$.
For these constants, the descent coefficient is $3/4$ and
the noise coefficient is $7/2\le4$. Moreover,
$$
4\gamma\sigma^\alpha\lambda^{2-\alpha}
=\frac{\nu R_0\lambda}{4M}
\left(\frac{\sigma}{\lambda}\right)^\alpha
\le\frac{\nu R_0\lambda}{4K}
\le\frac{R_0\lambda}{K}.
$$
Hence
$$
\E[f(x_{k+1})-f(x_k)\mid\mathcal F_k]
\le-\frac{3\gamma}{4}\ip{\nabla f(x_k)}{q_k}
+\frac{R_0\lambda}{K}.
$$

For $k<\tau$ with $k\notin\mathcal C$, convexity and $\norm{\nabla f(x_k)}\ge \frac{5\lambda}{6}$ give
$$
f(x_k)-f^\star+R_0\lambda \le\frac{16}{5}R_0\norm{\nabla f(x_k)} \le\frac{4R_0}{\lambda}\ip{\nabla f(x_k)}{q_k},
$$
where the last inequality uses $\norm{q_k}\ge \frac{5 \lambda}{6}$. Hence
$$
\ip{\nabla f(x_k)}{q_k} \ge \frac{\lambda}{4R_0} \bigl(f(x_k)-f^\star+R_0\lambda\bigr).
$$ 
Therefore, if  $k<\tau$ and $k\notin\mathcal C$, then
$$
\begin{aligned}
&\E[f(x_{k+1})-f^\star+R_0\lambda\mid\mathcal F_k]\\
&\le f(x_k)-f^\star+R_0\lambda-\frac{3\gamma\lambda}{16R_0}\bigl(f(x_k)-f^\star+R_0\lambda\bigr)
+\frac{R_0\lambda}{K}\\
&\le \left(1-\frac{3\gamma\lambda}{16R_0} +\frac1K\right) \bigl(f(x_k)-f^\star+R_0\lambda\bigr),
\end{aligned}
$$
For $k<\tau$ with $k\in\mathcal C$, using the same drift bound and $f(x_k)-f^\star+R_0\lambda\ge R_0\lambda$, we obtain
$$
\E[f(x_{k+1})-f^\star+R_0\lambda\mid\mathcal F_k] \le \left(1+\frac1K\right) \bigl(f(x_k)-f^\star+R_0\lambda\bigr).
$$
From the two previous cases, for every $k<\tau$,
$$
\begin{aligned}
\E[f(x_{k+1})-f^\star+R_0\lambda\mid\mathcal F_k]&\le \left(1-\frac{3\gamma\lambda}{16R_0}\,\mathbf 1_{\{k\notin\mathcal C\}}+\frac1K\right)\bigl(f(x_k)-f^\star+R_0\lambda\bigr)\\
&\le \exp\left( -\frac{3\gamma\lambda}{16R_0}\,\mathbf 1_{\{k\notin\mathcal C\}}+\frac1K\right)\bigl(f(x_k)-f^\star+R_0\lambda\bigr),
\end{aligned}
$$
where we used $1+t\le e^t$. On $\{k<\tau\}$, we have $k\wedge\tau=k$ and
$(k+1)\wedge\tau=k+1$. On $\{k\ge\tau\}$, both stopped indices equal $\tau$,
and the exponential multiplier below is $e^{-1/K}\le1$. Thus, for every $0\le k<K$,
$$
\E\left[\bigl(f(x_{(k+1)\wedge\tau})-f^\star+R_0\lambda\bigr)\exp\!\left\{\frac{3\gamma\lambda}{16R_0}\,\mathbf 1_{\{k<\tau,\;k\notin\mathcal C\}}-\frac1K\right\}\,\middle|\,\mathcal F_k\right]
\le f(x_{k\wedge\tau})-f^\star+R_0\lambda.
$$

Multiplying by the exponential factor accumulated over $j<k$,
which is $\mathcal F_k$-measurable, and taking expectations gives

$$
\begin{aligned}
&\E\left[\bigl(f(x_{(k+1)\wedge\tau})-f^\star+R_0\lambda\bigr)\exp\left\{\frac{3\gamma\lambda}{16R_0}\sum_{j\le k}\mathbf1_{\{j<\tau,\;j\notin\mathcal C\}}-\frac{k+1}{K}\right\}\right]\\
&\le \E\left[\bigl(f(x_{k\wedge\tau})-f^\star+R_0\lambda\bigr)\exp\left\{
\frac{3\gamma\lambda}{16R_0}\sum_{j<k}\mathbf1_{\{j<\tau,\;j\notin\mathcal C\}}-\frac{k}{K}
\right\}\right].
\end{aligned}
$$
Iterating for $k=0,\ldots,K-1$ yields
$$
\E\left[\bigl(f(x_{K\wedge\tau})-f^\star+R_0\lambda\bigr)\exp\left\{\frac{3\gamma\lambda}{16R_0}
\sum_{k<K}\mathbf1_{\{k<\tau,\;k\notin\mathcal C\}}-1\right\}\right]
\le \Delta_0+R_0\lambda.
$$
Since $f(x_{K\wedge\tau})\ge f^\star$, we obtain
$$
\E\left[\exp\left\{\frac{3\gamma\lambda}{16R_0}\sum_{k<K}\mathbf1_{\{k<\tau,\;k\notin\mathcal C\}}
\right\}\right]\le e\left(1+\frac{\Delta_0}{R_0\lambda}\right).$$
Therefore, by Markov's inequality, there exists an event $\widetilde{\mathcal E}_{\mathcal C}$ with
$\Pr(\widetilde{\mathcal E}_{\mathcal C})\ge1-\delta/6$ such that
$$
\sum_{k<K}\mathbf1_{\{k<\tau,\;k\notin\mathcal C\}} \le \frac{16R_0}{3\gamma\lambda}
\log\left\{ \frac{6e}{\delta}\left(1+\frac{\Delta_0}{R_0\lambda}\right)\right\}.
$$
Define $\mathcal E_{\mathcal C}:=\mathcal E_1\cap
\widetilde{\mathcal E}_{\mathcal C}$. Since $\Pr(\mathcal E_1)\ge1-\delta/3$, the union bound gives
$\Pr(\mathcal E_{\mathcal C})\ge1-\delta/2$.
On $\mathcal E_{\mathcal C}$, we have $\tau>K$. Substituting $\gamma$ gives
$$
K-|\mathcal C|
\le\frac{256M}{3\nu}
\log\left\{\frac{6e}{\delta}
\left(1+\frac{\Delta_0}{R_0\lambda}\right)\right\}.
$$
We now show that this is at most $\frac{K}{4}$ under the assumed budget.

If $L_1R_0>\log( \frac{12}{\delta})$, then $M=L_1R_0$ and $K\ge M$.
Since $R_0\lambda\ge \frac{L_0 R_0^2}{M} \ge \frac{L_0 R_0^2}{K}$,
$$
\begin{aligned}
K-|\mathcal C| &\le\frac{256L_1R_0}{3\nu} \log\left\{\frac{12e}{\delta}
\left(1+\frac{K\Delta_0}{L_0R_0^2}\right)\right\}\\
&\le\frac{K}{12\nu}<\frac K4.
\end{aligned}
$$

If $L_1R_0\le\log(\frac{12}{\delta})$, then $M=\log( \frac{12}{\delta})>2$.
Generalized smoothness and $\nabla f(x^\star)=0$ give
$\Delta_0\le \frac{L_0 R_0^2}{2} e^{L_1R_0}$ so $1+\frac{\Delta_0}{R_0\lambda}\le1+\frac M2e^{L_1R_0} \le M e^{L_1R_0}.$ Consequently, using $\log t\le t-1$ with $t= \frac{M}{2}$,
$$
\begin{aligned}
\log\left\{\frac{6e}{\delta}
\left(1+\frac{\Delta_0}{R_0\lambda}\right)\right\}
&\le   \log\left(\frac{6 e}{\delta}\right)+\log M+ L_1R_0\\
&=M-\log2+1  +\log M+ L_1R_0\\
&=L_1R_0+M+1+\log\left( \frac{M}{2}\right)\\
&\le L_1R_0+M+1+\left(\frac M2-1\right)\\
&=L_1R_0+\frac32M,
\end{aligned}
$$

Thus, on $\mathcal E_{\mathcal C},\, |\mathcal C|\ge\frac{3K}{4}$.
\end{proof}

\begin{proof}[Proof of \cref{lem:observable-selection}]
If $k\in\mathcal C\setminus\mathcal U$, then
$\norm{\nabla f(x_k)}<\frac{5 \lambda}{6}$ and $\norm{G_k}>\lambda$.
Then $\norm{G_k-\nabla f(x_k)}>\frac{\lambda}{6}.$ By Markov's inequality and the moment bound in \cref{asm: stochastic}, 
$$
\Pr(k\in\mathcal C\setminus\mathcal U\mid\mathcal F_k)
\le6^\alpha\left(\frac{\sigma}{\lambda}\right)^\alpha
\le\frac{36M}{K},
$$
where we used $\alpha\le2$. Define
$$
Z_k:=\mathbf1_{\{k\in\mathcal C\setminus\mathcal U\}}-\E\left[\mathbf1_{\{k\in\mathcal C\setminus\mathcal U\}}\mid\mathcal F_k\right].
$$
Then $\E[Z_k\mid\mathcal F_k]=0$, $|Z_k|\le1$, and
$
\E[Z_k^2\mid\mathcal F_k] \le \E\left[\mathbf1_{\{k\in\mathcal C\setminus\mathcal U\}}\mid\mathcal F_k \right] \le \frac{36M}{K}.$ Thus,
$$
\sum_{k<K}\E[Z_k^2\mid\mathcal F_k]\le36M.
$$
By \cref{cor:freedman-hp}, with probability at least $1-\delta/6$, 
$$
\sum_{k<K} Z_k
\le
\sqrt{72M\log\left(\frac{12}{\delta}\right)}
+\frac23\log\left(\frac{12}{\delta}\right).
$$
On this event,
$$
\begin{aligned}
|\mathcal C\setminus\mathcal U|
&=
\sum_{k<K}
\E\left[
\mathbf1_{\{k\in\mathcal C\setminus\mathcal U\}}
\mid\mathcal F_k
\right]
+\sum_{k<K}Z_k\\
&\le
36M
+\sqrt{72M\log\left(\frac{12}{\delta}\right)}
+\frac23\log\left(\frac{12}{\delta}\right)\\
&\le
\left(36+\sqrt{72}+\frac23\right)M
<46M,
\end{aligned}
$$
where we used $M\ge\log(12/\delta)$. Using $K \ge 1024 M$, we deduce $$
|\mathcal C\setminus\mathcal U|
<46M
\le\frac K{12}.
$$
Let $\tilde{\mathcal E}$ denote the event established above and define
$\mathcal E_2:=\mathcal E_{\mathcal C}\cap\tilde{\mathcal E}$.
Since $\Pr(\mathcal E_{\mathcal C})\ge1-\frac{\delta}{2}$ and
$\Pr(\tilde{\mathcal E})\ge1-\frac{\delta}{6}$, the union bound gives
$\Pr(\mathcal E_2)\ge1-\frac{2\delta}{3}$. On $\mathcal E_2$, we have
$$
|\mathcal U| \ge |\mathcal C\cap\mathcal U| = |\mathcal C|-|\mathcal C\setminus\mathcal U|
\ge \frac{3K}{4}-\frac K{12}=\frac{2K}{3}.
$$
\end{proof}

\begin{proof}[Proof of \cref{lem:selected-clean}]
Recall that
$M:=\max\{L_1R_0,\log(12/\delta)\}$.

If $\sigma=0$, then
$G_k=\nabla f(x_k)$ almost surely, so $\mathcal O=\varnothing$.
Moreover, for $k\in\mathcal U$ we have
$\norm{\nabla f(x_k)}=\norm{G_k}\le\lambda$.
Thus, on $\mathcal E_2\subseteq\mathcal E_1$, we have $
f(x_k)-f^\star \le \langle \nabla f(x_k),x_k-x^\star\rangle =\langle q_k,x_k-x^\star\rangle,$
and hence
$$
\sum_{k\in\mathcal U} \bigl(f(x_k)-f^\star\bigr) \le \sum_{k<K}\langle q_k,x_k-x^\star\rangle
\le \frac{8R_0^2}{5\gamma}.
$$
Therefore the claim holds with $\mathcal E_3=\mathcal E_2$.
Hence, assume $\sigma>0$.

We first bound $|\mathcal O|$.
By Markov's inequality and the moment bound in
\cref{asm: stochastic},
$$
\Pr(k\in\mathcal O\mid\mathcal F_k) \le \Pr\left( \norm{G_k-\nabla f(x_k)}>
\sigma\left(\frac{K}{\log(12/\delta)}\right)^{1/\alpha}\,\middle|\,\mathcal F_k
\right) \le \frac{\log(12/\delta)}{K}.
$$
Define $ Z_k:= \mathbf1_{\{k\in\mathcal O\}}-\E\left[\mathbf1_{\{k\in\mathcal O\}}\mid\mathcal F_k
\right].$ Then $\E[Z_k\mid\mathcal F_k]=0$, $|Z_k|\le1$, and
$$
\sum_{k<K}\E[Z_k^2\mid\mathcal F_k] \le \sum_{k<K} \Pr(k\in\mathcal O\mid\mathcal F_k)
\le \log(12/\delta).
$$
By \cref{cor:freedman-hp}, with probability at least $1-\delta/6$,
$$
\sum_{k<K}Z_k \le \sqrt{2}\log(12/\delta)+ \frac23\log(12/\delta).
$$
Therefore, on this event,
$$
|\mathcal O| \le \left( 1+\sqrt2+\frac23 \right)\log(12/\delta) < 4\log(12/\delta).
$$

We next control the selected iterates with
$\norm{\nabla f(x_k)}>2\lambda$.
Define
$$
Y_k:= \mathbf1_{\{k<\tau\}} \mathbf1_{\{k\in\mathcal U\setminus\mathcal O,\,
\norm{\nabla f(x_k)}>2\lambda\}}\bigl(f(x_k)-f^\star\bigr).
$$
For $k<\tau,$ by convexity
$$
f(x_k)-f^\star \le \langle\nabla f(x_k),x_k-x^\star\rangle \le 2R_0\norm{\nabla f(x_k)}.
$$
If $\norm{\nabla f(x_k)}>2\lambda$ and $k\in\mathcal U$, then $\norm{G_k-\nabla f(x_k)} \ge \norm{\nabla f(x_k)}-\lambda.$
Hence, by Markov's inequality and \cref{asm: stochastic},
$$
\Pr(k\in\mathcal U\mid\mathcal F_k) \le \frac{\sigma^\alpha}{\bigl(\norm{\nabla f(x_k)}-\lambda\bigr)^\alpha}.
$$
Therefore, on
$\{k<\tau,\norm{\nabla f(x_k)}>2\lambda\}$,
$$
\E[Y_k\mid\mathcal F_k] \le\frac{ 2R_0\norm{\nabla f(x_k)}\sigma^\alpha}{\bigl(\norm{\nabla f(x_k)}-\lambda\bigr)^\alpha} \le
\frac{4R_0\sigma^\alpha}{\bigl(\norm{\nabla f(x_k)}-\lambda\bigr)^{\alpha-1}}
\le \frac{4R_0\sigma^\alpha}{\lambda^{\alpha-1}},
$$
where we used
$\norm{\nabla f(x_k)} \le2(\norm{\nabla f(x_k)}-\lambda)$.
Outside this event, $Y_k=0$, so the same bound holds for every $k$.
Since $\left(\frac{\sigma}{\lambda}\right)^\alpha\le\frac{M}{K}$,
$$
\sum_{k<K}\E[Y_k\mid\mathcal F_k] \le \frac{4KR_0\sigma^\alpha}{\lambda^{\alpha-1}} 
\le 4R_0\lambda M.
$$

On $\{k\in\mathcal U\setminus\mathcal O\}$, $ \norm{\nabla f(x_k)} \le \norm{G_k}
+ \norm{G_k-\nabla f(x_k)} \le \lambda + \sigma\left(\frac{K}{\log(12/\delta)}\right)^{1/\alpha}.$
Thus
$$
0\le Y_k \le 2R_0 \left[ \lambda +\sigma\left(\frac{K}{\log(12/\delta)}\right)^{1/\alpha}\right].
$$
Since $\lambda\ge\sigma\left(\frac{K}{M}\right)^{1/\alpha}$ and
$M\ge\log\left(\frac{12}{\delta}\right)$,
\begin{align*}
2R_0 \left[ \lambda + \sigma\left(\frac{K}{\log(12/\delta)}\right)^{1/\alpha}\right]
\log(12/\delta) &\le 2R_0\lambda\left[\log(12/\delta)+M^{1/\alpha}\log^{1-1/\alpha}(12/\delta)
\right]\\
&\le 4R_0\lambda M.
\end{align*}
Moreover,
$$
\begin{aligned}
\sum_{k<K}\E\!\left[\left(Y_k-\E[Y_k\mid\mathcal F_k]\right)^2\mid\mathcal F_k\right]
&\le \sum_{k<K}\E[Y_k^2\mid\mathcal F_k] \\
&\le 2R_0\!\left[\lambda+\sigma\!\left(\frac{K}{\log(12/\delta)}\right)^{1/\alpha}\right]
\sum_{k<K}\E[Y_k\mid\mathcal F_k] \\
&\le 8R_0^2\lambda M\!\left[\lambda+\sigma\!\left(\frac{K}{\log(12/\delta)}\right)^{1/\alpha}\right] \\
&\le \frac{16(R_0\lambda M)^2}{\log(12/\delta)}.
\end{aligned}
$$
Applying \cref{cor:freedman-hp} and using the preceding bounds,
with probability at least $1-\delta/6$,
$$
\sum_{k<K} \left(Y_k-\E[Y_k\mid\mathcal F_k]\right)\le\sqrt{32}R_0\lambda M
+\frac83R_0\lambda M.
$$
Consequently, on the same event,
$$
\sum_{k<K}Y_k \le \left( 4+\sqrt{32}+\frac83\right)R_0\lambda M
<13R_0\lambda M.
$$

Let $\mathcal E_3$ be the intersection of $\mathcal E_2$ with the
two concentration events established above.
Since $\Pr(\mathcal E_2)\ge1-\frac{2\delta}{3}$, the union bound gives
$$
\Pr(\mathcal E_3)
\ge
1-\frac{2\delta}{3}
-\frac{\delta}{6}
-\frac{\delta}{6}
=
1-\delta.
$$
We now work on $\mathcal E_3$.
Since $\mathcal E_3\subseteq\mathcal E_2\subseteq\mathcal E_1$,
we have $\tau>K$ and
$$
\sum_{k<K}
\langle q_k,x_k-x^\star\rangle
\le
\frac{8R_0^2}{5\gamma}.
$$

For
$\norm{\nabla f(x_k)}\le2\lambda$,
the definition of $q_k$ and convexity give
$$
f(x_k)-f^\star \le \langle\nabla f(x_k),x_k-x^\star\rangle \le 2\langle q_k,x_k-x^\star\rangle.
$$
Then
$$
\sum_{\substack{ k\in\mathcal U\setminus\mathcal O\\ \norm{\nabla f(x_k)}\le2\lambda }}
\bigl(f(x_k)-f^\star\bigr) \le 2\sum_{k<K} \langle q_k,x_k-x^\star\rangle
\le \frac{16R_0^2}{5\gamma}.
$$
Since $\tau>K$, the definition of $Y_k$ gives
$$
\sum_{\substack{ k\in\mathcal U\setminus\mathcal O\\
\norm{\nabla f(x_k)}>2\lambda}}\bigl(f(x_k)-f^\star\bigr) =\sum_{k<K}Y_k < 13R_0\lambda M.
$$
We conclude that
$$
\sum_{k\in\mathcal U\setminus\mathcal O} \bigl(f(x_k)-f^\star\bigr) \le \frac{16R_0^2}{5\gamma} + 13R_0\lambda M.
$$

\end{proof}

\begin{proof}[Proof of \cref{prop:selected-average}]
We first show that $\bar x_K$ is well-defined on $\mathcal E_3$.
By \cref{lem:observable-selection} and \cref{lem:selected-clean},
on $\mathcal E_3,\, |\mathcal U|\ge\frac{2K}{3}$ and $
|\mathcal O|<4\log\left(\frac{12}{\delta}\right).$
Moreover, the assumed iteration requirement gives
$K\ge1024\log^2\left(\frac{12}{\delta}\right).$
Since $\log\left(\frac{12}{\delta}\right)>2$, it holds that
$4\log\left(\frac{12}{\delta}\right)<\frac{K}{512}.$
Therefore,
$$
|\mathcal U\setminus\mathcal O|=|\mathcal U|-|\mathcal O|>\frac{2K}{3}-\frac{K}{512}>\frac{85K}{128}.
$$

By convexity and \cref{lem:selected-clean},
$$
f(\bar x_K)-f^\star \le\frac{1}{|\mathcal U\setminus\mathcal O|}\sum_{k\in\mathcal U\setminus\mathcal O}\bigl(f(x_k)-f^\star\bigr)<\frac{128}{85K}\left(\frac{16R_0^2}{5\gamma}
+13R_0\lambda M\right).
$$
Using
$\gamma=\frac{\nu R_0}{16\lambda M}$, we have $\frac{16R_0^2}{5\gamma}=\frac{256}{5\nu}R_0\lambda M.$
Hence
$$
\begin{aligned}
f(\bar x_K)-f^\star&<\frac{128}{85}\left(\frac{256}{5\nu}+13\right)
\frac{R_0\lambda M}{K}\\
&<\frac{128}{85\nu}\left(\frac{256}{5}+13\right)\frac{R_0\lambda M}{K}
<\frac{97}{\nu}\frac{R_0\lambda M}{K},
\end{aligned}
$$
where we used $\nu<1$.

It remains to bound
$\norm{\widehat x_K-\bar x_K}$.
Since $\mathcal O\subseteq\mathcal U$,
$$
\widehat x_K =\frac{|\mathcal U\setminus\mathcal O|}{|\mathcal U|}\bar x_K
+\frac{1}{|\mathcal U|}\sum_{k\in\mathcal O}x_k,
$$
and therefore
$$
\widehat x_K-\bar x_K = \frac{1}{|\mathcal U|} \sum_{k\in\mathcal O}(x_k-\bar x_K).
$$
Since $\mathcal E_3\subseteq\mathcal E_1$,
$\norm{x_k-x^\star}\le2R_0$ for every $k\le K$.
Also, since $\bar x_K$ is an average of such iterates,
$$
\norm{\bar x_K-x^\star} \le \frac{1}{|\mathcal U\setminus\mathcal O|}\sum_{k\in\mathcal U\setminus\mathcal O}\norm{x_k-x^\star}\le2R_0.
$$
Thus, for every $k\in\mathcal O$,
$$
\norm{x_k-\bar x_K} \le\norm{x_k-x^\star}+\norm{\bar x_K-x^\star}\le 4R_0.
$$
Consequently,
$$
\norm{\widehat x_K-\bar x_K}\le\frac{4R_0|\mathcal O|}{|\mathcal U|}<\frac{16R_0\log\left(\frac{12}{\delta}\right)}{2K/3}=\frac{24R_0\log\left(\frac{12}{\delta}\right)}{K}.
$$
\end{proof}

\begin{proof}[Proof of \cref{prop:computable-output}]

Let $M:=\max\{L_1R_0,\log(12/\delta)\}.$ By \cref{prop:selected-average}, on $\mathcal E_3$,
$$
\begin{cases}
f(\bar x_K)-f^\star < \displaystyle\frac{97}{\nu}\frac{R_0\lambda M}{K},\\[1ex]
D_K <\displaystyle\frac{24R_0\log(12/\delta)}{K},
\end{cases}
$$
where $D_K:=\norm{\widehat x_K-\bar x_K}.$
Since $\Pr(\mathcal E_3)\ge1-\delta$, it remains to prove the claimed
bound on this event.

We have $
K\ge 1024M\log\left(\frac{12}{\delta}\right).$ Since $M\ge L_1R_0$, we obtain
$$
L_1D_K < \frac{24L_1R_0\log(12/\delta)}{K} \le\frac{24M\log(12/\delta)}{K}\le
\frac{3}{128}<\frac{1}{32}.
$$
Therefore, applying \cref{lem:local-smoothness-short-segment} with
$x=\bar x_K$ and $y=\widehat x_K$, we obtain
$$
\begin{aligned}
f(\widehat x_K)-f^\star&\le f(\bar x_K)-f^\star+\norm{\nabla f(\bar x_K)}D_K
+2\left(L_0+L_1\norm{\nabla f(\bar x_K)}\right)D_K^2\\
&=f(\bar x_K)-f^\star+(1+2L_1 D_K)\norm{\nabla f(\bar x_K)}D_K+2L_0 D_K^2\\
&\le f(\bar x_K)-f^\star + \frac{17}{16}\norm{\nabla f(\bar x_K)}D_K+2L_0 D_K^2.
\end{aligned}
$$

Let $\Delta:=f(\bar x_K)-f^\star.$ By \cref{lemma_grad_bound},
$\norm{\nabla f(\bar x_K)}\le\frac{1}{\nu}\left(L_1\Delta+\sqrt{L_1^2\Delta^2+2\nu L_0\Delta}\right).
$ Using $\sqrt{a+b}\le\sqrt a+\sqrt b$ and $\nu>1/2$, follows that
$$
\norm{\nabla f(\bar x_K)} \le \frac{2L_1}{\nu}\Delta + \sqrt{\frac{2}{\nu}}\sqrt{L_0\Delta}
\le 4L_1\Delta+2\sqrt{L_0\Delta}.$$
Therefore
$$
f(\widehat x_K)-f^\star \le \Delta+\frac{17}{4}L_1 D_K\,\Delta
+ \frac{17}{8}D_K\sqrt{L_0\Delta} +2L_0 D_K^2.
$$

Since $L_1 D_K<1/32$, it holds that  $\frac{17}{4}L_1 D_K\,\Delta\le \frac{17}{128}\Delta.$
By Young's inequality, we obtain
$$
\frac{17}{8}D_K \sqrt{L_0\Delta}\le\frac{15}{128}\Delta+\frac{289}{30}L_0D_K^2.
$$
Hence
$$
\begin{aligned}
f(\widehat x_K)-f^\star &\le \left(1+\frac{17}{128}+\frac{15}{128}\right)\Delta
+ \left(2+\frac{289}{30}\right)L_0D_K^2\\
&<\frac54\Delta+12L_0D_K^2.
\end{aligned}
$$
By \cref{prop:selected-average},
$$
\begin{cases}
\Delta <\displaystyle\frac{97}{\nu}\frac{R_0\lambda M}{K},\\[1ex]
12L_0D_K^2< \displaystyle\frac{6912L_0R_0^2\log^2(12/\delta)}{K^2}.
\end{cases}
$$
Since $K\ge1024\log^2(12/\delta),$ it follows that $12L_0D_K^2\le\frac{27}{4}\frac{L_0R_0^2}{K}.$
By the definition of $\lambda$, we have $\lambda\ge\frac{L_0R_0}{M},$ and hence
$\frac{L_0R_0^2}{K} \le \frac{R_0\lambda M}{K}.$ Therefore
$$
\begin{aligned}
f(\widehat x_K)-f^\star &<\frac54\frac{97}{\nu}\frac{R_0\lambda M}{K}+\frac{27}{4}\frac{R_0\lambda M}{K}\\
&\le\frac{1}{4\nu}(485+27)\frac{R_0\lambda M}{K}\\
&=\frac{128}{\nu}\frac{R_0\lambda M}{K},
\end{aligned}
$$
where we used $\nu<1$.

By the definition of $\lambda$, $$\frac{R_0\lambda M}{K}=\max\left\{\frac{L_0R_0^2}{K},\,
R_0\sigma\left(\frac{M}{K}\right)^{\frac{\alpha-1}{\alpha}}\right\}.
$$ Substituting $M=\max\{L_1R_0,\log(12/\delta)\},$ we obtain
$$
f(\widehat x_K)-f^\star \le \frac{128}{\nu} \max\left\{ \frac{L_0R_0^2}{K},
\,\frac{R_0\sigma\max\left\{L_1R_0,\log\frac{12}{\delta}\right\}^{\frac{\alpha-1}{\alpha}}}{
K^{\frac{\alpha-1}{\alpha}}}
\right\}.
$$
Since $\Pr(\mathcal E_3)\ge1-\delta$, the result follows.
\end{proof}

\subsection{Detailed comparison underlying Table~\ref{tab:comparison}}
\label{app:detailed-convex-comparison}

We provide here the technical details supporting the comparisons summarized
in Table~\ref{tab:comparison}. Throughout this subsection,
$\Delta_0=f(x_0)-f^\star$.

\paragraph{Reduction to classical smoothness.}
For classical $L$-smooth objectives, \citet{sadiev2023high} obtain the
standard heavy-tailed rate using a localization argument. Their assumptions
are imposed on a ball containing the iterates with high probability. When
this guarantee is applied to an $(L_0,L_1)$-smooth objective, the gradient is
classically Lipschitz on $B_r(x^\star)$ with constant of order
$L_0\exp(L_1r)$. Since the localization radius is of order $R_0$, such a
reduction can therefore introduce exponential dependence on $L_1R_0$.

\paragraph{Deterministic benchmarks and refined transients.}
For deterministic $(L_0,L_1)$-smooth optimization, the rate displayed for
\citet{gorbunov2024methods} in Table~\ref{tab:comparison} holds under the
sufficient requirement $K=\Omega((L_1R_0)^2)$. Subsequent refined analyses
relax this quadratic transient. In particular, for $L_0,R_0>0$,
\citet[Theorem~3.2]{vankov2024optimizing} and
\citet[Theorem~3.5]{lobanov2024linear} imply the same rate under
$K=\Omega\!\left(L_1R_0\log\left(\frac{K\Delta_0}{L_0R_0^2}\right)\right)$.
If $\Delta_0\le\frac{2L_0R_0^2}{K}$, the initial point is already at the
displayed $\mathcal O\!\left(\frac{L_0R_0^2}{K}\right)$ accuracy scale, so
the logarithmic transient is relevant only in the complementary regime. The
implicit condition above admits the convenient explicit sufficient form
$K=\Omega\!\left(L_1R_0\log\left(e+\frac{L_1\Delta_0}{L_0R_0}\right)\right)$.
Thus, the refined initial-gap-dependent condition can be substantially
smaller than the quadratic requirement.

\paragraph{A favorable instance.}
To illustrate this improvement, consider the convex function
$f(x)=e^x-x-1$, which is globally $(1,1)$-smooth since
$e^x\le 1+\lvert e^x-1\rvert$ for all $x\in\mathbb R$. Moreover, $L_1=1$
is sharp globally: for any $\widetilde L_1<1$ and finite
$\widetilde L_0$, the inequality
$f''(x)\le\widetilde L_0+\widetilde L_1\lvert f'(x)\rvert$ fails for
sufficiently large $x$. Taking $x_0=-R_0$ gives
$\Delta_0=e^{-R_0}+R_0-1=\Theta(R_0)$, so the explicit refined sufficient
budget scales as $\Theta(L_1R_0)$ for large $R_0$, compared with the
quadratic sufficient budget in the earlier analysis. On the other hand,
generalized smoothness and $\nabla f(x^\star)=0$ imply
$\Delta_0\le\frac{L_0R_0^2}{2}\exp(L_1R_0)$, and hence
$\frac{L_1\Delta_0}{L_0R_0}\le
\frac{L_1R_0}{2}\exp(L_1R_0)$. Therefore, the explicit refined condition
has at most $\mathcal O((L_1R_0)^2)$ transient scaling in the worst case,
while allowing substantially smaller transients on favorable instances.

\paragraph{Accelerated deterministic methods.}
Our comparison with \citet{vankov2024optimizing} concerns their gradient
method in Theorem~3.2. Their accelerated method in Theorem~6.2 achieves the
$\mathcal O\!\left(\frac{L_0R_0^2}{K^2}\right)$ rate under a weaker
transient requirement. Here $K$ counts outer iterations, and each iteration
requires a line search using function values, whose oracle cost is
additional. These function values are not provided by our
stochastic-gradient oracle, so this accelerated guarantee does not directly
apply in our setting.

\paragraph{Comparison with the light-tailed setting.}
The sub-Gaussian result of \citet{gaash2025convergence} also treats convex
$(L_0,L_1)$-smooth objectives and obtains the rate displayed in
Table~\ref{tab:comparison} under a quadratic sufficient transient. Their
analysis considers a double-sampling variant in which one stochastic
gradient determines the clipping test and an independent stochastic
gradient determines the update; the output averages iterates selected by
the clipping test. In contrast, our result applies directly to standard
Clip-SGD, which clips the same stochastic gradient used for the update, and
uses a computable average of its iterates. Moreover, our oracle assumption
only requires a bounded central $\alpha$-th moment, with $\alpha\in(1,2]$,
rather than sub-Gaussian noise.

\paragraph{Stochastic transient.}
Our result extends the refined transient picture above to stochastic
gradients with bounded central $\alpha$-th moment. For $L_0,R_0>0$ and any
$\sigma\ge0$, a sufficient iteration budget is
$K=\Omega\!\left(
\log^2\left(\frac{12}{\delta}\right)
+
L_1R_0
\log\left\{
\frac{12e}{\delta}
\left(
1+\frac{K\Delta_0}{L_0R_0^2}
\right)
\right\}
\right)$, while the convergence guarantee is given in
Table~\ref{tab:comparison}. Thus, rather than requiring a quadratic
transient uniformly, our stochastic condition retains the refined dependence
on the initial gap while allowing heavy-tailed noise.

\paragraph{Noiseless reduction.}
To make the deterministic comparison precise, fix the confidence parameter
used in the tuning to $1/2$. Since $\sigma=0$, the oracle, the iterates, and
the output are deterministic, so any positive-probability guarantee holds
with probability one. Consider first the nontrivial regime
$\frac{K\Delta_0}{L_0R_0^2}>2$. Indeed, if
$\frac{K\Delta_0}{L_0R_0^2}\le2$, then objective monotonicity and convexity
give $f(\widehat x_K)-f^\star\le\Delta_0\le
\frac{2L_0R_0^2}{K}$, so the desired deterministic rate already holds.

In the nontrivial regime,
$\log\left(1+\frac{K\Delta_0}{L_0R_0^2}\right)
\le
2\log\left(\frac{K\Delta_0}{L_0R_0^2}\right)$.
After fixing the confidence parameter, the general iteration requirement
becomes
$K=\Omega\!\left(
1+L_1R_0+
L_1R_0\log\left(
1+\frac{K\Delta_0}{L_0R_0^2}
\right)
\right)$.
If $L_1R_0\ge1$, the nontrivial regime also gives
$\log\left(\frac{K\Delta_0}{L_0R_0^2}\right)>\log 2$, and hence
$1+L_1R_0=
\mathcal O\!\left(
L_1R_0
\log\left(
\frac{K\Delta_0}{L_0R_0^2}
\right)
\right)$.
Together with the preceding logarithmic bound, this reduces the sufficient
iteration requirement to
$K=\Omega\!\left(
L_1R_0
\log\left(
\frac{K\Delta_0}{L_0R_0^2}
\right)
\right)$.

If $L_1R_0<1$, generalized smoothness gives
$\Delta_0\le\frac{L_0R_0^2}{2}\exp(L_1R_0)
\le\frac{e}{2}L_0R_0^2$.
After fixing the confidence parameter, the right-hand side of the original
iteration requirement is therefore bounded by
$C(1+\log(1+K))$ for a universal constant $C$. Hence, there exists a
universal constant $K_0$ such that the original requirement holds for every
$K\ge K_0$. For $1\le K<K_0$, objective monotonicity and convexity give
$f(\widehat x_K)-f^\star\le\frac{e}{2}L_0R_0^2
\le\frac{eK_0}{2}\frac{L_0R_0^2}{K}$.
Thus, when $\sigma=0$, our method achieves the
$\mathcal O\!\left(\frac{L_0R_0^2}{K}\right)$ rate under the same refined
sufficient iteration condition as the gradient methods analyzed in
\citet[Theorem~3.2]{vankov2024optimizing} and
\citet[Theorem~3.5]{lobanov2024linear}. When $L_1R_0<1$, the argument
additionally shows that this rate holds for every $K\ge1$, with a possibly
larger universal numerical constant.

\section{High-Probability Guarantees for Nonconvex Objectives}
\label{appendix_nonconvex_extension}

\subsection{Convergence Analysis}

\begin{lemma}
\label{lem:nonconvex-absorbable-clipping-bias}
Assume that Assumption~\ref{asm: stochastic} holds and that
$\lambda\ge8\sigma$. Then, for every $k\ge0$, almost surely,
$$
-\ip{\nabla f(x_k)}
{\mathbb E[e_k\mid\mathcal F_k]}
\le
\frac14
\norm{\nabla f(x_k)}
\min\left\{
\norm{\nabla f(x_k)},\lambda
\right\}
+
2\frac{\sigma^{2\alpha}}
{\lambda^{2\alpha-2}}.
$$
\end{lemma}

\begin{lemma}
\label{lem:nonconvex-joint-one-step}
Assume that Assumptions~\ref{asm: smoothness} and
\ref{asm: stochastic} hold. Suppose that $\lambda \ge 8 \sigma,\, \gamma L_0 \le \frac{1}{64}$ and $\gamma L_1 \lambda \le \frac{1}{64}.$ For every $k\ge0$, define $$\Phi_k:=-\gamma\ip{\nabla f(x_k)}{g_k}+2\gamma^2\left(L_0+L_1\norm{\nabla f(x_k)}\right) \norm{g_k}^2.
$$
Then, almost surely,
$$
f(x_{k+1})-f(x_k)\le\Phi_k,
$$
and
$$
\begin{aligned}
\mathbb E[\Phi_k\mid\mathcal F_k] \le&
-\frac{\gamma}{2}\norm{\nabla f(x_k)}\min\left\{\norm{\nabla f(x_k)},\lambda\right\}
+3\gamma\frac{\sigma^{2\alpha}}{\lambda^{2\alpha-2}}+8\gamma^2L_0\sigma^\alpha\lambda^{2-\alpha}.
\end{aligned}
$$
Moreover,
$$
\left\{
\begin{aligned}
\left|\Phi_k-\mathbb E[\Phi_k\mid\mathcal F_k]\right|&\le4\gamma\lambda\norm{\nabla f(x_k)}
+8\gamma^2L_0\lambda^2,\\[0.5em]
\operatorname{Var}\left(\Phi_k\mid\mathcal F_k\right)&\le64\sigma^\alpha\lambda^{2-\alpha}\left(
\gamma^2\norm{\nabla f(x_k)}^2+\gamma^4L_0^2\lambda^2\right).
\end{aligned}
\right.
$$
\end{lemma}

Let $\Delta_k:=f(x_k)-f^\star$ and define the stopping time
$$
\widetilde\tau := \inf\left\{k\ge0:\ \Delta_k>2\Delta_0\right\},
$$
with the convention $\inf\emptyset=+\infty$. On the event
$\{k<\widetilde\tau\}$, we have $\Delta_k\le2\Delta_0.$

Define
$$
\left\{
\begin{aligned}
H_0 &:=\frac{2}{\nu} \left(L_1\Delta_0+\sqrt{L_1^2\Delta_0^2+\nu L_0\Delta_0}\right),\\[0.6em]
\rho_0&:=\max\left\{1,\frac{H_0}{\lambda}\right\}.
\end{aligned}
\right.
$$
\begin{lemma}
\label{lem:nonconvex-stopped-martingale} Let $K \ge 1$.
Assume that Assumptions~\ref{asm: smoothness} and
\ref{asm: stochastic} hold. Suppose that $\lambda \ge 8 \sigma,\, \gamma L_0 \le \frac{1}{64}$ and $\gamma L_1 \lambda \le \frac{1}{64}.$ Define $S_K:= \sum_{k=0}^{K-1} \mathbf 1_{\{k<\widetilde\tau\}}
\norm{\nabla f(x_k)}\min\left\{\norm{\nabla f(x_k)},\lambda\right\}$
and

$$
Z_k:=\mathbf 1_{\{k<\widetilde\tau\}} \left(\Phi_k-\mathbb E[\Phi_k\mid\mathcal F_k]\right).
$$

Then $|Z_k| \le \frac{5}{2}  \gamma\lambda H_0+4\gamma^2L_0\lambda^2$ almost surely and,
$$
\sum_{k=0}^{K-1}\mathbb E[Z_k^2\mid\mathcal F_k] \le
64\sigma^\alpha\lambda^{2-\alpha}\left( \gamma^2\rho_0S_K
+K\gamma^4L_0^2\lambda^2\right).
$$
\end{lemma}

\begin{proposition}
\label{prop:nonconvex-stopped-descent} Let $K \ge 1$ and $\delta \in (0,1).$
Assume that Assumptions~\ref{asm: smoothness} and
\ref{asm: stochastic} hold.  Suppose that $\lambda \ge 8 \sigma,\, \gamma L_0 \le \frac{1}{64}$ and $\gamma L_1 \lambda \le \frac{1}{64}.$ Then, with probability
at least $1-\delta$,
$$
\begin{aligned}
\Delta_{K\wedge\widetilde\tau}+\frac{\gamma}{4}S_K \le&
\Delta_0+3\gamma K\frac{\sigma^{2\alpha}}{\lambda^{2\alpha-2}}+
\left(8+4\sqrt{3}\right)\gamma^2L_0K \sigma^\alpha\lambda^{2-\alpha}\\
&+\Bigg[\left(4+4\sqrt{3}\right)\gamma^2L_0\lambda^2+192\gamma\sigma^\alpha\lambda^{2-\alpha}\rho_0
+\frac{5}{2}\gamma\lambda H_0\Bigg]\log\frac{1}{\delta}.
\end{aligned}
$$
\end{proposition}

\begin{proposition}\label{prop:nonconvex-averaged-stationarity}
Let $K\ge1$, $\delta\in(0,1)$, and $\varepsilon>0$.
Assume that Assumptions~\ref{asm: smoothness} and
\ref{asm: stochastic} hold. Suppose that
$\lambda\ge\max\{8\sigma,\varepsilon\},\,\gamma L_0\le\frac1{64}$ and $\gamma L_1\lambda\le\frac1{64}.$
Let
$$
\begin{aligned}
\mathcal B_K:=\;&
3\gamma K\frac{\sigma^{2\alpha}}{\lambda^{2\alpha-2}}
+
\left(8+4\sqrt3\right)
\gamma^2L_0K\sigma^\alpha\lambda^{2-\alpha}\\
&+\Bigg[\left(4+4\sqrt3\right)\gamma^2L_0\lambda^2 +192\gamma\sigma^\alpha\lambda^{2-\alpha}\rho_0
+\frac52\gamma\lambda H_0\Bigg]\log\frac1\delta.
\end{aligned}
$$
Suppose moreover that $\mathcal B_K\le\Delta_0$ and $K>\frac{32\Delta_0}{\gamma\varepsilon^2}.$
Then, with probability at least $1-\delta$, we have
$\widetilde\tau>K$ and
$$
\frac1K\sum_{k=0}^{K-1}\norm{\nabla f(x_k)} \le \varepsilon.
$$
\end{proposition}

\paragraph{Derivation of the tuning scales.}  For a fixed iteration budget $K$, we seek to choose $\lambda$ and
$\gamma$ so that the accuracy $\varepsilon$ certified by
\cref{prop:nonconvex-averaged-stationarity} is as small as possible.
The stepsize is constrained by the smoothness and localization
conditions $\gamma L_0\le\frac1{64},\,\gamma L_1\lambda\le\frac1{64}$ and $\mathcal B_K\le\Delta_0.$
The proposition then certifies any accuracy $\varepsilon$ satisfying
$K>\frac{32\Delta_0}{\gamma\varepsilon^2},$ or equivalently
\begin{equation}
\varepsilon>
\sqrt{\frac{32\Delta_0}{K\gamma}}.
\label{eq:nonconvex-accuracy-condition}
\end{equation}

For a fixed $\lambda$, this lower bound decreases as $\gamma$
increases. We therefore derive upper bounds on $\gamma$ from the
three preceding conditions, substitute them into \cref{eq:nonconvex-accuracy-condition}, and then minimize the
resulting accuracy bound over $\lambda$.

The smoothness conditions give $\gamma\le\frac1{64L_0},\, \gamma\le\frac1{64L_1\lambda},$ and hence
$$
\varepsilon> 32\sqrt2 \max\left\{\sqrt{\frac{\Delta_0L_0}{K}},\sqrt{\frac{\Delta_0L_1\lambda}{K}} \right\}.
$$
The
condition $\mathcal B_K\le\Delta_0$ implies
$$
\begin{cases}
\displaystyle \gamma\le\frac{\Delta_0\lambda^{2\alpha-2}}{3K\sigma^{2\alpha}},\\[0.75em]
\displaystyle \gamma\le\sqrt{\frac{\Delta_0}{\left(8+4\sqrt3\right)L_0K\sigma^\alpha\lambda^{2-\alpha}}},\\[0.75em]
\displaystyle \gamma\le\sqrt{\frac{\Delta_0}{\left(4+4\sqrt3\right)L_0\lambda^2\log\frac1\delta}},\\[0.75em]
\displaystyle \gamma\le\frac{\Delta_0}{192\sigma^\alpha\lambda^{2-\alpha}\rho_0\log\frac1\delta},\\[0.75em]
\displaystyle \gamma\le\frac{2\Delta_0}{5\lambda H_0\log\frac1\delta}.
\end{cases}
$$
Substituting these bounds into \cref{eq:nonconvex-accuracy-condition} shows that, up to universal
numerical factors, 
$$
\begin{aligned}
\varepsilon\gtrsim\max\Bigg\{&
\frac{\sigma^\alpha}{\lambda^{\alpha-1}},
\left(\frac{\Delta_0L_0\sigma^\alpha\lambda^{2-\alpha}}{K}\right)^{\frac14},
\left(\frac{\Delta_0L_0\lambda^2\log\frac1\delta}{K^2}\right)^{\frac14},
\sqrt{\frac{\sigma^\alpha\lambda^{2-\alpha}\rho_0\log\frac1\delta}{K}},
\sqrt{\frac{\lambda H_0\log\frac1\delta}{K}}
\Bigg\}.
\end{aligned}
$$
We next remove the redundant scales. By the definition of $H_0$,
$
L_1\Delta_0\le\frac{\nu}{4}H_0$ and $L_0\Delta_0\le\frac{\nu}{4}H_0^2.$
Then
$$
\begin{cases}
\displaystyle
\sqrt{\frac{\Delta_0L_1\lambda}{K}}
\lesssim
\sqrt{\frac{\lambda H_0\log\frac{e}{\delta}}{K}},\\[0.75em]
\displaystyle
\left(\frac{\Delta_0L_0\lambda^2\log\frac1\delta}{K^2}\right)^{\frac14}
\lesssim
\sqrt{\frac{\lambda H_0\log\frac{e}{\delta}}{K}}.
\end{cases}
$$
It remains to simplify the term involving $\rho_0$. Since $\rho_0= \max\left\{1,\frac{H_0}{\lambda}\right\}\le 1+\frac{H_0}{\lambda}$ and $\lambda\ge8\sigma$, we have
$$
\begin{aligned}
\sigma^\alpha\lambda^{2-\alpha}\rho_0 &\le \sigma^\alpha\lambda^{2-\alpha}
+\sigma^\alpha H_0\lambda^{1-\alpha}\\
&=
\sigma^\alpha\lambda^{2-\alpha}+\lambda H_0\left(\frac{\sigma}{\lambda}\right)^\alpha\\
&\le \sigma^\alpha\lambda^{2-\alpha} +\frac{\lambda H_0}{8^\alpha}.
\end{aligned}
$$
Consequently,
$$
\sqrt{\frac{\sigma^\alpha\lambda^{2-\alpha}\rho_0\log\frac1\delta}{K}}\le
\sqrt{\frac{\sigma^\alpha\lambda^{2-\alpha}\log\frac{e}{\delta}}{K}}+\frac1{8^{\alpha/2}} \sqrt{\frac{\lambda H_0\log\frac{e}{\delta}}{K}}.
$$
The second term is already present in the accuracy bound. For the
first term, observe that $\frac{\sigma^\alpha\lambda^{2-\alpha}\log\frac{e}{\delta}}{K}
=\frac{\sigma^\alpha}{\lambda^{\alpha-1}} \frac{\lambda\log\frac{e}{\delta}}{K}.
$ If $
\frac{\sigma^\alpha}{\lambda^{\alpha-1}}\le H_0,$
then $
\sqrt{\frac{\sigma^\alpha\lambda^{2-\alpha}\log\frac{e}{\delta}}{K}}\le
\sqrt{\frac{\lambda H_0\log\frac{e}{\delta}}{K}}.
$
If instead $
\frac{\sigma^\alpha}{\lambda^{\alpha-1}}>H_0,$
then, on the event of \cref{prop:nonconvex-stopped-descent}, the
condition $\mathcal B_K\le\Delta_0$ implies
$\widetilde\tau>K$ by the stopping-time argument. Consequently, $
\frac1K\sum_{k=0}^{K-1}\norm{\nabla f(x_k)}
\le H_0<\frac{\sigma^\alpha}{\lambda^{\alpha-1}}.$
Thus, in either case, the term involving $\rho_0$ introduces no
additional tuning scale.

Combining the preceding estimates, it suffices, up to universal
numerical factors, to consider
\begin{equation}
\varepsilon
\gtrsim
\max\left\{
\sqrt{\frac{\Delta_0L_0}{K}},
\frac{\sigma^\alpha}{\lambda^{\alpha-1}},
\sqrt{\frac{\lambda H_0\log\frac{e}{\delta}}{K}},
\left(
\frac{\Delta_0L_0\sigma^\alpha\lambda^{2-\alpha}}{K}
\right)^{\frac{1}{4}}
\right\}.
\label{eq:nonconvex-tuning-envelope}
\end{equation}

Before optimizing over $\lambda$, we incorporate the condition
$\lambda\ge\max\{8\sigma,\varepsilon\}$ required by
\cref{prop:nonconvex-averaged-stationarity}. In view of
\cref{eq:nonconvex-tuning-envelope}, the condition
$\varepsilon\le\lambda$ is ensured if every term in the maximum is at
most $\lambda$.

The first term requires $\lambda\ge\sqrt{\frac{\Delta_0L_0}{K}}.$
The second is automatically controlled by $\lambda\ge8\sigma$, since
$\frac{\sigma^\alpha}{\lambda^{\alpha-1}}=\lambda\left(\frac{\sigma}{\lambda}\right)^\alpha
\le\frac{\lambda}{8^\alpha}.$
The third term is at most $\lambda$ whenever $\lambda\ge\frac{H_0\log\frac{e}{\delta}}{K}.$
Finally, the fourth term introduces no additional condition. Indeed,
if $\lambda\ge\sqrt{\frac{\Delta_0L_0}{K}}$ and $\lambda\ge 8\sigma,$
then $\frac{\Delta_0L_0\sigma^\alpha\lambda^{2-\alpha}}{K}\le\lambda^2\left(\frac{\lambda}{8}\right)^\alpha\lambda^{2-\alpha} =\frac{\lambda^4}{8^\alpha},$
and hence $\left( \frac{\Delta_0L_0\sigma^\alpha\lambda^{2-\alpha}}{K}\right)^{\frac{1}{4}}
\le\frac{\lambda}{8^{\alpha/4}}.$

We are therefore led to the feasible range
$$
\lambda\ge\underline{\lambda}_K,
\qquad
\underline{\lambda}_K
:=
\max\left\{
8\sigma,
\sqrt{\frac{\Delta_0L_0}{K}},
\frac{H_0\log\frac{e}{\delta}}{K}
\right\}.
$$

\begin{lemma}\label{lem:nonconvex-tuning}
Let $K\ge1$, $\delta\in(0,1)$, and $\alpha\in(1,2]$, and suppose that
$\Delta_0>0$. Define

$$
\begin{cases}
\displaystyle
\lambda_K
:=
\max\left\{
\underline{\lambda}_K,
\min\left\{
\left(
\frac{K\sigma^{2\alpha}}
{H_0\log\frac{e}{\delta}}
\right)^{\frac1{2\alpha-1}},
\left(
\frac{K\sigma^{3\alpha}}
{\Delta_0L_0}
\right)^{\frac1{3\alpha-2}}
\right\}
\right\},
\\[1.2em]
\displaystyle
\varepsilon_K
:=
\max\left\{
\sqrt{\frac{\Delta_0L_0}{K}},
\frac{H_0\log\frac{e}{\delta}}{K},
\left(
\frac{H_0^{\alpha-1}\sigma^\alpha
\log^{\alpha-1}\frac{e}{\delta}}
{K^{\alpha-1}}
\right)^{\frac1{2\alpha-1}},
\left(
\frac{(\Delta_0L_0)^{\alpha-1}\sigma^\alpha}
{K^{\alpha-1}}
\right)^{\frac1{3\alpha-2}}
\right\}.
\end{cases}
$$

When $L_0=0$, the quantity $
\left(
\frac{K\sigma^{3\alpha}}
{\Delta_0L_0}
\right)^{\frac1{3\alpha-2}}$ is interpreted as $+\infty$.

Then

$$
\begin{aligned}
\varepsilon_K
&\le
\inf_{\lambda\ge\underline{\lambda}_K}
\max\left\{
\sqrt{\frac{\Delta_0L_0}{K}},
\frac{\sigma^\alpha}{\lambda^{\alpha-1}},
\sqrt{\frac{\lambda H_0\log\frac{e}{\delta}}{K}},
\left(
\frac{\Delta_0L_0\sigma^\alpha\lambda^{2-\alpha}}{K}
\right)^{\frac14}
\right\}
\\
&=
\max\left\{
\sqrt{\frac{\Delta_0L_0}{K}},
\frac{\sigma^\alpha}{\lambda_K^{\alpha-1}},
\sqrt{\frac{\lambda_K H_0\log\frac{e}{\delta}}{K}},
\left(
\frac{\Delta_0L_0\sigma^\alpha\lambda_K^{2-\alpha}}{K}
\right)^{\frac14}
\right\}
\\
&\le
2\sqrt2\,\varepsilon_K.
\end{aligned}
$$
In particular, $\lambda_K$ is an exact minimizer.
\end{lemma}

\paragraph{Choice of parameters.}
By \cref{lem:nonconvex-tuning}, we choose $\lambda:=\lambda_K$ and denote
$$
E_K:=\max\left\{\sqrt{\frac{\Delta_0L_0}{K}},\frac{\sigma^\alpha}{\lambda_K^{\alpha-1}},
\sqrt{\frac{\lambda_KH_0\log\frac{e}{\delta}}{K}},\left(\frac{\Delta_0L_0\sigma^\alpha\lambda_K^{2-\alpha}}{K}\right)^{\frac14}\right\}.
$$
Then
$$
\varepsilon_K \le E_K \le 2\sqrt2\,\varepsilon_K.
$$
By \cref{eq:nonconvex-accuracy-condition}, the relevant stepsize scale is
$\gamma \asymp \frac{\Delta_0}{KE_K^2}.$
To simultaneously control the term involving $\rho_0$ in the
localization budget, we set
$$
\gamma
:=
c\min\left\{
\frac{\Delta_0}{KE_K^2},
\frac{\Delta_0}
{\sigma^\alpha\lambda_K^{2-\alpha}\rho_0
\log\frac{e}{\delta}}
\right\},
$$
where the second quantity is interpreted as $+\infty$ when $\sigma=0$,
and $c>0$ is a universal constant to be fixed below.

Since $\gamma\le c\frac{\Delta_0}{KE_K^2}$and $E_K^2\ge \frac{\Delta_0L_0}{K},$
we have
$$
\gamma L_0
\le
c\frac{\Delta_0L_0}{KE_K^2}
\le
c.
$$

Similarly, since $E_K^2 \ge \frac{\lambda_KH_0\log\frac{e}{\delta}}{K},$ and
$L_1\Delta_0 \le \frac{\nu}{4}H_0,$
we obtain
$$
\gamma L_1\lambda_K \le c\frac{\Delta_0L_1\lambda_K}{KE_K^2} \le
c\frac{L_1\Delta_0}{H_0\log\frac{e}{\delta}} \le c\frac{L_1\Delta_0}{H_0}
\le \frac{c\nu}{4}.
$$

We next verify the localization condition. Since $E_K^2 \ge
\frac{\sigma^{2\alpha}}{\lambda_K^{2\alpha-2}},$
we have
$$
3\gamma K\frac{\sigma^{2\alpha}}{\lambda_K^{2\alpha-2}} \le 3c\Delta_0.
$$
Since $E_K^4 \ge\frac{\Delta_0L_0\sigma^\alpha\lambda_K^{2-\alpha}}{K},$ we have
$$
(8+4\sqrt3)\gamma^2L_0K\sigma^\alpha\lambda_K^{2-\alpha} \le(8+4\sqrt3)c^2
\frac{\Delta_0^2L_0\sigma^\alpha\lambda_K^{2-\alpha}}{KE_K^4}
\le (8+4\sqrt3)c^2\Delta_0.
$$
Since $E_K^4\ge \frac{\lambda_K^2H_0^2 \log^2\frac{e}{\delta}}{K^2},$ and
$
L_0\Delta_0\le\frac{\nu}{4}H_0^2,
$ we have
$$
(4+4\sqrt3)\gamma^2L_0\lambda_K^2\log\frac1\delta \le (4+4\sqrt3)c^2
\frac{\Delta_0^2L_0\lambda_K^2\log\frac1\delta}{K^2E_K^4} \le\frac{\nu}{4}(4+4\sqrt3)c^2\Delta_0.
$$
By the definition of $\gamma$,
$$
192\gamma\sigma^\alpha\lambda_K^{2-\alpha}\rho_0
\log\frac1\delta \le 192c\Delta_0.
$$

Since $
E_K^2 \ge \frac{\lambda_KH_0\log\frac{e}{\delta}}{K},$
we have
$$
\frac52\gamma\lambda_KH_0\log\frac1\delta
\le
\frac52c\Delta_0
\frac{\lambda_KH_0\log\frac1\delta}{KE_K^2}
\le
\frac52c\Delta_0.
$$
Consequently,
$$
\mathcal B_K
\le
\left[
\frac{395}{2}c
+
\left(
8+4\sqrt3+\nu(1+\sqrt3)
\right)c^2
\right]\Delta_0.
$$

Since $\nu\in(0,1)$, choosing $c:=\frac1{256}$ gives
$$
\begin{cases}
\gamma L_0<\frac1{64},
\\
\gamma L_1\lambda_K<\frac1{64},
\\
\mathcal B_K<\Delta_0.
\end{cases}
$$

We now give the parameter choices used in
\cref{thm:nonconvex-main}.
Let $\ell_\delta:=\log\frac{e}{\delta}.$ Define $H_0:=\frac{2}{\nu}\left(L_1\Delta_0
+\sqrt{L_1^2\Delta_0^2+\nu L_0\Delta_0}\right)$ and $\underline{\lambda}_K:=\max\left\{8\sigma,
\sqrt{\frac{\Delta_0L_0}{K}},\frac{H_0\ell_\delta}{K}\right\}.$ The clipping level is
$$
\lambda_K:=\max\left\{\underline{\lambda}_K,\min\left\{\left(\frac{K\sigma^{2\alpha}
}{H_0\ell_\delta}\right)^{\frac1{2\alpha-1}},
\left(\frac{K\sigma^{3\alpha}}{\Delta_0L_0}\right)^{\frac1{3\alpha-2}}\right\}\right\}.
$$
For this choice of $\lambda_K$, set
$$
\begin{cases}
E_K:=\max\left\{\sqrt{\frac{\Delta_0L_0}{K}},\frac{\sigma^\alpha}{\lambda_K^{\alpha-1}},
\sqrt{\frac{\lambda_KH_0\ell_\delta}{K}},\left(\frac{\Delta_0L_0\sigma^\alpha\lambda_K^{2-\alpha}
}{K}\right)^{\frac14}\right\},\\[0.3em]
\rho_0:=\max\left\{1,\frac{H_0}{\lambda_K}\right\}.
\end{cases}
$$
The stepsize is then chosen as
$$
\gamma_K:=\frac1{256}\min\left\{\frac{\Delta_0}{KE_K^2},\frac{\Delta_0}{\sigma^\alpha
\lambda_K^{2-\alpha}\rho_0\ell_\delta}\right\}.
$$
The second balancing term defining $\lambda_K$
is understood as $+\infty$ when $L_0=0$, and the second quantity
in the minimum defining $\gamma_K$ is understood as $+\infty$
when $\sigma=0$.

\begin{theorem}
\label{thm:nonconvex-main}
Assume that Assumptions~\ref{asm: smoothness} and
\ref{asm: stochastic} hold.
Let $\delta\in(0,1)$, $K\ge1$, and suppose that $\Delta_0>0$.
Run \algname{Clip-SGD} with the clipping level and stepsize specified
above. Then, with probability at least $1-\delta$,
$$
\begin{aligned}
\frac1K \sum_{k=0}^{K-1} \norm{\nabla f(x_k)} \le 4704\sqrt2 \max\Bigg\{
&\sqrt{\frac{L_0\Delta_0}{K}},\frac{\frac{4}{\nu}\left(L_1\Delta_0+\sqrt{L_0\Delta_0}
\right)\log\frac{e}{\delta}}{K},\\
&\left(\frac{\left(\frac{4}{\nu}\right)^{\alpha-1}\left(L_1\Delta_0+\sqrt{L_0\Delta_0}
\right)^{\alpha-1}\sigma^\alpha\log^{\alpha-1}\frac{e}{\delta}}{K^{\alpha-1}
}\right)^{\frac1{2\alpha-1}},\\
&\left(\frac{(L_0\Delta_0)^{\alpha-1}\sigma^\alpha}{K^{\alpha-1}}\right)^{\frac1{3\alpha-2}}
\Bigg\}.
\end{aligned}
$$
\end{theorem}

\subsection{Comparison with prior work}
\label{subsec:nonconvex-comparison}

We compare the nonconvex guarantees using the averaged gradient norm
$$
\frac{1}{K}\sum_{k=0}^{K-1}\norm{\nabla f(x_k)}.
$$
When a prior result controls the averaged squared-gradient norm,
we apply Cauchy--Schwarz:
$$
\frac{1}{K}\sum_{k=0}^{K-1}\norm{\nabla f(x_k)}
\le
\left(
\frac{1}{K}\sum_{k=0}^{K-1}\norm{\nabla f(x_k)}^2
\right)^{1/2}.
$$
Thus, the comparisons below concern guarantees for the averaged
gradient norm.

Let
$
\Delta_0:=f(x_0)-f^\star,
\,
G_0:=L_1\Delta_0+\sqrt{L_0\Delta_0}$ and $
\ell_\delta:=\log\frac{e}{\delta}.
$
For the prescribed parameter choices,
\cref{thm:nonconvex-main} gives, with probability at least $1-\delta$,
$$
\begin{aligned}
\frac{1}{K}\sum_{k=0}^{K-1}\norm{\nabla f(x_k)}
\le \mathcal{O}\Bigg(
\max\Bigg\{&
\sqrt{\frac{L_0\Delta_0}{K}},
\frac{G_0\ell_\delta}{K},
\\
&
\left(
\frac{G_0^{\alpha-1}\sigma^\alpha\ell_\delta^{\alpha-1}}
{K^{\alpha-1}}
\right)^{\frac{1}{2\alpha-1}},
\left(
\frac{(L_0\Delta_0)^{\alpha-1}\sigma^\alpha}
{K^{\alpha-1}}
\right)^{\frac{1}{3\alpha-2}}
\Bigg\}
\Bigg).
\end{aligned}
$$
where $\mathcal{O}$ hides only universal numerical constants.
For fixed problem parameters satisfying $L_0\Delta_0>0$ and
$\sigma>0$, and fixed $\delta$, the asymptotically dominant term is
$$
\left(
\frac{
(L_0\Delta_0)^{\alpha-1}\sigma^\alpha
}{
K^{\alpha-1}
}
\right)^{\frac{1}{3\alpha-2}}.
$$
This term contains no explicit dependence on $L_1$ or $\delta$.
The full bound and the parameter choices retain such dependencies
through the other terms.

\phantomsection
\label{par:comparison-nguyen-nonconvex}
\paragraph{Comparison with \citet{nguyen2023improved} in the nonconvex case ($L_1=0$).}
In the classical $L$-smooth setting, \citet{nguyen2023improved}
establish a time-optimal high-probability guarantee for
\algname{Clip-SGD} under bounded central $\alpha$-th moment noise.
Applying Cauchy--Schwarz to their averaged squared-gradient bound
gives the optimal stochastic dependence
$K^{-\frac{\alpha-1}{3\alpha-2}}$
for the averaged gradient norm. To compare the problem-parameter dependence, suppose that
$L_0\Delta_0>0$ and $\sigma>0$, and set
$z:=\sigma/\sqrt{L_0\Delta_0}$.
Ignoring logarithmic factors in $\delta$, their stochastic terms
with this $K$-dependence scale as
$
\sqrt{L_0\Delta_0}\,
K^{-\frac{\alpha-1}{3\alpha-2}}
\max\left\{
z^{\frac12},
z^{\frac{\alpha}{2(\alpha-1)}}
\right\},
$
whereas our leading stochastic term is
$
\sqrt{L_0\Delta_0}\,
K^{-\frac{\alpha-1}{3\alpha-2}}
z^{\frac{\alpha}{3\alpha-2}}.
$
Since, for every $\alpha\in(1,2]$,
$
\frac12
\le \frac{\alpha}{3\alpha-2}
< \frac{\alpha}{2(\alpha-1)},
$
our leading term has uniformly no worse joint polynomial
dependence on $L_0\Delta_0$ and $\sigma$.
For $\alpha\in(1,2)$, the improvement is strict whenever $z\neq1$,
with equality at $z=1$.
For $\alpha=2$, the dependence coincides when $0<z\le1$
and is strictly better when $z>1$.
Our leading term also contains no logarithmic dependence on $\delta$.
These comparisons concern the leading stochastic terms;
the full bounds also contain faster-decaying terms.

 \phantomsection
\label{par:comparison-an-nonconvex}
\paragraph{Comparison with \citet{an2025tight} in the nonconvex case.}
\citet{an2025tight} also study \algname{Clip-SGD} under heavy-tailed
noise and $(L_0,L_1)$-smoothness.
After applying Cauchy--Schwarz to their averaged squared-gradient
bound, their stated rate becomes
$
\mathcal O\!\left(
K^{-\frac{\alpha-1}{3\alpha-2}}
\log^{\frac{\alpha-1}{2\alpha-1}}\frac{K}{\delta}
\right),
$
with the problem-parameter dependence suppressed. To make this dependence explicit, we use the quantities $R$
and $\mathcal R$ introduced in Appendix B.1 of
\citet{an2025tight}, just below their Eq.~(15) and just after
their Eq.~(19), respectively.
Writing their initial objective gap as $\Delta_0$, these are
$
R:=\frac{L_0+(4L_1+1)\Delta_0}{2}
$
and
$
\mathcal R:=L_0+2(L_1+1)R.
$
We also set $\rho_K:=\log(4K/\delta)$.
Substituting their clipping and stepsize choices in Eq.~(32)
into their descent bound in Eq.~(33), and applying
Cauchy--Schwarz, gives $\frac1K\sum_{k=0}^{K-1}\norm{\nabla f(x_k)}
\le
16\sqrt{\mathcal R\rho_K\lambda}\,
K^{-\frac{2\alpha-1}{2(3\alpha-2)}}.$
Writing $a:=(\alpha-1)/(3\alpha-2)$ and $b:=(2\alpha-1)/(2(3\alpha-2))$, expanding their displayed clipping threshold gives the full parameter-dependent expression
$$\mathcal O\left(\max\left\{\frac{\sqrt{\mathcal R R\rho_K}}{K^b},\frac{(\mathcal R\rho_K)^{\frac{\alpha-1}{2\alpha-1}}\sigma^{\frac{\alpha}{2\alpha-1}}}{K^a},\frac{\mathcal R^{\frac{\alpha-1}{2\alpha}}\Delta_0^{\frac1{2\alpha}}\sigma^{1/2}\rho_K^{\frac12-\frac1\alpha}}{K^a},\frac{\sqrt{\mathcal R\sigma}\left(1+\frac{L_1\Delta_0}{32\mathcal R}\right)^{1/\alpha}\rho_K^{\frac12-\frac1\alpha}}{K^a}\right\}\right).$$
Here only universal numerical constants are suppressed. For fixed problem parameters with $\sigma>0$ and fixed $\delta$, the first stochastic term eventually dominates, giving
$
\mathcal O\!\left(
\frac{
(\mathcal R\rho_K)^{\frac{\alpha-1}{2\alpha-1}}
\sigma^{\frac{\alpha}{2\alpha-1}}
}{
K^{\frac{\alpha-1}{3\alpha-2}}
}
\right).
$
Thus, their leading stochastic prefactor retains an
$L_1$-dependence through $\mathcal R$.
Our corresponding term,
$
\left(\frac{(L_0\Delta_0)^{\alpha-1}\sigma^\alpha}{K^{\alpha-1}}\right)^{\frac{1}{3\alpha-2}},
$
has the same polynomial dependence on $K$, contains no explicit $L_1$ dependence, and has no logarithmic factors in $K$ or $1/\delta$. This comparison concerns the asymptotically dominant stochastic terms for fixed confidence and fixed problem parameters with $L_0\Delta_0>0$ and $\sigma>0$. The guarantees also differ in the noiseless regime.
Formally specializing their displayed parameter choices and bound to $\sigma=0$ gives $\lambda=4R$ and the expression
$
\sqrt{\mathcal R R\rho_K}\,
K^{-\frac{2\alpha-1}{2(3\alpha-2)}}.
$
For every $\alpha>1$, this decays more slowly than $K^{-1/2}$,
so their prescribed tuning does not recover the deterministic
generalized-smoothness rate.
Our bound reduces to
$
\mathcal O\!\left(
\sqrt{\frac{L_0\Delta_0}{K}}
+
\frac{
(L_1\Delta_0+\sqrt{L_0\Delta_0})\log(e/\delta)
}{K}
\right),
$
recovering the deterministic rate up to a confidence logarithm
in the faster-decaying transient term.

\paragraph{Comparison with \cite{liu2025nonconvex}.}
\citet[Theorem~3.2]{liu2025nonconvex} establish an in-expectation
guarantee for normalized SGD with momentum under
$(L_0,L_1)$-smoothness and heavy-tailed noise. With batch size one,
their leading stochastic term matches ours in its dependence on
$K,L_0,\Delta_0,\sigma$ and contains no explicit dependence on $L_1$.
Their stated bound additionally contains the faster-decaying term
$\sigma/K^{(\alpha-1)/(2\alpha-1)}$.
Our result establishes the same leading stochastic dependence
with high probability for standard \algname{Clip-SGD}.

\paragraph{Dependence on the initial objective gap.}
The nonconvex comparisons above are made at fixed
$\Delta_0=f(x_0)-f^\star$.
Under $(L_0,L_1)$-smoothness, this gap can grow exponentially
with $L_1R_0$.
Consequently, the absence of explicit $L_1$-dependence in our
leading stochastic term does not imply polynomial dependence
on $R_0$.
The distance-based guarantee avoiding exponential dependence on
$L_1R_0$ is specific to our convex results.

\begin{table*}[t!]
\centering
\caption{
Comparison of deterministic and high-probability convergence
guarantees for first-order nonconvex optimization under
$L$-smoothness and $(L_0,L_1)$-smoothness
(Assumption~\ref{asm: smoothness}).
All rates are stated for the averaged gradient norm
$\frac{1}{K}\sum_{k=0}^{K-1}\norm{\nabla f(x_k)}$.
We use $G_0:=L_1\Delta_0+\sqrt{L_0\Delta_0}$
and $\ell_\delta:=\log(e/\delta)$.
}
\label{tab:comparison-nonconvex}

\small
\setlength{\tabcolsep}{7pt}
\renewcommand{\arraystretch}{1.15}
\setcellgapes{4pt}
\makegapedcells

\resizebox{\textwidth}{!}{%
\begin{threeparttable}
\begin{tabular}{|c|c|c|c|}
\hline

\textbf{Reference}
&
\textbf{Smoothness}
&
\textbf{Heavy-tailed noise}
&
\textbf{Rate}
\\
\hline

\cite{vankov2024optimizing}\tnote{\color{blue}(1)}
&
$(L_0,L_1)$
&
\xmark
&
$\displaystyle
\mathcal{O}\!\left(
\sqrt{\frac{L_0\Delta_0}{K}}
+
\frac{L_1\Delta_0}{K}
\right)$
\\
\hline

\cite{nguyen2023improved}\tnote{\color{blue}(2)}
&
$L$-smooth
&
\cmark
&
$\displaystyle
\mathcal O\!\left(
(L_0\Delta_0)^{\frac14}\sqrt{\ell_\delta}
\max\left\{
\left(
\frac{\ell_\delta}{\sqrt{L_0\Delta_0}}
\right)^{\frac1{2(\alpha-1)}}
\frac{\sigma^{\frac{\alpha}{2(\alpha-1)}}}
{K^{\frac{\alpha-1}{3\alpha-2}}},
\frac{(L_0\Delta_0)^{\frac14}}
{K^{\frac{2\alpha-1}{2(3\alpha-2)}}},
\frac{\sqrt{\sigma}}
{K^{\frac{\alpha-1}{3\alpha-2}}}
\right\}
\right)$
\\
\hline

\cite{an2025tight}
&
$(L_0,L_1)$
&
\cmark
&
$\displaystyle
\mathcal O\!\left(
\frac{
\log^{\frac{\alpha-1}{2\alpha-1}}(K/\delta)
}{
K^{\frac{\alpha-1}{3\alpha-2}}
}
\right)$
\tnote{\color{blue}(3)}
\\
\hline

\textbf{This work}
&
$(L_0,L_1)$
&
\cmark
&
$\displaystyle
\mathcal O\!\left(
\max\left\{
\sqrt{\frac{L_0\Delta_0}{K}},
\frac{G_0\ell_\delta}{K},
\left(
\frac{
G_0^{\alpha-1}\sigma^\alpha
\ell_\delta^{\alpha-1}
}{
K^{\alpha-1}
}
\right)^{\frac1{2\alpha-1}},
\left(
\frac{
(L_0\Delta_0)^{\alpha-1}\sigma^\alpha
}{
K^{\alpha-1}
}
\right)^{\frac1{3\alpha-2}}
\right\}
\right)$
\\
\hline

\end{tabular}

\begin{tablenotes}\footnotesize

\item[\color{blue}(1)]
Deterministic setting. The result of \citet{vankov2024optimizing}
is written for the averaged gradient norm to match our stationarity
criterion.

\item[\color{blue}(2)]
\citet{nguyen2023improved} consider the $L$-smooth case,
corresponding to $L_1=0$ and $L=L_0$.
Their averaged squared-gradient bound is converted using
Cauchy--Schwarz. We replace
$\max\{\log(1/\delta),1\}$ by the equivalent quantity
$\log(e/\delta)$ up to universal constants.
For a parameter-level comparison with our result, see
\hyperref[par:comparison-nguyen-nonconvex]{the comparison with
Nguyen et al.~(2023)}.

\item[\color{blue}(3)]
The displayed rate of \citet{an2025tight} is their stated stochastic iteration rate, converted from the averaged squared-gradient criterion using Cauchy--Schwarz. The dependence on $L_0$, $L_1$, $\Delta_0$, and $\sigma$ is suppressed. The full parameter-dependent expression and the formal noiseless specialization of their prescribed tuning are given in \hyperref[par:comparison-an-nonconvex]{the comparison with An et al.~(2026)}.

\end{tablenotes}
\end{threeparttable}%
}

\vspace{-0.3cm}
\end{table*}
\subsection{Proofs}

\begin{proof}[Proof of \cref{lem:nonconvex-absorbable-clipping-bias}]
By \cref{lem:directional-clipping-bias}, we distinguish two cases.
\begin{itemize}
    \item If $\norm{\nabla f(x_k)}\le\lambda$, then
    $$
    \begin{aligned}
    -\ip{\nabla f(x_k)}{\mathbb E[e_k\mid\mathcal F_k]}&\le
    \frac18 \norm{\nabla f(x_k)}\min\left\{\norm{\nabla f(x_k)},\lambda
    \right\}+\norm{\nabla f(x_k)}\frac{\sigma^\alpha}{\lambda^{\alpha-1}}\\
    &\le \frac14 \norm{\nabla f(x_k)}\min\left\{\norm{\nabla f(x_k)},\lambda\right\}
    +2\frac{\sigma^{2\alpha}}{\lambda^{2\alpha-2}}.\end{aligned}
    $$
    The first inequality uses $\frac{\sigma}{\lambda} \le \frac18$. For the
    second inequality, since $\norm{\nabla f(x_k)}\le\lambda$, we have
    $\norm{\nabla f(x_k)} \min\left\{\norm{\nabla f(x_k)},\lambda\right\}
    =\norm{\nabla f(x_k)}^2,$
    and Young's inequality gives $$\norm{\nabla f(x_k)}
    \frac{\sigma^\alpha}{\lambda^{\alpha-1}} \le \frac18\norm{\nabla f(x_k)}^2
    +2\frac{\sigma^{2\alpha}}{\lambda^{2\alpha-2}}.$$

    \item If $\norm{\nabla f(x_k)}>\lambda$, then
    $$
    \begin{aligned}
    -\ip{\nabla f(x_k)}
    {\mathbb E[e_k\mid\mathcal F_k]}
    &\le \frac18  \norm{\nabla f(x_k)} \min\left\{
    \norm{\nabla f(x_k)},\lambda \right\} +\lambda\norm{\nabla f(x_k)}
    \frac{\sigma^\alpha}{\lambda^\alpha}\\
    &\le  \frac14  \norm{\nabla f(x_k)} \min\left\{
    \norm{\nabla f(x_k)},\lambda \right\},
    \end{aligned}
    $$
    where we used
    $
    \lambda\norm{\nabla f(x_k)}=\norm{\nabla f(x_k)}\min\left\{
    \norm{\nabla f(x_k)},\lambda \right\}
    $
    and $(\sigma/\lambda)^\alpha\le1/8$.
\end{itemize}
\end{proof}

\begin{proof}[Proof of \cref{lem:nonconvex-joint-one-step}]
Fix $k\ge0$. We have
$$
L_1\norm{x_{k+1}-x_k}= \gamma L_1\norm{g_k} \le\gamma L_1\lambda\le 1.
$$
By \cref{lem:local-smoothness-short-segment},
$$ 
\begin{aligned}
f(x_{k+1})-f(x_k) &\le-\gamma\ip{\nabla f(x_k)}{g_k}+
2\gamma^2 \left(L_0+L_1\norm{\nabla f(x_k)}\right)\norm{g_k}^2\\
&=\Phi_k.
\end{aligned}
$$

We remark that $\ip{\nabla f(x_k)}{q_k}= \norm{\nabla f(x_k)}\min\left\{ \norm{\nabla f(x_k)},\lambda \right\}.
$
By the proof of Lemma \ref{lem:quadratic-error-hp},
$$
\mathbb E[\norm{e_k}^2\mid\mathcal F_k]\le 2^{2-\alpha}\sigma^\alpha\lambda^{2-\alpha}
\le 2\sigma^\alpha\lambda^{2-\alpha}.
$$
Consequently,
$$
\mathbb E[\norm{g_k}^2\mid\mathcal F_k]
\le 2\norm{q_k}^2+2\mathbb E[\norm{e_k}^2\mid\mathcal F_k] \le 2\min\left\{\norm{\nabla f(x_k)},\lambda\right\}^2 +4\sigma^\alpha\lambda^{2-\alpha}.
$$
Moreover, by \cref{lem:nonconvex-absorbable-clipping-bias},
$$
\begin{aligned}
-\ip{\nabla f(x_k)}{\mathbb E[e_k\mid\mathcal F_k]}
\le{}& \frac14 \norm{\nabla f(x_k)} \min\left\{\norm{\nabla f(x_k)},\lambda\right\}
+2\frac{\sigma^{2\alpha}}{\lambda^{2\alpha-2}}.
\end{aligned}
$$
It follows that
\begin{equation}
\label{eq:nonconvex-conditional-phi}
\begin{aligned}
\mathbb E[\Phi_k\mid\mathcal F_k] \le& -\frac{3\gamma}{4}\norm{\nabla f(x_k)}
\min\left\{\norm{\nabla f(x_k)},\lambda\right\}+2\gamma\frac{\sigma^{2\alpha}}{\lambda^{2\alpha-2}}\\
&+4\gamma^2\left(L_0+L_1\norm{\nabla f(x_k)}\right)\min\left\{\norm{\nabla f(x_k)},\lambda
\right\}^2\\
&+8\gamma^2\left(L_0+L_1\norm{\nabla f(x_k)}\right)\sigma^\alpha\lambda^{2-\alpha}.
\end{aligned}
\end{equation}

Since $\left(L_0+L_1\norm{\nabla f(x_k)}\right)
\min\left\{\norm{\nabla f(x_k)},\lambda\right\}^2 \le
\left(L_0+L_1\lambda\right) \norm{\nabla f(x_k)}
\min\left\{\norm{\nabla f(x_k)},\lambda\right\},$
the stepsize conditions imply
$$
4\gamma^2\left(L_0+L_1\norm{\nabla f(x_k)}\right) \min\left\{
\norm{\nabla f(x_k)},\lambda\right\}^2\le
\frac{\gamma}{8}\norm{\nabla f(x_k)}\min\left\{\norm{\nabla f(x_k)},\lambda
\right\}.$$

We now control the last term of \cref{eq:nonconvex-conditional-phi} . If
$\norm{\nabla f(x_k)}\le\lambda$, Young's inequality gives
\begin{align*}
8\gamma^2L_1\norm{\nabla f(x_k)}\sigma^\alpha\lambda^{2-\alpha} &\le
\frac{\gamma}{8}\norm{\nabla f(x_k)}^2+128\gamma^3L_1^2\sigma^{2\alpha}\lambda^{4-2\alpha}\\
&\le\frac{\gamma}{8}\norm{\nabla f(x_k)}\min\left\{\norm{\nabla f(x_k)},\lambda\right\}
+\frac{\gamma}{32}\frac{\sigma^{2\alpha}}{\lambda^{2\alpha-2}},
\end{align*}
where the last inequality uses
$\gamma L_1\lambda\le \frac{1}{64}$.

If $\norm{\nabla f(x_k)}>\lambda$, then
\begin{align*}
8\gamma^2L_1\norm{\nabla f(x_k)} \sigma^\alpha\lambda^{2-\alpha}
&=8\gamma \norm{\nabla f(x_k)}\lambda \left(\gamma L_1\lambda\right)
\left(\frac{\sigma}{\lambda}\right)^\alpha\\
&\le
\frac{\gamma}{64} \norm{\nabla f(x_k)}
\min\left\{\norm{\nabla f(x_k)},\lambda\right\}.
\end{align*}
Thus, in both cases,
$$
8\gamma^2L_1 \norm{\nabla f(x_k)} \sigma^\alpha\lambda^{2-\alpha}\le
\frac{\gamma}{8} \norm{\nabla f(x_k)}
\min\left\{\norm{\nabla f(x_k)},\lambda\right\}
+\frac{\gamma}{32}\frac{\sigma^{2\alpha}}{\lambda^{2\alpha-2}}.
$$
We deduce that
$$
\mathbb E[\Phi_k\mid\mathcal F_k]\le-\frac{\gamma}{2}
\norm{\nabla f(x_k)}\min\left\{
\norm{\nabla f(x_k)},\lambda\right\}
+3\gamma\frac{\sigma^{2\alpha}}{\lambda^{2\alpha-2}}
+8\gamma^2L_0\sigma^\alpha\lambda^{2-\alpha}.
$$

It remains to prove the last two inequalities of the lemma. Conditionally on
$\mathcal F_k$, define
$$
\varphi_k(z):=-\gamma\ip{\nabla f(x_k)}{z}+2\gamma^2\left(L_0+L_1\norm{\nabla f(x_k)}\right)\norm{z}^2.
$$
For any $z,z'$ satisfying $\norm{z},\norm{z'}\le\lambda$,
$$
\begin{aligned}
|\varphi_k(z)-\varphi_k(z')|&\le \Big(\gamma\norm{\nabla f(x_k)}+4\gamma^2 \left(
L_0+L_1\norm{\nabla f(x_k)}\right)\lambda\Big)\norm{z-z'}\\
&\le \left(2\gamma\norm{\nabla f(x_k)} +4\gamma^2L_0\lambda\right)\norm{z-z'},
\end{aligned}
$$
where the last inequality uses
$ \gamma L_1\lambda\le \frac{1}{64}$.

Since $\Phi_k=\varphi_k(g_k)$ and $\norm{g_k}\le\lambda$, 

\begin{align*}
\left|\Phi_k-\mathbb E[\Phi_k\mid\mathcal F_k]\right|&\le2\lambda\left(2\gamma\norm{\nabla f(x_k)}
+4\gamma^2L_0\lambda\right)\\
&=4\gamma\lambda\norm{\nabla f(x_k)}+8\gamma^2L_0\lambda^2.
\end{align*}

Since $\Phi_k=\varphi_k(g_k)$,
\begin{align*}
\operatorname{Var}(\Phi_k\mid\mathcal F_k)
&\le
\mathbb E\left[\left(\varphi_k(g_k)-\varphi_k(q_k)\right)^2
\middle|\mathcal F_k\right]\\
&\le \left(2\gamma\norm{\nabla f(x_k)}
+4\gamma^2L_0\lambda\right)^2\mathbb E[\norm{e_k}^2\mid\mathcal F_k]\\
&\le 2^{2-\alpha} \sigma^\alpha\lambda^{2-\alpha}\left(
2\gamma\norm{\nabla f(x_k)}+4\gamma^2L_0\lambda\right)^2\\
&\le 64\sigma^\alpha\lambda^{2-\alpha}
\left(\gamma^2\norm{\nabla f(x_k)}^2+\gamma^4L_0^2\lambda^2\right),
\end{align*}
where the third inequality uses
$$
\mathbb E[\norm{e_k}^2\mid\mathcal F_k]
\le
2^{2-\alpha}\sigma^\alpha\lambda^{2-\alpha},
$$
established in the proof of
\cref{lem:quadratic-error-hp}, and the last inequality
uses $2^{2-\alpha}\le2$ and
$(a+b)^2\le2a^2+2b^2$.

\end{proof}

\begin{proof}[Proof of \cref{lem:nonconvex-stopped-martingale}]
We have $\Phi_k:=-\gamma\ip{\nabla f(x_k)}{g_k}
+2\gamma^2\left(L_0+L_1\norm{\nabla f(x_k)}\right) \norm{g_k}^2.$ 
By \cref{lemma_grad_bound}, 

$$
|\Phi_k|  1_{\{k<\widetilde\tau\}} \le \gamma \lambda H_0+2 \gamma^2 L_0 \lambda^2 + 2\gamma \lambda \frac{1}{64} H_0 =\frac{33}{32} \gamma \lambda H_0+2\gamma^2L_0\lambda^2.
$$
Similarly, we have
$$
|\mathbb E[\Phi_k\mid\mathcal F_k]|  1_{\{k<\widetilde\tau\}} \le  \frac{33}{32} \gamma \lambda H_0+2\gamma^2L_0\lambda^2.
$$
Therefore  $|Z_k| \le \frac{5}{2}  \gamma\lambda H_0+4\gamma^2L_0\lambda^2$  almost surely. We note that $\mathbf 1_{\{k<\widetilde\tau\}}$ is $\mathcal F_k$-measurable. By
\cref{lem:nonconvex-joint-one-step},

\begin{align*}
 \sum_{k=0}^{K-1}\mathbb E[Z_k^2\mid\mathcal F_k] &\le \sum_{k=0}^{K-1} \mathbf 1_{\{k<\widetilde\tau\}}\operatorname{Var}(\Phi_k\mid\mathcal F_k) \\
 &\le 64\sigma^\alpha\lambda^{2-\alpha}\left(
\gamma^2  \sum_{k=0}^{K-1} \mathbf 1_{\{k<\widetilde\tau\}}\norm{\nabla f(x_k)}^2+ K \gamma^4L_0^2\lambda^2\right)\\
&\le  64\sigma^\alpha\lambda^{2-\alpha}\left( \gamma^2\rho_0S_K
+K\gamma^4L_0^2\lambda^2\right).
\end{align*}

\end{proof}

\begin{proof}[Proof of \cref{prop:nonconvex-stopped-descent}]
By \cref{lem:nonconvex-stopped-martingale}, $(Z_k)_{k\ge0}$ is a martingale-difference sequence satisfying $|Z_k|\le \frac{5}{2}\gamma\lambda H_0 +4\gamma^2L_0\lambda^2$ almost surely
and
$$
\sum_{k=0}^{K-1} \E[Z_k^2\mid\mathcal F_k] \le 64\sigma^\alpha\lambda^{2-\alpha}
\left(\gamma^2\rho_0S_K+K\gamma^4L_0^2\lambda^2\right).
$$
Applying \cref{prop:freedman-absorbable} with
$\eta=\frac{\gamma}{4}$ gives, with probability at least $1-\delta$,
\begin{align*}
\sum_{k=0}^{K-1}Z_k \le & \frac{\gamma}{4}S_K+ 192\gamma \sigma^\alpha\lambda^{2-\alpha}\rho_0
\log\frac{1}{\delta}+\sqrt{ 192K\gamma^4L_0^2\sigma^\alpha\lambda^{4-\alpha}\log\frac{1}{\delta}}\\
&+\left(\frac{5}{2}\gamma\lambda H_0+4\gamma^2L_0\lambda^2\right)\log\frac{1}{\delta}.
\end{align*}
Young's inequality gives
\begin{align*}
\sqrt{192K\gamma^4L_0^2\sigma^\alpha\lambda^{4-\alpha}\log\frac{1}{\delta}}
&=8\sqrt{3}\gamma^2L_0\lambda \sqrt{ K\sigma^\alpha\lambda^{2-\alpha}\log\frac{1}{\delta}}\\
&\quad\le 4\sqrt{3}\gamma^2L_0 \left(K\sigma^\alpha\lambda^{2-\alpha}+\lambda^2\log\frac{1}{\delta} \right).
\end{align*}
Therefore,
$$
\begin{aligned}
\sum_{k=0}^{K-1}Z_k \le& \frac{\gamma}{4}S_K
+ 4\sqrt{3}\gamma^2L_0K \sigma^\alpha\lambda^{2-\alpha} +\left(4+4\sqrt{3}\right) \gamma^2L_0\lambda^2 \log\frac{1}{\delta}\\
&+ 192\gamma \sigma^\alpha\lambda^{2-\alpha}\rho_0 \log\frac{1}{\delta}
+ \frac{5}{2}\gamma\lambda H_0 \log\frac{1}{\delta}.
\end{aligned}
$$

On the other hand, by
\cref{lem:nonconvex-joint-one-step},
$$
\begin{aligned}
\Delta_{K\wedge\widetilde\tau}-\Delta_0 &\le \sum_{k=0}^{K-1} \mathbf 1_{\{k<\widetilde\tau\}}\Phi_k
\\
&= \sum_{k=0}^{K-1} \mathbf 1_{\{k<\widetilde\tau\}} \E[\Phi_k\mid\mathcal F_k]
+\sum_{k=0}^{K-1}Z_k\\
&\le -\frac{\gamma}{2}S_K
+ 3\gamma K \frac{\sigma^{2\alpha}}{\lambda^{2\alpha-2}}+ 8\gamma^2L_0K \sigma^\alpha\lambda^{2-\alpha}
+ \sum_{k=0}^{K-1}Z_k.
\end{aligned}
$$
We deduce that
$$
\begin{aligned}
\Delta_{K\wedge\widetilde\tau} + \frac{\gamma}{4}S_K
\le& \Delta_0 + 3\gamma K \frac{\sigma^{2\alpha}}{\lambda^{2\alpha-2}}+
\left(8+4\sqrt{3}\right) \gamma^2L_0K \sigma^\alpha\lambda^{2-\alpha}\\
&+
\left(4+4\sqrt{3}\right)\gamma^2L_0\lambda^2
\log\frac{1}{\delta}+
192\gamma\sigma^\alpha\lambda^{2-\alpha}\rho_0
\log\frac{1}{\delta} + \frac{5}{2}\gamma\lambda H_0\log\frac{1}{\delta}.
\end{aligned}
$$
\end{proof}

\begin{proof}[Proof of \cref{prop:nonconvex-averaged-stationarity}]
We work on the event of
\cref{prop:nonconvex-stopped-descent}, which has probability at least
$1-\delta$. Since $\mathcal B_K\le\Delta_0$, on this event,
$$
\Delta_{K\wedge\widetilde\tau} +\frac{\gamma}{4}S_K
\le \Delta_0+\mathcal B_K \le 2\Delta_0.
$$
If $\widetilde\tau\le K$, then $\Delta_{K\wedge\widetilde\tau}= \Delta_{\widetilde\tau}>2\Delta_0,$
which is a contradiction. Hence $\widetilde\tau>K$, and therefore
$$
S_K
=
\sum_{k=0}^{K-1}
\norm{\nabla f(x_k)}
\min\{\norm{\nabla f(x_k)},\lambda\}.
$$
For every $u\ge0$,
$$
u
\le
\sqrt{u\min\{u,\lambda\}}
+
\frac{u\min\{u,\lambda\}}{\lambda}.
$$
Indeed, the first term equals $u$ when $u\le\lambda$, while the
second term equals $u$ when $u>\lambda$. Applying this inequality
with $u=\norm{\nabla f(x_k)}$ and using Jensen's inequality gives
$$
\begin{aligned}
\frac1K\sum_{k=0}^{K-1}\norm{\nabla f(x_k)}
&\le
\frac1K\sum_{k=0}^{K-1}
\sqrt{
\norm{\nabla f(x_k)}
\min\{\norm{\nabla f(x_k)},\lambda\}
}
+
\frac{S_K}{K\lambda}\\
&\le
\sqrt{\frac{S_K}{K}}
+
\frac{S_K}{K\lambda}\\
&\le
\sqrt{\frac{8\Delta_0}{\gamma K}}
+
\frac{8\Delta_0}{\gamma K\lambda}.
\end{aligned}
$$
Finally, since $
K>\frac{32\Delta_0}{\gamma\varepsilon^2}$ and $\lambda\ge\varepsilon,$
we obtain
$$
\frac1K\sum_{k=0}^{K-1}\norm{\nabla f(x_k)}< \frac{\varepsilon}{2}
+ \frac{\varepsilon^2}{4\lambda}
\le \frac{3\varepsilon}{4} \le \varepsilon.
$$
\end{proof}

\begin{proof}[Proof of \cref{lem:nonconvex-tuning}]
For brevity, set $A:=\sqrt{\frac{\Delta_0L_0}{K}}$ and $B:=\frac{H_0\log\frac{e}{\delta}}{K}.$ Then
$$
\underline{\lambda}_K=\max\{8\sigma,A,B\}.
$$

If $\sigma=0$, then, by the definition of $\lambda_K$, $\lambda_K=\underline{\lambda}_K=\max\{A,B\}.
$ Hence

$$
\inf_{\lambda\ge\underline{\lambda}_K}\max\{A,\sqrt{\lambda B}\}=
\max\left\{A,\sqrt{\underline{\lambda}_K B}\right\}=\max\{A,B\}
=\varepsilon_K.
$$

Thus, the claim holds with equality. Suppose now that $\sigma>0$, and define
$$
F(\lambda):= \max\left\{A,\frac{\sigma^\alpha}{\lambda^{\alpha-1}},
\sqrt{\lambda B},\left(A^2\sigma^\alpha\lambda^{2-\alpha}\right)^{\frac14}\right\}.
$$

Among the three $\lambda$-dependent terms, $\frac{\sigma^\alpha}{\lambda^{\alpha-1}},
\, \sqrt{\lambda B}$ and $\left(A^2\sigma^\alpha\lambda^{2-\alpha}\right)^{\frac14},$
the first is decreasing in $\lambda$, whereas the other two are
nondecreasing.

The first term balances the second when $
\lambda =\left( \frac{\sigma^{2\alpha}}{B}
\right)^{\frac1{2\alpha-1}},$ and balances the third when $ \lambda=\left(\frac{\sigma^{3\alpha}}{A^2}\right)^{\frac1{3\alpha-2}},$ where the latter is interpreted as $+\infty$ when $A=0$. Therefore, the maximum of the three terms is nonincreasing until the
smaller of these two balancing points and nondecreasing afterward.
Its unconstrained minimum is thus attained at

$$\bar\lambda:=
\min\left\{ \left(\frac{\sigma^{2\alpha}}{B} \right)^{\frac1{2\alpha-1}},
\left(\frac{\sigma^{3\alpha}}{A^2}\right)^{\frac1{3\alpha-2}}\right\}.
$$ 

The term $A$ is independent of $\lambda$, so the above point remains a minimizer after adding $A$. Since $F$ is nondecreasing on $[\bar\lambda,\infty)$, imposing
$\lambda\ge\underline{\lambda}_K$ yields a constrained minimizer at

$$
\max\left\{\underline{\lambda}_K, \bar\lambda \right\}=\lambda_K.
$$

It remains to compare the constrained minimum with $\varepsilon_K$.
We first note that the unconstrained minimum has a simple form.

Since the first term is decreasing, the smaller balancing point corresponds
to the larger balancing value. Hence,
$$
F(\bar\lambda)
=
\max\left\{
A,
\left(
B^{\alpha-1}\sigma^\alpha
\right)^{\frac1{2\alpha-1}},
\left(
A^{2\alpha-2}\sigma^\alpha
\right)^{\frac1{3\alpha-2}}
\right\}.
$$
Moreover, by the definitions of $A$, $B$, and $\varepsilon_K$,

$$
\varepsilon_K
=
\max\left\{
B,F(\bar\lambda)
\right\}.
$$

Since $\bar\lambda$ is an unconstrained minimizer, $F(\lambda_K)\ge F(\bar\lambda).$
Moreover, $\lambda_K\ge\underline{\lambda}_K\ge B$, so
$F(\lambda_K) \ge \sqrt{\lambda_KB} \ge B.$ Therefore,
$$
F(\lambda_K) \ge \max\left\{B,F(\bar\lambda)\right\} =\varepsilon_K.
$$

It remains to bound $F(\lambda_K)$ from above.  Since
$\lambda_K\ge\bar\lambda$,
$$
\frac{\sigma^\alpha}{\lambda_K^{\alpha-1}}
\le
\frac{\sigma^\alpha}{\bar\lambda^{\alpha-1}}
\le
F(\bar\lambda)
\le
\varepsilon_K.
$$
Next, since $ \lambda_K =\max\left\{8\sigma,A,B,\bar\lambda\right\},$
we have $\sqrt{\lambda_KB} =\max\left\{\sqrt{8\sigma B}, \sqrt{AB},B,\sqrt{\bar\lambda B}
\right\}.$ We remark that
$$
\begin{cases}
\displaystyle \sqrt{AB} \le\max\{A,B\}\le\varepsilon_K,\\[0.8em]
\displaystyle \sqrt{\bar\lambda B}\le F(\bar\lambda) \le \varepsilon_K,\\[0.8em]
\displaystyle \sqrt{\sigma B} = B^{\frac1{2\alpha}} \left[ \left(B^{\alpha-1}\sigma^\alpha
\right)^{\frac1{2\alpha-1}} \right]^{\frac{2\alpha-1}{2\alpha}}\le \varepsilon_K.
\end{cases}
$$
Hence $\sqrt{\lambda_KB} \le 2\sqrt2\,\varepsilon_K.$

Finally, if $A=0$, the last term vanishes. Otherwise, since
$2-\alpha\ge0$,
$$
\left( A^2\sigma^\alpha\lambda_K^{2-\alpha}\right)^{\frac14}
=\max\Bigg\{\left(A^2\sigma^\alpha(8\sigma)^{2-\alpha}\right)^{\frac14},
\left(A^2\sigma^\alpha A^{2-\alpha}\right)^{\frac14},\left(A^2\sigma^\alpha B^{2-\alpha}
\right)^{\frac14},\left(A^2\sigma^\alpha\bar\lambda^{2-\alpha}\right)^{\frac14}
\Bigg\}.
$$ 
The last term is bounded by $F(\bar\lambda)\le\varepsilon_K$, while
the first three satisfy
$$
\begin{aligned}
\left(A^2\sigma^\alpha(8\sigma)^{2-\alpha}\right)^{\frac14}&=8^{\frac{2-\alpha}{4}}
A^{\frac{2-\alpha}{2\alpha}}\left[\left(A^{2\alpha-2}\sigma^\alpha\right)^{\frac1{3\alpha-2}}
\right]^{\frac{3\alpha-2}{2\alpha}},\\
\left(A^2\sigma^\alpha A^{2-\alpha}\right)^{\frac14}&=A^{\frac{3(2-\alpha)}4}
\left[\left(A^{2\alpha-2}\sigma^\alpha\right)^{\frac1{3\alpha-2}}\right]^{\frac{3\alpha-2}{4}},\\
\left(A^2\sigma^\alpha B^{2-\alpha}\right)^{\frac14}&=A^{\frac{2-\alpha}{2}}B^{\frac{2-\alpha}{4}}
\left[\left(A^{2\alpha-2}\sigma^\alpha\right)^{\frac1{3\alpha-2}}\right]^{\frac{3\alpha-2}{4}}.
\end{aligned}
$$
Therefore  $\left( A^2\sigma^\alpha\lambda_K^{2-\alpha}\right)^{\frac14} \le 8^{\frac{2-\alpha}{4}} \varepsilon_K \le 2\sqrt{2} \varepsilon_K.$ Combining the above bounds gives
$$
F(\lambda_K)\le2\sqrt2\,\varepsilon_K.
$$
We deduce that 
$$
\varepsilon_K \le \inf_{\lambda\ge\underline{\lambda}_K}F(\lambda) =
F(\lambda_K) \le 2\sqrt2\,\varepsilon_K,
$$
\end{proof}

\begin{proof}[Proof of \cref{thm:nonconvex-main}]
Set $\lambda:=\lambda_K$ and $\gamma:=\gamma_K$. By the preceding parameter
verification, with $c=\frac{1}{256}$,
$$
\begin{cases}
\lambda \ge 8\sigma,\\
\gamma L_0 < \dfrac{1}{64},\\
\gamma L_1\lambda < \dfrac{1}{64},\\
\mathcal{B}_K < \Delta_0.
\end{cases}
$$

We now apply \cref{prop:nonconvex-stopped-descent}. On an event of
probability at least $1-\delta$,
$$
\Delta_{K\wedge\widetilde\tau}
+\frac{\gamma}{4}S_K
\le
\Delta_0+\mathcal B_K
<
2\Delta_0.
$$
If $\widetilde\tau\le K$, then, by the definition of
$\widetilde\tau$,
$
\Delta_{\widetilde\tau}>2\Delta_0,
$
which contradicts the preceding inequality. Thus
$
\widetilde\tau>K.
$
Since $\Delta_K\ge0$, it follows that
$$
S_K
=
\sum_{k=0}^{K-1}
\norm{\nabla f(x_k)}
\min\left\{
\norm{\nabla f(x_k)},\lambda
\right\}
\le
\frac{8\Delta_0}{\gamma}.
$$

For every $u\ge0$,
$
u
\le
\sqrt{u\min\{u,\lambda\}}
+
\frac{u\min\{u,\lambda\}}{\lambda}.
$
Therefore, by Cauchy--Schwarz,
$$
\begin{aligned}
\frac1K\sum_{k=0}^{K-1}\norm{\nabla f(x_k)}
&\le
\sqrt{\frac{S_K}{K}}
+
\frac{S_K}{K\lambda}
\\
&\le
\sqrt{\frac{8\Delta_0}{\gamma K}}
+
\frac{8\Delta_0}{\gamma K\lambda}.
\end{aligned}
$$

The only remaining difficulty is that $\gamma$ is defined as the minimum of
two quantities. To obtain a bound of order $E_K$ from the preceding estimate,
we need $\gamma$ to remain, up to a universal constant, of the scale
$\Delta_0/(KE_K^2)$. The additional term introduced to control the
$\rho_0$-dependent localization error may make $\gamma$ smaller. We show that
either this term is of the same order as the first one, or the desired
stationarity bound follows directly from localization. This leads to two
cases.

\emph{Case 1.}
Suppose that $\frac{\sigma^\alpha}{\lambda^{\alpha-1}}\le H_0.$
In this regime, the additional $\rho_0$-dependent term does not reduce the
stepsize beyond the scale $\Delta_0/(KE_K^2)$. Indeed, since
$
\rho_0\le 1+\frac{H_0}{\lambda}
$
and $\lambda\ge8\sigma$,
$$
\sigma^\alpha\lambda^{2-\alpha}\rho_0 \le \sigma^\alpha\lambda^{2-\alpha}
+ \sigma^\alpha H_0\lambda^{1-\alpha}=\lambda\frac{\sigma^\alpha}{\lambda^{\alpha-1}}
+ \lambda H_0\left(\frac{\sigma}{\lambda}\right)^\alpha
\le \left(1+\frac{1}{8^\alpha}\right)\lambda H_0 \le \frac{9}{8}\lambda H_0.
$$
Since $
E_K^2 \ge \frac{\lambda H_0\log\frac{e}{\delta}}{K},$
we obtain $\sigma^\alpha\lambda^{2-\alpha}\rho_0 \log\frac{e}{\delta} \le \frac{9}{8}KE_K^2.$
Hence, by the definition of $\gamma$,
$$
\gamma \ge \frac{8c}{9} \frac{\Delta_0}{KE_K^2}.
$$

We also need $E_K\le\lambda$. By the definition of $\lambda_K$, $\lambda
\ge \max\left\{ 8\sigma,\sqrt{\frac{\Delta_0L_0}{K}},\frac{H_0\log\frac{e}{\delta}}{K}\right\}.
$ Therefore,
$$
\begin{cases}
\displaystyle
\sqrt{\frac{\Delta_0L_0}{K}}\le\lambda,\\[2mm]
\displaystyle\frac{\sigma^\alpha}{\lambda^{\alpha-1}}\le\frac{\lambda}{8^\alpha},\\[2mm]
\displaystyle \sqrt{\frac{\lambda H_0\log\frac{e}{\delta}}{K}}\le\lambda,\\[2mm]
\displaystyle \left( \frac{\Delta_0L_0\sigma^\alpha\lambda^{2-\alpha}}{K}\right)^{\frac14}
\le\frac{\lambda}{8^{\alpha/4}},
\end{cases}
$$
and thus $E_K\le\lambda$. Substituting the lower bound on $\gamma$ into the preceding stationarity
estimate gives
$
\sqrt{\frac{8\Delta_0}{\gamma K}} \le \frac{3}{\sqrt c}E_K$ and $
\frac{8\Delta_0}{\gamma K\lambda} \le \frac{9}{c}\frac{E_K^2}{\lambda}\le \frac{9}{c}E_K.$
Hence
$$
\frac{1}{K}\sum_{k=0}^{K-1}\norm{\nabla f(x_k)} \le \left( \frac{3}{\sqrt c}+\frac{9}{c}\right)E_K.
$$
Since $c=\frac{1}{256}$,
$$
\frac{1}{K}\sum_{k=0}^{K-1}\norm{\nabla f(x_k)} \le 2352E_K.
$$

\emph{Case 2.}
Suppose instead that $ \frac{\sigma^\alpha}{\lambda^{\alpha-1}}>H_0.$
In this regime, we do not need to control the stepsize from below: localization
already gives the desired stationarity bound. Indeed, since
$\widetilde\tau>K$, we have $\Delta_k\le2\Delta_0$ for every
$k=0,\ldots,K-1$. By the localized gradient bound,
$$
\norm{\nabla f(x_k)} \le H_0 < \frac{\sigma^\alpha}{\lambda^{\alpha-1}} \le E_K.
$$
Therefore,
$$
\frac{1}{K}\sum_{k=0}^{K-1}\norm{\nabla f(x_k)} \le E_K.
$$

Combining the two cases,
$$
\frac{1}{K}\sum_{k=0}^{K-1}\norm{\nabla f(x_k)} \le 2352E_K.
$$

By \cref{lem:nonconvex-tuning},
$$
E_K \le 2\sqrt{2}\, \max\left\{ \sqrt{\frac{L_0\Delta_0}{K}},\frac{H_0\log\frac{e}{\delta}}{K},
\left(\frac{H_0^{\alpha-1}\sigma^\alpha\log^{\alpha-1}\frac{e}{\delta}}{K^{\alpha-1}}\right)^{\frac{1}{2\alpha-1}},\left(\frac{(L_0\Delta_0)^{\alpha-1}\sigma^\alpha}{K^{\alpha-1}
}\right)^{\frac{1}{3\alpha-2}}\right\}.
$$
Therefore,
$$
\begin{aligned}
\frac{1}{K}\sum_{k=0}^{K-1}\norm{\nabla f(x_k)}\le 4704\sqrt{2}\,
\max\Bigg\{&\sqrt{\frac{L_0\Delta_0}{K}},\frac{H_0\log\frac{e}{\delta}}{K},\\
&
\left(\frac{H_0^{\alpha-1}\sigma^\alpha \log^{\alpha-1}\frac{e}{\delta}}{K^{\alpha-1}}
\right)^{\frac{1}{2\alpha-1}},\left(\frac{(L_0\Delta_0)^{\alpha-1}\sigma^\alpha}{K^{\alpha-1}
}\right)^{\frac{1}{3\alpha-2}}\Bigg\}.
\end{aligned}
$$

It remains only to eliminate $H_0$. By its definition, $H_0=\frac{2}{\nu}\left(L_1\Delta_0
+\sqrt{L_1^2\Delta_0^2+\nu L_0\Delta_0}\right).$ Using
$
\sqrt{a+b}\le\sqrt{a}+\sqrt{b}
$
and $\nu\le1$, we obtain
$$
H_0 \le \frac{2}{\nu} \left( 2L_1\Delta_0+ \sqrt{\nu L_0\Delta_0}\right)
\le \frac{4}{\nu} \left(L_1\Delta_0+\sqrt{L_0\Delta_0}\right).
$$
We deduce that
$$
\begin{aligned}
\frac1K \sum_{k=0}^{K-1} \norm{\nabla f(x_k)}\le 4704\sqrt2 \max\Bigg\{&\sqrt{\frac{L_0\Delta_0}{K}},\frac{\frac{4}{\nu}\left(L_1\Delta_0+\sqrt{L_0\Delta_0}\right)\log\frac{e}{\delta}}{K},\\
&\left(\frac{\left(\frac{4}{\nu}\right)^{\alpha-1}\left(L_1\Delta_0+\sqrt{L_0\Delta_0}
\right)^{\alpha-1}\sigma^\alpha \log^{\alpha-1}\frac{e}{\delta}}{K^{\alpha-1}} \right)^{\frac1{2\alpha-1}},\\
&\left(\frac{(L_0\Delta_0)^{\alpha-1}\sigma^\alpha}{K^{\alpha-1}}\right)^{\frac1{3\alpha-2}}
\Bigg\},
\end{aligned}
$$

\end{proof}

\section{A Lower Bound for Fixed-Parameter Clip-SGD in the Convex Setting}
\label{sec:fixed-clip-lower-bound}
We establish complementary lower bounds for the dependence on
$L_1R_0$ and on the failure probability. Together, under the iteration
conditions of \cref{thm:fixed-clip-lower-bound}, they match the convex
stochastic upper bound up to logarithmic factors in the iteration
budget for fixed-parameter \algname{Clip-SGD}.
Here, $R_0$ denotes a common upper bound on the initial distance to
a minimizer. The lower bound applies to any possibly randomized
output rule measurable with respect to the observed iterates and
stochastic gradients. The hard-function family below is the one
used in \citet[proof of Lemma~8]{crawshaw2025complexity}.
For the construction below, we fix $L_0=1$; the resulting objectives
are therefore $(1,L_1)$-smooth.
\begin{lemma}
\label{lem:fixed-clip-hard-family}
Let $L_1>0$ and $a>0$. Define
$$
\phi_a(x):=\int_0^{|x|} \min\left\{a,\frac{e^{L_1t}-1}{L_1}\right\}\,dt,
\qquad r_a:=\frac{\log(1+L_1a)}{L_1}.
$$
Then $\phi_a$ is convex and continuously differentiable, with minimum
$\phi_a(0)=0$, and, for every $x<y$,
$|\phi_a'(y)-\phi_a'(x)|
\le\left(1+L_1\sup_{u\in[x,y]}|\phi_a'(u)|\right)(y-x)$.
Moreover,
$$
\phi_a'(x)=
\begin{cases}
\operatorname{sign}(x)\dfrac{e^{L_1|x|}-1}{L_1},& |x|<r_a,\\[1ex]
a\,\operatorname{sign}(x),& |x|\ge r_a,
\end{cases}
$$
so $|\phi_a'(x)|\le a$. In addition, $r_a\le a$ and
$\phi_a(x)\ge a(|x|-r_a)_+$, where $(v)_+:=\max\{v,0\}$.
\end{lemma}

For $i\in\{1,2\}$, let
$$
\mathcal T_i:=\left(x_0^{(i)},G_0^{(i)},x_1^{(i)},G_1^{(i)},\ldots,x_{K-1}^{(i)},G_{K-1}^{(i)},x_K^{(i)}\right)
$$
denote the transcript observed by the algorithm on instance $i$, where
$G_k^{(i)}$ is the stochastic gradient observed at $x_k^{(i)}$.

The hard function consists of a curved region around its minimizer and linear tails with slopes \(\pm a\). The linear tails are useful because shifted copies of the function have exactly the same gradient on regions lying on the same side of both minimizers. Thus, while the two objectives have different minimizers, an algorithm whose trajectory remains in such a region may receive identical stochastic-gradient observations on both instances. The next lemma converts this indistinguishability into an objective lower bound.
\begin{lemma}
\label{lem:fixed-clip-indistinguishability}
Let $f_i(x):=\phi_a(x-u_i)$ for $i\in\{1,2\}$, where
$\Delta_u:=|u_1-u_2|>2r_a$. Then $u_i$ is a minimizer of $f_i$ and
$f_i^\star:=f_i(u_i)=0$.

Suppose the randomness of the two runs can be defined jointly so that
each run has its original distribution and
$$
\mathbb P(\mathcal T_1=\mathcal T_2)\ge q
$$
for some $q\in[0,1]$.

Then, for any fixed possibly randomized output rule used on both instances,
depending only on the observed transcript and on independent auxiliary
randomness, let $x_K^{\mathrm{out},(i)}$ denote the point returned by this
rule on instance $i$. There exists $i\in\{1,2\}$ such that
$$
\mathbb P_i\left(
f_i(x_K^{\mathrm{out},(i)})-f_i^\star
\ge
a\left(\frac{\Delta_u}{2}-r_a\right)
\right)\ge \frac q2.
$$
Here, $\mathbb P_i$ denotes probability under instance $i$, including
all randomness used by the algorithm and the output rule.
\end{lemma}
\begin{lemma}
\label{lem:bernoulli-oracle}
Let $f:\mathbb R\to\mathbb R$ be differentiable, with
$|f'(x)|\le a$ for all $x$. For $p\in(0,1]$ and
$Z\sim\operatorname{Bernoulli}(p)$, define
$$
G(x,Z):=\frac{f'(x)}{p}Z.
$$
Then $\mathbb E[G(x,Z)]=f'(x)$ and, for every $\alpha\in(1,2]$,
$$
\mathbb E\left[|G(x,Z)-f'(x)|^\alpha\right]
\le 2a^\alpha p^{1-\alpha}.
$$
With fresh independent draws at successive iterations, these
statements also hold conditionally on the past.
\end{lemma}

\subsection{Dependence on the failure probability}

To obtain the dependence on $\delta$, we consider two hard instances
whose minimizers are located at $-R_0$ and $R_0$. We use the stochastic
oracle
$$
G(x,Z)=\frac{f'(x)}{p}Z,
\qquad
Z\sim\operatorname{Bernoulli}(p),
$$
so that a nonzero stochastic gradient is observed only with probability
$p$. On the event that all $K$ Bernoulli draws are zero, \algname{Clip-SGD} receives
the same sequence of zero gradients on both instances. Consequently, both
trajectories remain at $x_0=0$ and the two observed transcripts are
identical, regardless of the choices of $\gamma$ and $\lambda$.

We choose $p$ so that the all-zero event has probability greater than
$2\delta$. Hence, with probability of order $\delta$, the algorithm cannot
distinguish between two objectives whose minimizers are separated by
$2R_0$. The indistinguishability argument then yields the desired
high-probability lower bound.

\begin{lemma}
\label{lem:fixed-clip-confidence-lower-bound}
Let $L_1, R_0,\sigma>0$, $\alpha\in(1,2]$, and $\delta\in(0,1/4)$. Suppose the integer $K$ satisfies
$$
K\ge \log\left(\frac{1}{2\delta}\right) \max\left\{1,\left(\frac{\sigma}{R_0}\right)^{\alpha/(\alpha-1)}\right\}.
$$
For every fixed $\gamma,\lambda>0$ and every fixed, possibly randomized, output rule depending only on the observed transcript and independent auxiliary randomness, let $x_K^{\mathrm{out}}$ denote the point returned by this rule. Then there exist a convex $(1,L_1)$-smooth function
$f:\mathbb R\to\mathbb R$, with a minimizer $|x^\star|\le R_0$, and an
unbiased stochastic-gradient oracle satisfying \cref{asm: stochastic},
such that \algname{Clip-SGD}, started from $x_0=0$ and using fresh
independent oracle samples, satisfies
$$
\mathbb P\left( f(x_K^{\mathrm{out}})-f^\star \ge \frac{\sigma R_0}{4}
\left( \frac{\log(1/(2\delta))}{K} \right)^{(\alpha-1)/\alpha}\right) >\delta.
$$
\end{lemma}

\subsection{Dependence on \texorpdfstring{$L_1R_0$}{L1 R0}}

The proof separates the fixed choices of $\gamma$ and $\lambda$ according
to whether the clipped step generated by a common tail gradient is large
or small. If this update is large, we use two noiseless shifted copies of $\phi_H$
whose minimizers lie at different positive locations to the right of
$x_0=0$. The shifts are chosen so that $x_0$ lies in the common left
linear tail of both objectives, while the first update lands in their
common right linear tail. Thus both runs observe the same negative
gradient at $x_0$, make the same first update, then observe the same
positive gradient and return to $x_0$. This two-step trajectory repeats,
so the two transcripts are identical.

If instead the update is small, we consider two shifted copies of
$\phi_a$ with minimizers at $R_0/2$ and $R_0$, equipped with
Bernoulli oracles. We choose the sampling probability so that, with
constant probability, the total displacement over $K$ iterations is
less than $R_0/4$. We then choose the slope $a$ to satisfy the noise
constraint and verify that the corresponding transition radius satisfies
$r_a<R_0/4$. On this event, both trajectories remain in the common left
linear tail of the two objectives, where their population gradients
coincide. Using the same Bernoulli draws for both runs therefore makes
their stochastic gradients, iterates, and hence their entire transcripts
identical. In both constructions, the indistinguishability lemma converts
this transcript equality into the desired objective lower bound.

\begin{lemma}
\label{lem:fixed-clip-displacement-lower-bound}
Let $L_1,R_0,\sigma>0$ and $\alpha\in(1,2]$.
Suppose the integer $K$ satisfies
$$
\begin{cases}
K\ge\max\{1,L_1R_0,\sigma L_1\},\\[0.5ex]
\dfrac{L_1R_0}{\log(eK)}>64.
\end{cases}
$$
For every fixed $\gamma,\lambda>0$ and every fixed possibly
randomized output rule depending only on the observed transcript
and independent auxiliary randomness, let $x_K^{\mathrm{out}}$ denote
the point returned by this rule. There exist a convex
$(1,L_1)$-smooth function $f:\mathbb R\to\mathbb R$, with a minimizer
$x^\star\in[0,R_0]$, and an unbiased stochastic-gradient oracle
satisfying \cref{asm: stochastic}, such that \algname{Clip-SGD},
started from $x_0=0$ and using fresh independent oracle samples,
satisfies
$$
\mathbb P\left(
f(x_K^{\mathrm{out}})-f^\star
\ge\frac{\sigma R_0}{128}
\left(\frac{L_1R_0}{K\log(eK)}\right)^{(\alpha-1)/\alpha}
\right)\ge\frac14.
$$
\end{lemma}

\subsection{Proofs}
\begin{proof}[Proof of \cref{lem:fixed-clip-hard-family}]
Define $g$ on $[0,\infty)$ by
$g(t):=\min\left\{a,\frac{e^{L_1t}-1}{L_1}\right\}$.
Since $g$ is continuous, nonnegative, and nondecreasing, $\phi_a(x)=\int_0^{|x|}g(t)\,dt$ is continuously differentiable and convex. Since $g\ge0$, $\phi_a(x)\ge0=\phi_a(0)$ for all $x$, so $0$ is a global minimizer. 

We have, $\phi_a'(x)=\operatorname{sign}(x)\,g(|x|)$ and 
by the definition of $r_a$, we also have $\frac{e^{L_1r_a}-1}{L_1}=a.$
Therefore
$$
\phi_a'(x)=
\begin{cases}
\operatorname{sign}(x)\dfrac{e^{L_1|x|}-1}{L_1}, & |x|<r_a,\\[1ex]
a\,\operatorname{sign}(x),& |x|\ge r_a.
\end{cases}
$$
In particular, $|\phi_a'(x)|\le a$.

To verify generalized smoothness, fix $x<y$ and split $[x,y]$, if
necessary, at $-r_a$, $0$, and $r_a$. On each resulting subinterval
$[s,t]$, either $\phi_a'$ is constant, or
$$
\frac{d}{du}\phi_a'(u)=e^{L_1|u|}=1+L_1|\phi_a'(u)|.
$$
Hence, by the mean-value theorem,
$$
|\phi_a'(t)-\phi_a'(s)| \le \left(1+L_1\sup_{u\in[s,t]}|\phi_a'(u)|\right)(t-s).
$$
Summing these bounds over the subintervals and using
$\sup_{u\in[s,t]}|\phi_a'(u)| \le \sup_{u\in[x,y]}|\phi_a'(u)|$
gives
$$
|\phi_a'(y)-\phi_a'(x)| \le \left(1+L_1\sup_{u\in[x,y]}|\phi_a'(u)|\right)(y-x).
$$
Finally, using $\log(1+z)\le z$ for $z\ge0$, we obtain $r_a=\frac{\log(1+L_1a)}{L_1}\le a.$

It remains to prove the lower bound on $\phi_a$. If $|x|\le r_a$, then $a(|x|-r_a)_+=0\le\phi_a(x)$. If $|x|>r_a$, then $g(t)=a$ for every $t\in[r_a,|x|]$, and therefore $$ \phi_a(x) =\int_0^{r_a}g(t)\,dt+\int_{r_a}^{|x|}a\,dt \ge a(|x|-r_a). $$ Hence, for every $x\in\mathbb R$, $$ \phi_a(x)\ge a(|x|-r_a)_+. $$ 
\end{proof}

\begin{proof}[Proof of \cref{lem:fixed-clip-indistinguishability}]
Let $E:=\{\mathcal T_1=\mathcal T_2\}.$ By assumption, $
\mathbb P(E)\ge q.$

Let $\mu$ denote the distribution of the auxiliary randomness used by
the randomized output rule. Since the same output rule is used on both
instances, there exists a measurable function $\Psi$ such that, given a
transcript $t$, the output can be represented as
$$
\Psi(t,U), \qquad U\sim\mu.
$$
Enlarge the probability space supporting the coupling
$(\mathcal T_1,\mathcal T_2)$ by an independent random variable $U\sim\mu$, and define
$$
x_K^{\mathrm{out},(i)}:=\Psi(\mathcal T_i,U), \qquad i\in\{1,2\}.
$$

On the event $E$, we have $\mathcal T_1=\mathcal T_2$, and therefore
$$
x_K^{\mathrm{out},(1)}=\Psi(\mathcal T_1,U)= \Psi(\mathcal T_2,U)=x_K^{\mathrm{out},(2)}.
$$
Hence, on $E$,
$$
\left|x_K^{\mathrm{out},(1)}-u_1\right| +\left|x_K^{\mathrm{out},(1)}-u_2\right| \ge |u_1-u_2|= \Delta_u.
$$
Consequently, on $E$,
$$
\max\left\{\left|x_K^{\mathrm{out},(1)}-u_1\right|,\left|x_K^{\mathrm{out},(1)}-u_2\right|
\right\}\ge\frac{\Delta_u}{2}.
$$
Thus, on $E$, for at least one $i\in\{1,2\}$,
$\left|x_K^{\mathrm{out},(i)}-u_i\right| \ge \frac{\Delta_u}{2}.$
For such an $i$, Lemma~\ref{lem:fixed-clip-hard-family} gives
\begin{align*}
f_i(x_K^{\mathrm{out},(i)})-f_i^\star&=\phi_a(x_K^{\mathrm{out},(i)}-u_i) \\
&\ge a\left(\left|x_K^{\mathrm{out},(i)}-u_i\right|-r_a\right)_+ \\
&\ge a\left(\frac{\Delta_u}{2}-r_a\right),
\end{align*}
where the last quantity is positive because $\Delta_u>2r_a$.

For $i\in\{1,2\}$, define
$$
A_i:=\left\{f_i(x_K^{\mathrm{out},(i)})-f_i^\star\ge a\left(\frac{\Delta_u}{2}-r_a\right)\right\}.
$$
The preceding argument shows that $E\subseteq A_1\cup A_2.$ Hence, $\mathbf 1_{A_1}+\mathbf 1_{A_2} \ge \mathbf 1_E.$ Therefore
$$
\mathbb P(A_1)+\mathbb P(A_2) \ge \mathbb P(E) \ge q.
$$
Since each coupled run has its original marginal distribution,
$$
\mathbb P(A_i)= \mathbb P_i\left( f_i(x_K^{\mathrm{out},(i)})-f_i^\star 
\ge a\left(\frac{\Delta_u}{2}-r_a\right)\right).
$$
We deduce that at least one $i\in\{1,2\}$,
$$
\mathbb P_i\left(f_i(x_K^{\mathrm{out},(i)})-f_i^\star \ge a\left(\frac{\Delta_u}{2}-r_a\right)
\right) \ge \frac q2.
$$
\end{proof}

\begin{proof}[Proof of \cref{lem:bernoulli-oracle}]
We have $\mathbb E[G(x)]=\frac{f'(x)}{p}\mathbb E[Z]=f'(x).$
Moreover,
\begin{align*}
\mathbb E\left[|G(x)-f'(x)|^\alpha\right]&=|f'(x)|^\alpha
\left[p\left|\frac{1}{p}-1\right|^\alpha+(1-p)\right] \\
&=|f'(x)|^\alpha\left[p^{1-\alpha}(1-p)^\alpha+(1-p)\right].
\end{align*}
Since $|f'(x)|\le a$ and $p\in(0,1]$, we obtain
\begin{align*}
\mathbb E\left[|G(x)-f'(x)|^\alpha\right]&\le a^\alpha\left(p^{1-\alpha}+1\right) \\
&\le 2a^\alpha p^{1-\alpha},
\end{align*}
where the last inequality follows from $p^{1-\alpha}\ge 1$ since $\alpha>1$.
\end{proof}

\begin{proof}[Proof of \cref{lem:fixed-clip-confidence-lower-bound}]
We consider the two hard instances $f_1(x)=\phi_a(x+R_0)$ and $f_2(x)=\phi_a(x-R_0)$, whose minimizers are located at $-R_0$ and $R_0$, respectively. We couple their stochastic oracles so that, with probability greater than
$2\delta$, both runs observe only zero gradients.

Consider the oracles $G_i(x,Z)=f_i'(x)Z/p$ with
$Z\sim\operatorname{Bernoulli}(p)$, for $i\in\{1,2\}$. If the same Bernoulli draws are used
for the two instances, then the two transcripts are identical on the
event that all $K$ draws are zero. This event has probability $(1-p)^K$.
For $p\in(0,1/2]$, $\log(1-p)>-2p$, and hence $(1-p)^K>e^{-2Kp}$.
To make this lower bound equal to $2\delta$, we choose $p$ so that
$e^{-2Kp}=2\delta$, which gives
$$
p=\frac{\log(1/(2\delta))}{2K}.
$$
Since $K\ge\log(1/(2\delta))$, we indeed have $p\le1/2$.

We now determine the largest tail slope allowed by the stochastic-oracle
assumption. Since $|f_i'(x)|\le a$ for $i\in\{1,2\}$, \cref{lem:bernoulli-oracle} gives
$$
\mathbb E\left[|G_i(x,Z)-f_i'(x)|^\alpha\right] \le 2a^\alpha p^{1-\alpha}.
$$
Thus, to ensure the stochastic-oracle assumption, it suffices to impose $2a^\alpha p^{1-\alpha}\le\sigma^\alpha$.

On the other hand, if $r_a\le R_0/2$, then the hard-function
construction in \cref{lem:fixed-clip-indistinguishability} gives an
objective gap of at least $a(R_0-r_a)\ge aR_0/2$. Hence, in this regime,
the guaranteed lower bound increases with $a$. We therefore take the
boundary value permitted by the moment sufficient condition:
$$
a=\frac{\sigma}{2} \left( \frac{\log(1/(2\delta))}{K}\right)^{(\alpha-1)/\alpha}.
$$

It remains to verify that $r_a\le R_0/2$. We have
$$
r_a\le \frac{R_0}{2} \iff a\le \frac{e^{L_1R_0/2}-1}{L_1} \iff K\ge \log\left(\frac{1}{2\delta}\right) \left( \frac{\sigma L_1}{2(e^{L_1R_0/2}-1)}\right)^{\alpha/(\alpha-1)}.
$$
Since $e^x-1\ge x$ for $x\ge0$,
$$
\frac{\sigma L_1}{2(e^{L_1R_0/2}-1)} \le \frac{\sigma}{R_0},
$$
so the required condition $r_a\le R_0/2$ follows from the assumed lower
bound on $K$.

Therefore, since $\mathbb P(\mathcal T_1=\mathcal T_2)\ge \mathbb P(E)>2\delta,$
\cref{lem:fixed-clip-indistinguishability} implies that there exists
$i\in\{1,2\}$ such that
$$
\mathbb P_i\left(f_i(x_K^{\mathrm{out},(i)})-f_i^\star
\ge a(R_0-r_a)\right)>\delta.
$$
Since $r_a\le R_0/2$, the definition of $a$ gives
$$
a(R_0-r_a)\ge \frac{aR_0}{2} = \frac{\sigma R_0}{4} \left(\frac{\log(1/(2\delta))}{K}\right)^{(\alpha-1)/\alpha}.
$$
\end{proof}

\begin{proof}[Proof of \cref{lem:fixed-clip-displacement-lower-bound}]
Fix $\gamma,\lambda$ and the output rule. Since $K\ge L_1R_0$ and
$\log(eK)\ge1$, it follows that
$
\frac{L_1R_0}{K\log(eK)}\le1.
$
Thus an objective gap of $\sigma R_0$ is stronger than the claimed
lower bound.

\paragraph{First, consider large updates.}
Let $H,b>0$ satisfy $r_H<b/2$ and $2b\le R_0$, and consider
$$
f_1(x):=\phi_H(x-b),
\qquad
f_2(x):=\phi_H(x-2b).
$$
Their minimizers are located at $b$ and $2b$, respectively. For this
construction, we use the exact-gradient oracle. Since the minimizers are separated by $b>2r_H$,
\cref{lem:fixed-clip-indistinguishability} gives the objective threshold
$H\left(\frac{b}{2}-r_H\right).$

Let
$$
S_H:=\gamma\min\{\lambda,H\}.
$$
Because $r_H<b/2$, the point $x_0=0$ lies in the common left linear
tail of the two objectives, where both gradients are equal to $-H$.
If
$$
S_H\ge 2b+r_H,
$$
then the first update moves both runs from $0$ into their common right
linear tail, where both gradients are equal to $H$. The next update
therefore returns both runs to $0$. The same two-step trajectory then
repeats, and the two transcripts are identical. Hence, for one of the
two instances,
$$
\mathbb P_i\left(f_i(x_K^{\mathrm{out},(i)})-f_i^\star \ge H\left(\frac{b}{2}-r_H\right)
\right)\ge\frac12.
$$
It remains to consider the complementary case $S_H<2b+r_H.$

\paragraph{Next, consider small updates.}
Consider
$$
h_1(x):=\phi_a\left(x-\frac{R_0}{2}\right),
\qquad
h_2(x):=\phi_a(x-R_0),
$$
where $a>0$ will be chosen later. Both runs start from $x_0=0$.
For each instance, use the Bernoulli oracle
$$
G_i(x,Z):=\frac{h_i'(x)}{p}Z,
\qquad Z\sim\operatorname{Bernoulli}(p),
$$
where $p\in(0,1)$ will be chosen later. Successive oracle calls use independent Bernoulli draws.

We require $a/p\le H$. Since $|h_i'(x)|\le a$, every update
corresponding to a draw $Z$ has length at most
$$
\gamma\min\left\{\lambda,\frac ap\right\}Z \le S_H Z,
$$
where $S_H<2b+r_H$ in the present case.

For the comparison of the two instances, couple the runs using the same Bernoulli draws 
$Z_0,\ldots,Z_{K-1}$, and let
$$
N:=\sum_{k=0}^{K-1}Z_k
$$
be the number of nonzero Bernoulli draws.
Since $\mathbb E[N]=Kp$, Markov's inequality gives $
\mathbb P(N\le2Kp)\ge\frac12.$ To control the total distance traveled, we impose
$$
2Kp(2b+r_H)\le\frac{R_0}{4}.
$$
On the event $\{N\le2Kp\}$, for every $i\in\{1,2\}$ and
$0\le k\le K$,
$$
|x_k^{(i)}| \le S_H\sum_{j=0}^{k-1}Z_j \le S_H N \le 2KpS_H.
$$
Since $S_H<2b+r_H$ and
$2Kp(2b+r_H)\le R_0/4$, it follows that $|x_k^{(i)}|<\frac{R_0}{4}.$

Since $p<1$ and $a/p\le H$, we have $a\le pH<H$.
Together with $r_H<b/2$ and $2b\le R_0$, this gives
$r_a\le r_H<b/2\le R_0/4$.
Then every $x\le R_0/4$ lies in the common left linear tail,
where $h_1'(x)=h_2'(x)=-a$. Thus, on
$\{N\le2Kp\}$, the shared Bernoulli draws give
$$
G_k^{(1)}=G_k^{(2)}=-\frac ap Z_k.
$$
Since both runs start from zero and receive the same stochastic gradient at every iteration, their iterates also coincide. Consequently, $
\mathbb P(\mathcal T_1=\mathcal T_2)\ge\frac12.$

The minimizers are separated by $R_0/2>2r_a$.
By \cref{lem:fixed-clip-indistinguishability}, there exists
$i\in\{1,2\}$ such that
$$
\mathbb P_i\left(
h_i(x_K^{\mathrm{out},(i)})-h_i^\star
\ge a\left(\frac{R_0}{4}-r_a\right)
\right)\ge\frac14.
$$

It remains to ensure that the Bernoulli oracle satisfies the stochastic
oracle assumption. By \cref{lem:bernoulli-oracle}, it is enough to have
$2a^\alpha p^{1-\alpha}\le\sigma^\alpha.$

We have therefore reduced the proof to choosing $a,p,H$, and $b$, with
$p\in(0,1)$, so that
$$
\frac{a}{p}\le H,
\qquad
2Kp(2b+r_H)\le\frac{R_0}{4},
\qquad
2a^\alpha p^{1-\alpha}\le\sigma^\alpha.
$$
For the large-update construction, we also need $r_H<\frac{b}{2}$ and
$2b\le R_0$. The objective thresholds in the two constructions are
$$
H\left(\frac{b}{2}-r_H\right) \qquad\text{and}\qquad
a\left(\frac{R_0}{4}-r_a\right).
$$
We now choose the parameters so that these conditions hold and both
thresholds are at least the claimed lower bound.

Assuming $r_H<b/2$, we have $2b+r_H<\frac{5b}{2}.$
Hence $2Kp(2b+r_H)<5Kpb$, so the displacement condition is guaranteed by
$$
5Kpb\le\frac{R_0}{4}.
$$
For fixed $b$, increasing $p$ increases the largest admissible slope
$a$ under the noise constraint, while decreasing the corresponding
value of $a/p$. We therefore take the largest value allowed by the
sufficient displacement condition above,
$$p:=\frac{R_0}{20Kb}.$$
We will verify that $p<1$ after choosing $b$. For this value of $p$, the
stochastic oracle condition gives $a\le 2^{-1/\alpha}\sigma p^{(\alpha-1)/\alpha}.$
We take
$$
a:=2^{-1/\alpha}\sigma p^{(\alpha-1)/\alpha}.
$$

We next choose $H$. To obtain an objective threshold of at least
$\sigma R_0$ in the large-update construction, we need $
H\left(\frac b2-r_H\right)\ge\sigma R_0,$ or equivalently $$
\frac b2\ge r_H+\frac{\sigma R_0}{H}.$$
Increasing $H$ makes
the first term larger and the second term smaller. We choose the
slope that makes their sum as small as possible, subject to
$H\ge a/p$.

For a fixed $b$, the values of $a$ and $p$ are already determined.
We therefore set
$$
H:=\operatorname*{arg\,min}_{u\ge a/p}
\left\{r_u+\frac{\sigma R_0}{u}\right\}.
$$
We have
$$
\frac{d}{du}\left(r_u+\frac{\sigma R_0}{u}\right)
=
\frac{u^2-\sigma L_1R_0u-\sigma R_0}
     {u^2(1+L_1u)}.
$$
The numerator has exactly one positive root,
$
u_\star:=
\frac{\sigma L_1R_0+
\sqrt{(\sigma L_1R_0)^2+4\sigma R_0}}{2}.$
The sum decreases up to $u_\star$ and increases after it.
Taking into account the constraint $u\ge a/p$, we obtain
$$
H=\max\left\{\frac ap,u_\star\right\}.
$$

We first derive a bound on $H$ under the condition $2b\le R_0$,
and then verify that our choice of $b$ indeed satisfies this condition.
Under $2b\le R_0$, the definitions of $a$ and $p$ give
$$
L_1\frac ap = \sigma L_1\left(\frac{10Kb}{R_0}\right)^{1/\alpha}
\le\sigma L_1(5K)^{1/\alpha} \le 5K^2.
$$
We also have,
$
L_1u_\star \le (\sigma L_1)(L_1R_0)+ \sqrt{(\sigma L_1)(L_1R_0)} \le K^2+K \le 2K^2.$
Hence $$L_1H\le5K^2.$$
Since $H\ge u_\star>\sigma L_1R_0$, it holds that $
\frac{\sigma R_0}{H}<\frac{1}{L_1},$
and therefore $
r_H+\frac{\sigma R_0}{H}< \frac{\log(1+5K^2)+1}{L_1}.$
We choose $$b:=\frac{2\bigl(\log(1+5K^2)+1\bigr)}{L_1}.$$
Since $K\ge1$, we have $\frac{2}{L_1}\le b\le\frac{6\log(eK)}{L_1}<\frac{3R_0}{32}.$
In particular,
$$
2b<\frac{3R_0}{16}<R_0,
$$
so the condition $2b\le R_0$ used above is indeed satisfied.

These bounds also give
$$
0<\frac{L_1R_0}{120K\log(eK)} \le p \le
\frac{L_1R_0}{40K} \le\frac1{40}<1.
$$

Finally, by the choice of $b$, $
r_H+\frac{\sigma R_0}{H}<\frac{b}{2}.$
Hence $r_H<b/2$ and
$$
H\left(\frac{b}{2}-r_H\right)>\sigma R_0,
$$
which gives the required threshold in the large-update case.

For the small-update case, $a/p\le H$ holds by construction, and
$$
2Kp(2b+r_H)<5Kpb=\frac{R_0}{4},
\qquad
2a^\alpha p^{1-\alpha}=\sigma^\alpha.
$$
Since $a\le pH<H$, we also have
$r_a\le r_H<b/2<3R_0/64<R_0/8$. Therefore,
$$
\begin{aligned}
a\left(\frac{R_0}{4}-r_a\right)&\ge\frac{aR_0}{8}=\frac{\sigma R_0}{16}(2p)^{(\alpha-1)/\alpha}\\
&\ge\frac{\sigma R_0}{16} \left(\frac{L_1R_0}{60K\log(eK)}\right)^{(\alpha-1)/\alpha}\\
&\ge\frac{\sigma R_0}{128}\left(\frac{L_1R_0}{K\log(eK)}\right)^{(\alpha-1)/\alpha},
\end{aligned}
$$
where the last inequality uses $60^{(\alpha-1)/\alpha}\le\sqrt{60}<8$.

The large-update case gives the claimed lower bound with probability
at least $1/2$, and the small-update case with probability at least $1/4$.
\end{proof}

\begin{proof}[Proof of \cref{thm:fixed-clip-lower-bound}]
Fix $\gamma,\lambda$ and the output rule.
First, suppose that
$$
\frac{L_1R_0}{\log(eK)}>64 \qquad\text{and}\qquad \frac{L_1R_0}{\log(eK)}\ge\log(12/\delta).
$$
Then
$$
\max\left\{\frac{L_1R_0}{\log(eK)},\,\log(12/\delta)\right\}=\frac{L_1R_0}{\log(eK)}.
$$
By \cref{lem:fixed-clip-displacement-lower-bound}, there exists
an instance such that the objective gap is at least
$$
\frac{\sigma R_0}{128}
\left(\frac{L_1R_0}{K\log(eK)}\right)^{(\alpha-1)/\alpha}
$$
with probability at least $1/4>\delta$.
This proves the claimed bound in this case.

It remains to consider the case where
$L_1R_0/\log(eK)\le64$ or $L_1R_0/\log(eK)<\log(12/\delta)$.
Since $\log(12/\delta)>1$, we have
$$
\max\left\{\frac{L_1R_0}{\log(eK)},\,\log(12/\delta)\right\} \le64\log(12/\delta).
$$

Set
$$
p:=\frac{\log(1/(2\delta))}{2K}, \qquad a:=2^{-1/\alpha}\sigma p^{(\alpha-1)/\alpha}
=\frac{\sigma}{2} \left(\frac{\log(1/(2\delta))}{K}\right)^{(\alpha-1)/\alpha},
$$
and consider
$$
f_1(x):=\phi_a(x+R_0),
\qquad
f_2(x):=\phi_a(x-R_0).
$$
Their minimizers are located at $-R_0$ and $R_0$, respectively.
Both objectives are convex and $(1,L_1)$-smooth by
\cref{lem:fixed-clip-hard-family}.
The iteration assumptions give $0<p\le1/2$ and
$r_a\le a\le R_0/2$.

For each instance, use the Bernoulli oracle
$G_i(x,Z):=f_i'(x)Z/p$, with fresh independent draws.
By \cref{lem:bernoulli-oracle}, both oracles are unbiased and
have central $\alpha$-th moments at most
$$
2a^\alpha p^{1-\alpha}=\sigma^\alpha.
$$

Couple the runs using the same Bernoulli draws.
On the event that all $K$ draws are zero, both runs remain at
zero and their transcripts coincide.
Since $\log(1-p)>-2p$ for $0<p\le1/2$, this event has probability
$$
(1-p)^K>e^{-2Kp}=2\delta.
$$
The minimizers are separated by $2R_0>2r_a$.
By \cref{lem:fixed-clip-indistinguishability}, there exists
$i\in\{1,2\}$ such that
$$
\mathbb P_i\left(f_i(x_K^{\mathrm{out},(i)})-f_i^\star \ge a(R_0-r_a)\right)>\delta.
$$
Since $r_a\le R_0/2$, we have $
a(R_0-r_a) \ge\frac{aR_0}{2} =\frac{\sigma R_0}{4} \left(\frac{\log(1/(2\delta))}{K}\right)^{(\alpha-1)/\alpha}.$

Finally, since $\delta<1/4$, it holds that
$$
\log(12/\delta)=\log(1/(2\delta))+\log24 <6\log(1/(2\delta)).
$$
Using $4\cdot6^{(\alpha-1)/\alpha}<16$ and
$64^{(\alpha-1)/\alpha}\le8$, we obtain
$$
\begin{aligned}
a(R_0-r_a) &\ge \frac{\sigma R_0}{16} \left(\frac{\log(12/\delta)}{K}\right)^{(\alpha-1)/\alpha}\\
&\ge \frac{\sigma R_0}{128} \left[\frac{\max\left\{L_1R_0/\log(eK),\,\log(12/\delta)
\right\}}{K}\right]^{(\alpha-1)/\alpha}.
\end{aligned}
$$
\end{proof}

\end{document}